\documentclass[11pt]{amsart}

\usepackage{amsfonts,amssymb,verbatim,amsmath,amsthm,latexsym,textcomp,amscd}
\usepackage{graphicx}
\usepackage[svgnames]{xcolor}
\usepackage{url}
\usepackage{array}
\usepackage{enumitem}
\usepackage{hyperref}

\def\g{{\gamma}}
\def\u{{\omega}}

\def\A{{\mathcal A}}
\def\B{{\mathcal B}}
\def\C{{\mathcal C}}
\def\D{{\mathcal D}}

\def\F{{\mathcal F}}
\def\G{{\mathcal G}}
\def\H{{\mathcal H}}
\def\I{{\mathcal I}}
\def\L{{\mathcal L}}

\def\P{{\mathcal P}}

\def\R{{\mathcal R}}
\def\S{{\mathcal S}}
\def\T{{\mathcal T}}
\def\U{{\mathcal U}}
\def\V{{\mathcal V}}

\def\TR{{\mathbf{T}_{\mathcal{R}}}}

\def\DD{{\mathbb D}}

\def\ZZ{{\mathbb Z}}

\def\dddots{\mathord{\rlap{$\cdots$}\dotsc}}

\newcommand{\eps}{\varepsilon}

\newcommand{\EE}{{\mathsf E}} %Expected value

\newcommand{\Paths}{{\mathbf P}}

\DeclareMathOperator*{\Clos}{clos}
\DeclareMathOperator*{\Int}{int}
\DeclareMathOperator*{\dist}{dist}

\newcommand{\norm}[1]{\lVert#1\rVert}

\def\dd{ {\partial}}

\newtheorem{thm}{Theorem}[section]
\newtheorem{introthm}{Theorem}[section]

\newtheorem{lemma}[thm]{Lemma}
\newtheorem{cor}[thm]{Corollary}
\newtheorem{prop}[thm]{Proposition}
\newtheorem{claim}[thm]{Claim}
\newtheorem{asm}[thm]{Assumption}

\theoremstyle{definition}
\newtheorem{definition}[thm]{Definition}

\newtheorem{remark}[thm]{Remark}

\title[Convergence of spherical averages for actions of Fuchsian Groups]{Convergence of spherical averages\\for actions of Fuchsian Groups}

\author{Alexander I. Bufetov}
\address{\begin{flushleft} Alexander I. Bufetov\\
\upshape Aix--Marseille Universit\'e, CNRS, Centrale Marseille, I2M, UMR 7373,\\
39 rue F. Joliot Curie Marseille France;\\
Steklov Mathematical Institute of Russian Academy of Sciences,\\Gubkina str. 8,  119991, Moscow, Russia\\
\texttt{bufetov@mi-ras.ru}\\
  \end{flushleft}}

\author{Alexey Klimenko}
\address{\begin{flushleft} Alexey Klimenko\\
\upshape Steklov Mathematical Institute of Russian Academy of Sciences,\\Gubkina str. 8,  119991, Moscow, Russia;\\
and National Research University Higher School of Economics,\\ Usacheva str. 6, 119048,  Moscow, Russia \\
\texttt{klimenko@mi-ras.ru} \end{flushleft}}

\author{Caroline Series}
\address{\begin{flushleft}Caroline Series\\
\upshape Mathematics Institute, University of Warwick \\
Coventry CV4 7AL, UK\\
\texttt{C.M.Series@warwick.ac.uk \\http://www.maths.warwick.ac.uk/\~{}masbb/} \end{flushleft}}

\date{\today}

\begin{document}

 \begin{abstract}
 We prove pointwise convergence of spherical averages for a measure-preser\-ving action of a Fuchsian group. The proof is based on a new variant of the Bowen--Series symbolic coding for Fuchsian groups that, developing a method introduced by Wroten, simultaneously encodes all possible shortest paths representing a given group element. The resulting coding is self-inverse, giving a reversible Markov chain to which methods previously introduced by the first author for the case of free groups may be applied.   
\end{abstract}

\maketitle

\tableofcontents

\section{Introduction}

\subsection{Formulation of the main result}Let  $G$ be a finitely generated group with a symmetric set of generators $G_0$.
For $g\in G$, denote by $|g|$ the length of the shortest
word in $G_0$ representing $g$.
Let $S(n)$ be be the sphere of radius $n$ in $G$:
\begin{equation*}
 S(n)=\{g\in G: |g|=n\}.
\end{equation*}
Suppose that $G$ acts on a probability space $(X, \mu)$ by measure-preserving
transformations $T_g$, $g\in G$.
For a function $f\in L^1(X,\mu)$ consider spherical averages
\begin{equation}\label{eq:sph-avg}
\mathbf{S}_n(f) = \frac1{\#S(n)} \sum\limits_{g\in S(n)} f\circ T_g.
\end{equation}

The main result of this paper, Theorem \ref{thm:main} below, gives the almost sure convergence of spherical averages for measure-preserving actions of Fuchsian groups and for $f \in L \log L(X,\mu)$, that is, whenever $\int |f | \log^+ | f |d \mu < \infty$.

Let $G$ be a Fuchsian group and let $\R$ be a fundamental domain for $G$. Assume that the sides of $\R$ are paired by a set of elements $G_0 \subset G$. As is well known, $G_0$ is a symmetric set of generators for $G$.   The images of $\R$ under the action of $G$ induce a tessellation $\TR=\{g\R:g\in G\}$ of the hyperbolic disk $\DD$.  Following~\cite{BowSer}, we say that $\R$ has \emph{even corners}
if the geodesic extension of every side of $\R$ is entirely contained in $\TR$, more precisely in the  union of boundaries of all domains $g\R\in\TR$.

 Let $v\in\DD$ be a vertex of $\TR$. If $\R$ has even corners, then the boundary of $\TR$ in a small neighbourhood of $v$ consists of $n$ geodesic segments intersecting at $v$ and dividing our neighbourhood into $2n$ sectors. Write $n=n(v)$ and let $N(\R)$ denote the number of sides of $\R$ inside~$\DD$.  We need the following assumption on $\R$.

\begin{asm}\label{asm:R}\quad
\begin{enumerate}[itemsep=0pt,label={(\roman*)}]
	\item $\R$ has even corners,
	\item\label{cond:asm-R-ii}One of the following conditions holds for $\R$:
	\begin{itemize}
		\item $N(\R)\ge5$,
		\item $N(\R)=4$ and either $\R$ is non-compact or $\R$ is compact and does not have two opposite vertices $v,v'$ such that $n(v)=n(v')=2$,
		\item $N(\R)=3$ and $\R$ is non-compact.
	\end{itemize}
\end{enumerate}
\end{asm}

Our main result is the following:

\begin{introthm}\label{thm:main}
Let $G$ be a non-elementary Fuchsian group $G$  and let $\R$ be its fundamental domain with side-pairing transformations $G_0$ and  satisfying Assumption~\ref{asm:R}.
Let $G$ act on a Lebesgue  probability space $(X,\mu)$ by measure-preserving transformations. 
Denote by $\I_{G_0^2}$ the sigma-algebra of sets invariant under all maps $T_{g_1g_2}$, $g_1,g_2\in G_0$.
Then, for any function $f\in L\log L(X,\mu)$, as $n\to\infty$, we have
\begin{equation*}
\mathbf{S}_{2n}(f)\to\EE(f|\I_{G_0^2})\quad\text{almost surely and in $L^1$}.
\end{equation*}
\end{introthm}

The condition that $\R$ has even corners is not as restrictive as it  appears. In fact it is clear that our result only depends on the generators $G_0$ and the coding, and not on the precise geometry of $\R$. Thus Theorem~\ref{thm:main}  extends immediately to any presentation of a  Fuchsian group for which one can find deformed group  $G'$ which has a fundamental domain $\R'$ with the same  pattern of sides and side-pairings and even corners, see~\cite{BufSer, BowSer} and \cite{Ser-Trieste} for a detailed discussion.   The need to restrict to spheres of even radius can be seen by considering the action of the free group $F_2$ on the two-element set ${\{0,1\}}$ in which both generators of ${F_2}$ act by interchanging the elements, in which case the value of $\mathbf{S}_{n}(f)$ depends on the parity of $n$.

We note that the conditions of Assumption 1.1 are not quite identical with those in~\cite{Ser-Inf}, \cite{BirSer} and elsewhere, the main difference being  the weaker restriction if $N(\R) = 4$. In fact all results of those papers should apply under these somewhat weaker assumptions.

The Ces\`aro convergence of the averages $\mathbf{S}_{2n}(f)$ is proven in~\cite{BufSer} using the Bowen--Series Markovian coding \cite{BowSer}, see also \cite{BirSer, Ser-Inf, Ser-Trieste}, in order to reduce the statement to the ergodic theorem for Markov operators, cf. \cite{Buf-FAA,Buf-AMSTr}. The Bowen--Series coding allows one to assign states to group generators in a suitable  product representing an arbitrary group element as a shortest word in such a way that the admissible sequences of states form a Markov chain.  This gives rise to a Markov operator as described in~\cite{BufSer}.  However the proof in~\cite{Buf-Annals},  which establishes   convergence of the spherical averages themselves for free groups, does not extend in any obvious way. This is because  the argument of~\cite{Buf-Annals} relies on a symmetry condition for the coding, namely that the coding is  reversible or self-inverse, which allows one to relate the Markov operator generated by the coding to its adjoint.  The Bowen--Series coding of~\cite{BowSer} fails to be symmetric in this sense.

The main construction of this paper is a new self-inverse coding for Fuchsian groups, which allows us to adapt the proof in~\cite{Buf-Annals} to this new case.

This new coding is constructed using a variant of the coding introduced by Matthew Wroten~\cite{Wroten}, see also a related idea  in~\cite{FlPlot} and~\cite{Ser-Zeeman}.
Wroten's idea is  to encode all possible representations of a group element as a shortest word simultaneously. This involves assigning states to all possible ways of building up shortest words step by step. The set of states together with allowed transitions defines a Markov chain with the property that  the transition rules, that is the set of all admissible paths, can be inverted.  From this we construct an associated Markov operator  with the required symmetry condition on its adjoint, and then derive a suitably modified version of  the convergence theorem in~\cite{Buf-Annals}.

It would be interesting to obtain a similar coding for a more general hyperbolic groups. In particular, it is not clear to us how to invert paths in the classical Cannon--Gromov coding \cite{Cannon,Gromov}.

To explain the ideas in a bit more detail, let us briefly describe Wroten's approach in our setting. Every shortest word in the Fuchsian group $G$ corresponds to a shortest path in the Cayley graph of $G$ relative to the given generators $G_0$. This graph is embedded in~$\DD$ by sending $g \in G$ to $gO \in \DD$, where $O$ is some fixed base point in $\Int \R$. Vertices $gO, hO$ are joined by an edge if and only if $g^{-1}h \in G_0$.
If $\beta$ is a shortest path in the Cayley graph, we refer to the sequence of domains traversed by the edges of $\beta$ also as a shortest path.
If $g \in G$ then the \emph{thickened path} $[g]$  associated to $g$ is by definition the collection of all those $h\R, h \in G$  which are traversed by some shortest path  from $\R$ to $g\R$. Every domain $h\R\in [g]$ is endowed with an \emph{index}, which equals the distance in the Cayley graph from $\R$ to $h\R$. The set of all domains with index $k$ we will refer to as a \emph{level} of $[g]$ and denote by $[g]_k$.

Now the coding works as follows. 
We will define a space of states $\Xi = \{X_1, \ldots, X_k\}$ and a $\Xi \times \Xi$  transition matrix $\Pi=(\Pi_{ij})$ such that $\Pi_{ij} = 1$ if transition from $X_i$ to $X_j$ is possible and $\Pi_{ij} = 0$ otherwise. There is a subset $\Xi_S\subset \Xi$  of start states, and another subset $\Xi_F \subset \Xi$ of end states.

The states in $\Xi$ represent how $[g]_k$ and $[g]_{k+1}$ are attached to each other. It turns out that every $[g]_k$ contains at most two fundamental domains and  the domains from $[g]_{k+1}$ are glued to the ones from $[g]_k$ across  one, two or three sides, see Figure~\ref{fig:configs}. We endow this geometrical configuration with some additional data to obtain a Markov chain generating thickened paths; in particular, the data records the generators needed to carry out the gluing. Then we define a transition matrix $\Pi$ and subsets $\Xi_S$ and $\Xi_F$ of $\Xi$ and prove that thickened paths from $\R$ to $g\R$ with $|g|=n$ are in one-to-one correspondence with admissible sequences of length $n$  starting in $\Xi_S$ and ending in $\Xi_F$.  The required  reversiblity or self-inverse symmetry condition  follows since inverting a thickened path yields a thickened path and the coding preserves this symmetry.

In terms of the associated Markov operators, this symmetry property can be expressed as follows. We introduce two maps $\g,\u\colon \Xi\to G$, closely related to the attaching maps between $[g]_k$ and $[g]_{k+1}$, see Section~\ref{sec:sph-avg}.
These maps satisfy certain relations, see  Lemma~\ref{lem:g-and-u}. Following~\cite{Buf-Annals} we then construct Markov operators $P$ and $U$ on $L^1(X\times\Xi)$, which as a consequence of these relations satisfy   $P^*=UPU$ and $U^*=U^{-1}=U$. Then we can apply the Alternierende Verfahren method in a manner similar to~\cite{Buf-Annals}. 

For this application we need an inequality between $P^n$ and $(P^*)^kP^k$, which is the basis for the maximal inequality in the Alternierende Verfahren scheme. For free groups this inequality was $cUP^{2n-1}\varphi\le (P^*)^nP^n\varphi$ for any nonnegative $\varphi$, see~\cite{Buf-Annals}.
In the present case the inequality becomes more complicated, both  because the index  on the right hand side may vary slightly and also because there are a small number of possible sequences for which the required geometrical statements fail. To correct this,  terms on the right hand side of the inequality have to be summed over a small bounded interval of  indices near $n$,  and the inequality also contains an error term $A_n\varphi$, see~\eqref{eq:asm-ineq} in Section~\ref{sec:proof-main} below.  The proof of the geometrical statement associated to the proof of this inequality, Lemma~\ref{lem:wye}, is one of the most technically complicated parts of the paper.

A short announcement of the results of this paper with a more detailed outline of the coding can be found in~\cite{BKS-MRR}.

\subsection{Organization of the paper}
The paper is organized as follows. In the next section we give some notation and preliminaries regarding Fuchsian groups and their fundamental domains. In particular, we show in  Lemma~\ref{lem:no-trig} that under Assumption~\ref{asm:R} three geodesic lines in the boundary of the tessellation~$\TR$ cannot form a triangle.

Section~\ref{sec:struct-th-path} deals with the local structure of the thickened paths. Namely, we show that each thickened path is split into \emph{bottles} by levels (\emph{bottlenecks}) which contain only one copy of the fundamental domain, and the structure of each bottle is then described by Lemma~\ref{lem:bottle-local}. The local rules from this description give rise to the construction of the topological Markov chain in Section~\ref{sec:markov-coding}. 
The main result of this section, Theorem~\ref{thm:th-path-vs-sphere}, shows that every thickened path can be produced by this Markov chain, and conversely that every path defined by this chain is indeed a thickened path. This result  should be of independent interest and may have application elsewhere.
 
In Section~\ref{sec:operations} we present  some techniques for cutting and joining  thickened paths which we use in Section~\ref{subsec:strconn-aper}  to show that the Markov chain is strongly connected and aperiodic or, in other words, its adjacency matrix $\Pi$ has a power with all elements positive. The same techniques are  also  used in Section~\ref{sec:proof-main}.   
 
Section~\ref{sec:sph-avg} shows that these properties of the Markov chain allow us to construct its Parry measure and then to relate the spherical averages for our group to powers of a Markov operator $P$ associated to this coding. We also show that  the symmetry of the coding yields a relation between $P$ and its adjoint $P^*$.

Section~\ref{sec:proof-main} concludes the proof of the main theorem. To do this, we first formulate the new general theorem on pointwise convergence of powers of a Markov operator, Theorem~\ref{thm:convergence}.  Most  of Section~\ref{sec:proof-main} is then devoted to checking    that the conditions necessary for this theorem apply in our case,  including the most complicated one, that involving an inequality between the operator and its adjoint, as discussed above. This is proved in Section~\ref{subsec:proof-ineq} using techniques from Section~\ref{sec:operations}.  The proof of Theorem~\ref{thm:main} assuming  Theorem~\ref{thm:convergence} is concluded in Section~\ref{subsec:conclusion}. 

Finally, in Section~\ref{sec:proof-convergence} we give the proof of the new general result, Theorem~\ref{thm:convergence} on pointwise convergence for Markov operators.
As discussed above, the argument here follows that in~\cite{Buf-Annals} and is based on Rota's ``Alternierende Verfahren'' scheme.

We remark that many of the proofs, especially in Sections~\ref{sec:operations} and~\ref{subsec:proof-ineq}, may seem rather long and complicated; this is partly because of the generality in which we are working. In many cases the situation with $N(\R) \geq 5$ simplifies considerably; on the other hand the cases  $N(\R) = 3,4$ simplify in different ways and $N(\R) = 3$ encompasses in particular the modular group $SL(2,\ZZ)$. In almost all cases (to be precise, everywhere except in case (4) of Proposition~\ref{lem:str-conn}), our proofs depend only on the geometry of $R$ and not on analysing the particular pattern of side pairings.

\subsection{Historical remarks}

For two rotations of a sphere, convergence of spherical averages was established by Arnold and Krylov \cite{ArnKr}, and a general mean ergodic theorem for actions of free groups was proved by Guivarc'h \cite{Guivarch}.

The first general pointwise ergodic theorem for convolution averages on a countable group is due to Oseledets \cite{Oseled} who
relied on the martingale convergence theorem.
The first general pointwise ergodic theorems for free semigroups and groups were given
by R.I.~Grigorchuk in 1986 \cite{Grig86}, where the main result is Ces{\`a}ro convergence of spherical averages
for measure-preserving actions of a  free semigroup and group.  Convergence of the actual spherical averages for free groups was established by Nevo \cite{Nevo94} for functions in $L^2$ and Nevo and Stein \cite{NeSt94} for functions in $L^p$, $p>1$ using spectral theory methods.
Nevo, Stein, and Margulis \cite{NeSt97,MarNeSt} considered ball averages for actions of connected semisimple Lie group with finite center and no nontrivial compact factors and showed that these ball averages converge almost everywhere and in $L^p$, $p>1$.
Note that, as shown by Tao~\cite{Tao}, whose argument is inspired by Ornstein's counterexample \cite{Ornst}, pointwise convergence of spherical averages for functions in~$L^1$ does not hold even for actions of free groups.

The method of Markov operators in the proof of ergodic theorems for actions of free semigroups and groups
was suggested by R. I. Grigorchuk \cite{Grig99,Grig00}, J.-P. Thouvenot (oral communication), and in \cite{Buf-FAA}.
In \cite{Buf-Annals} pointwise convergence is proved for Markovian spherical averages under the additional assumption that
the Markov chain be reversible. The key step in \cite{Buf-Annals} is the triviality of the tail sigma-algebra for the corresponding
Markov operator; this is proved using Rota's ``Alternierende Verfahren'' \cite{Rota}, that is to say, martingale convergence.
Another result in this direction was obtained in \cite{BowBufRom}; it states the mean convergence for analogues of spherical averages for an arbitrary Markov chain satisfying very mild conditions. It is not known whether similar result holds for pointwise convergence.

The study of Markovian averages is  motivated by the problem of ergodic theorems for general countable groups, specifically, for groups admitting a Markovian coding such as Gromov hyperbolic groups \cite{Gromov} (see e.g. Ghys--de la Harpe \cite{GhysDelaharpe} for a detailed discussion of the Markovian coding for Gromov hyperbolic groups). The first results on convergence of spherical averages for Gromov hyperbolic groups, obtained  under strong exponential mixing assumptions on the action, are due to Fujiwara and Nevo \cite{FuNe}. For actions of hyperbolic groups on finite spaces, an ergodic theorem was obtained by L. Bowen in \cite{Bow}.

Ces{\`a}ro convergence of spherical averages for all measure-preserving actions of Markov semigroups, and, in particular, Gromov hyperbolic groups, was established in \cite{BufKhrKlim,BufKlim-TrMIAN}; earlier partial results were obtained in \cite{Buf-RMS,Buf-AMSTr}. In the special case of hyperbolic groups a shorter proof of this theorem, using the method of Calegari and Fujiwara \cite{CalFuji}, was later given by Pollicott and Sharp \cite{PolSh}.
Using the method of amenable equivalence relations, Bowen and Nevo \cite{BowNe-EM,BowNe-JAM,BowNe-AnENS, BowNe-GGD} established ergodic theorems for ``spherical shells''  in Gromov hyperbolic groups. For further background see the surveys~\cite{Nevo06,GorNev,BufKlim-EJC}.

\medskip

\subsection{Acknowledgments}
The research of A.~Bufetov on this project has received funding from the European Research Council (ERC) under the European Union's Horizon 2020 research and innovation programme under grant agreement No 647133 (ICHAOS).
A.~Klimenko's research was partially funded by Russian Fund of Basic Research (grants No. 18-31-20031 and 18-51-15010).

\section{Definitions and notation}
\label{sec:defs-notat}

Let $G$  be a finitely generated non-elementary
Fuchsian group acting on the hyperbolic disk $\DD$ with  fundamental domain $\R$, which we take to be closed.
We suppose $\R$ to be a finite-sided convex polygon with
vertices contained in $\overline\DD=\DD \cup \dd \DD$,
such that the interior angle at each vertex is strictly less than $\pi$.
By a \emph{side} of $\R$ we mean the closure in $\DD$ of the
geodesic arc joining a pair of adjacent vertices.
We allow the infinite area case in which some adjacent vertices on $ \dd \DD$
are joined by an arc contained in $ \dd \DD$; we do not
count these arcs as sides of $\R$.
Further we usually mean by \emph{vertices} of $\R$ only vertices inside~$\DD$.
Sometimes it is convenient to count as vertices also the side ends that belong to $\dd\DD$, these instances will be specified explicitly.
Two sides are called \emph{adjacent} if they share a common vertex lying in $\DD$. We refer to each image $g\R$ of $\R$ by an element $g \in G$  as a \emph{domain}.

We assume that the sides of $\R$ are paired; that is, for each side $s$
of $\R$ there is a (unique) element $e\in G$ such that
$e(s)$ is also a side of $\R$ and the domains $\R$ and $e(\R)$ are adjacent
along $e(s)$.  Notice that this includes the possibility that
$e(s) = s$,  in which case $e$ is elliptic of order $2$
and the side $s$ contains the   fixed point of $e$ in its interior.
The condition  that the vertex angle be strictly less than $\pi$  excludes the possibility that
the fixed point of $e$ is counted as a vertex of $\R$. Since the element pairs the side to itself the possibility of more than one elliptic fixed point on one side is excluded, for the existence of two such points implies the existence of infinitely many contained in the one side. Note also that the treatment of order two elliptic fixed points in~\cite{BirSer} and elsewhere is slightly different.

%%%%%%%%%%%%%%%%%%%%
\begin{figure}[hbt]
\centering
\includegraphics{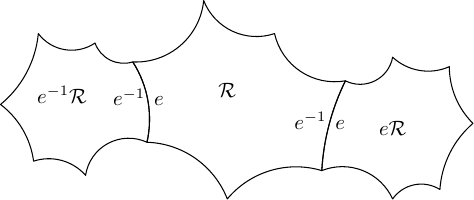}
\caption{Labelling the sides of the fundamental domain $\R$. Note that the label $e$ is interior to $\R$ on  the side of $\R$ adjacent to the domain $e^{-1}\R$.}
\label{fig:labelling}
\end{figure}
%%%%%%%%%%%%%%%%

We denote by $\dd \R$ the union of the sides of $\R$, in other words,
$\dd \R$ is the part of the  boundary of $\R$ inside the disk $\DD$.
 Each side of $\dd \R$ is assigned  two labels, one interior to $\R$ and one exterior, in such a way that
the interior and exterior labels are mutually inverse elements of $G$.
We label  the side $s \subset \dd R$ interior to $\R$   by  $e$ if $e$ carries $s$ to another side $e(s)$  of $\R$, while we label the
same side exterior to $\R$ by $e^{-1}$, see Figure~\ref{fig:labelling}. With this convention, $\R$ and $e^{-1}(\R)$ are adjacent along the side $s$ whose \emph{interior} label is $e$, while the side  $e(s)$  has interior label $e^{-1}$.

Let $G_0$ denote the set of group elements which label  sides of $\R$. The labelling extends to a $G$-invariant labelling of all sides of the tessellation  $\TR$ of $\DD$ by images of $\R$, where by a \emph{side} of $\TR$, we mean a side of $g\R$ for some $g \in G$.  The conventions have been chosen in such a way that if two domains $g \R, h\R$ are adjacent along a common side $s$, then  $h^{-1}g \in G_0$ and the label on $s$ interior to $g \R$ is $h^{-1}g$, while that on the side interior to $h\R$ is $g^{-1}h$.
Suppose that $O$ is a fixed basepoint in $\R$ and that $\gamma$ is an oriented path in $\DD$ from $O$ to $gO$, $g \in G$, which avoids   all vertices of  $\TR$,   and which passes through in order adjacent domains  $\R = g_0\R, g_1\R,  \ldots, g_n \R = g\R$. Then the  labels of the sides crossed by $\gamma$, read in such a way that if $\gamma$ crosses from $g_{i-1}\R $ into $g_i\R$ we read off the label $e_i = g_{i-1}^{-1} g_i$ of the common side  interior to $g_i\R$,
are in order $e_1,  e_2, \ldots, e_n$ so that   $g= e_1 e_2 \ldots e_n$.
This proves the well-known fact that  $G_0$ generates $G$, see for example~\cite{Bea}.

As explained in the introduction, the fundamental domain $\R$ is said to have \emph{even corners} if for each side $s$ of $\R$,
the complete geodesic in $\DD$ which extends $s$ is contained in the sides of $\TR$.
This condition is satisfied for example, by the regular $4g$-gon of interior angle ${\pi}/{2g}$ whose sides can be paired with the standard generating set
\begin{equation*}
\biggl\{ a_i, b_i, i = 1, \dots, g \biggm|\prod_{i=1}^g [a_i,b_i]  \biggr\}
\end{equation*}
to form a surface of genus $g$. It is also satisfied by the modular group $SL(2,\ZZ)$ with the classical fundamental domain $\{z:|\Re z|<1/2, |z|>1\}$ in the upper half plane. For further discussion on the even corners condition, see the references in the introduction.

Note that under the even corners condition there exists a ``chequered coloring'' of the domains in $\TR$ (or elements of $G$):
one can color each  domain either in black or white in such a way that each side of $\TR$ separates domains of different color.

We will frequently consider the union (or the collection) of all $2n(v)$ domains in $\TR$ adjacent to a vertex~$v$. We call
this the \emph{flower} at $v$ and denote it by $\F_v$ and refer to the individual domains in $\F_v$ as \emph{petals}, while the sides between its petals we call its \emph{radii}.
Note that $\F_v$ is a convex polygonal domain. Indeed, it is a star domain with respect to $v$ and the internal angle at any vertex $u$ on its boundary
contains either one or two sectors; since $n(u)\ge 2$, this angle does not exceed $\pi$.
Moreover, the angle  $\pi$ may occur only at the common vertex $w$ of two petals of the flower, and in this case $n(w)=2$.

Let us also denote the geodesic line passing through a side $s$ or a pair of vertices $u,v$ in $\TR$ as $\ell(s)$ or $\ell(uv)$.

We start with some properties of the tessellation~$\TR$ which are consquences of Assumption~\ref{asm:R}.

\begin{lemma}\label{lem:no-trig}
Under Assumption~\ref{asm:R} there are no vertices $a$, $b$, $c$ of $\TR$ such that the lines $\ell(ab)$, $\ell(bc)$, and $\ell(ca)$ belong to $\dd\TR$.
\end{lemma}

\begin{proof}
Assume the contrary: there exists a triangle $\Delta=abc$ in $\dd\TR$.
Note that $\Delta$ cannot be a fundamental domain  since Assumption~\ref{asm:R} excludes compact triangular domains.
Therefore, on $\dd\Delta$ there is a point $p$ belonging to at least two fundamental domains in $\Delta$. Then there is a ray $\alpha$ in $\dd\TR$ that starts at $p$ and goes inside $\Delta$. The ray $\alpha$ cuts $\Delta$ into two regions, at least one of them being triangular. Choose this region as a new triangle $\Delta'$, which also violates the statement of the lemma.

This process can be repeated indefinitely, and each iteration decreases the number of fundamental domains inside the triangle. However, this number is finite since the area of the triangle is finite, and we arrive at a contradiction.
\end{proof}

The next proposition was stated under slightly stronger assumptions in~\cite[Lemma 2.2]{BowSer} in the case in which  $P$ is a fundamental domain. We will use it for $P$ equal to either a fundamental domain or a flower, see Corollary~\ref{cor:flower-no-intersect} below.

\begin{lemma}\label{lem:no-intersect}
Suppose that Assumption~\ref{asm:R} holds for $\TR$ and consider a convex polygon~$P$ with sides lying in $\dd\TR$.
Take any two different lines $\ell_1,\ell_2$ from $\dd\TR$ that intersect $\dd P$ but not $\Int P$. Then either $\ell_1$ and $\ell_2$ do not intersect or they intersect at a vertex of $P$.
\end{lemma}

\begin{proof}
Assume the contrary: $\ell_1\cap\ell_2=p\notin\dd P$, see Figure~\ref{fig:lem-no-intersect}. The line $\ell_j$ meets $\dd P$ either in a side $s_j$ or a vertex $v_j$. In the former case let $v_j$ be the end of $s_j$ closest to $p$.
Let $\gamma=u_0u_1\dots u_k$ ($u_0=v_1$, $u_k=v_2$) be the part of $\partial\P$ between $v_1$ and $v_2$ that lies inside the triangle $\Delta=v_1v_2p$. Note that the $s_j$'s are not included in $\gamma$. 

%%%%%%%%%%%%%%%%%%%%
\begin{figure}[hbt]
\centering
\includegraphics{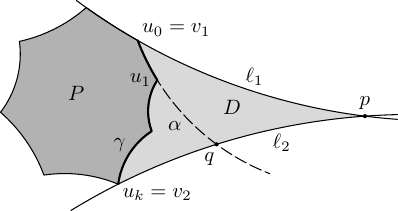}
\caption{The proof of Lemma~\ref{lem:no-intersect}.}
\label{fig:lem-no-intersect}
\end{figure}
%%%%%%%%%%%%%%%%

Among all pairs $(\ell_1,\ell_2)$ violating the statement of the lemma choose the one for which the number $k$ of sides of $\gamma$ is minimal.

If $k=1$,  Lemma~\ref{lem:no-trig} for the triangle $v_1v_2p$ in $\dd\TR$ yields a contradiction.
Otherwise, consider the ray $\alpha\subset\ell(u_0u_1)$ that is the continuation of the side $u_0u_1$ past the point $u_1$. This ray enters the domain $D$ bounded by the curve $\gamma$ and the segments $v_1p$, $v_2p$, since the inner angle of $D$ at $u_1$ is more that $\pi$.
This ray must exit $D$ at some point $q$, which cannot belong to $\gamma$ (otherwise $P$ is not convex) or to $v_1p$ (since $\ell(u_0u_1)$ and $\ell_1$ intersect only at $v_1$).
Therefore, $q$ belongs to $v_2p$, and $\ell_1'=\ell(u_0u_1)$, $\ell_2'=\ell_2$ is a pair of lines also violating the statement of the lemma but for which 
$\gamma'=\gamma\setminus u_0u_1$ has 
$k-1$ sides.
\end{proof}

\begin{cor}\label{cor:flower-no-intersect}Suppose that Assumption~\ref{asm:R} holds for $\TR$.
Consider any two different sides $s_1,s_2$ of fundamental domains lying on the boundary of the flower $\F_v$. Then either $\ell(s_1)$ and $\ell(s_2)$ coincide,
or they   intersect  at one or other end of either $s_1$ or $s_2$, or they do not intersect.
\end{cor}

\begin{proof}
We need some care since \emph{one} side of $\F_v$ \emph{as a convex polygon} can contain \emph{two} sides of~$\TR$: if a common vertex $w\ne v$ of two domains in $\F_v$
has $n(w)=2$, then two sides on the boundary of $\F_v$ adjacent to $w$ lie on the same geodesic and thus they form one side of~$\F_v$ as a polygon.
Let us refer to such side of the polygon $\F_v$ as a \emph{compound} side.

In view of Lemma~\ref{lem:no-intersect}, we only need to rule out the case in which
 $s_1$ and $s_2$ are contained in two adjacent compound sides of the polygon $\F_v$, separated by an intervening side or sides. We will show that compound sides cannot be adjacent.
Indeed, assume that both $u_1u_2$ and $u_2u_3$ are compound, where the $u_i$ are vertices of petals of $\F_v$. Thus $u_1u_2$ and $u_2u_3$ contain vertices $w_1$ and $w_2$  in their interiors.  Since each $vw_j$ is a side common to two petals of $\F_v$, it follows that either $vw_1u_2w_2$  or $vw_1u_2$ is a fundamental domain. We see that either $N(\R)=4$ and $n(w_{1,2})=2$,  or $N(\R)=3$ and $\R$ is compact, both of which cases are excluded by Assumption~\ref{asm:R}. 
\end{proof}
 
\section{Structure of thickened paths}
\label{sec:struct-th-path}

As explained in the introduction, a thickened path between two  domains $\A$ and $\B$ in $\TR$ is the union of all translates of $\R$
crossed by any possible shortest paths between $\A$ and $\B$. In this section we describe the detailed structure of thickened paths. The  results will be used in  Section~\ref{sec:markov-coding}  to construct a Markov coding that generates all possible thickened paths:
we will show there that  the  features discussed are also \emph{sufficient} conditions for a union of fundamental domains to be a thickened path.

From now on, we suppose Assumption~\ref{asm:R}  holds for $G$ and $\R$ and do not mention this explicitly in the statements.

\subsection{Thickened paths}
\label{subsec:thick-paths-definition}

 Let $\A$, $\B$ be two fundamental domains from $\TR$. A \emph{path} from $\A$ to $\B$ is a sequence $\underline\R=(\R_0=\A,\R_1,\dots,\R_n=\B)$ of domains from $\TR$,
so that $\R_i$ and $\R_{i+1}$ have a common side for all $i=0,\dots, n-1$. The number $n$ here is the \emph{length} of the path $\underline\R$.
Equivalently, if $\R_i=g_i\R$ then $\underline\R$ is a path if and only if $\underline g=(g_0,\dots,g_n)$ is a path from $a=g_0$ to $b=g_n$ in the Cayley graph of $G$ with respect to $G_0$.

The path $\underline\R$ is \emph{shortest} if its length is minimal among all paths with the same ends.
The \emph{distance} $\dist(\A,\B)$ between $\A$ and $\B$ is the length of  a shortest path between them.

The \emph{thickened path} from $\A$ to $\B$ is the collection of all domains in $\TR$ that belong to some shortest path from $\A$ to $\B$.
This thickened path $\underline\S$ is decomposed into \emph{levels} $\S_k$, $k=0,1, \dots, n=\dist(\A,\B)$.
Namely, a domain $\C\in\underline\S$ belongs to the level $\S_k$ if $\dist(\A,\C)=k$ and therefore, $\dist(\C,\B)=n-k$.
We also observe that two domains $\C$, $\C'$ in $\underline\S=(\S_0,\dots,\S_n)$ can have a common side only if their levels $k,k'$ differ by one.
Indeed, if $k<k'-1$, then $n=\dist(\A,\B)\le \dist(\A,\C)+\dist(\C,\C')+\dist(\C',\B)=k+1+(n-k')<n$.
The case $k=k'$ is impossible: the cycle $\A\C\C'\kern -2pt\A$ of odd length $2k+1$ contradicts  the ``chequered coloring'' of $\TR$ (see the discussion of the even corners condition in Section~\ref{sec:defs-notat}).

\subsection{Convexity of thickened paths}
\label{subsec:thick-paths-convex}

In this subsection we prove that the thickened path between two domains $\A$ and $\B$ is the smallest convex union of domains containing them. 
We begin by describing an alternative method of finding the distance between two fundamental domains.

Let us say that a geodesic $\gamma$ in $\dd\TR$ \emph{separates} domains $\A$ and $\B$ if $\A$ and $\B$ lie in different half-planes with respect to $\gamma$ and denote the set of geodesics separating $\A$ and $\B$ by $\mathbf{S}_{\A,\B}$.
In particular, if $\A$ and $\B$ share a common side $s$, we have $\mathbf{S}_{\A,\B}=\{\ell(s)\}$.

Now consider any path $\underline{\R}=(\R_0,\dots,\R_n)$ from $\A$ to $\B$ and let $s_i$ be the common side of $\R_i$ and $\R_{i+1}$. Then every geodesic $\gamma\in\mathbf{S}_{\A,\B}$ appears at least once among the $\ell(s_i)$, $i=0,\dots,n-1$. (Note that we are not at this point assuming the $\ell(s_i)$ are distinct, see below.) Indeed, otherwise for every $i$ the domains $\R_i$ and $\R_{i+1}$ lie in the same half-plane with respect to $\gamma$, so by transitivity $\A=\R_0$ and $\B=\R_n$ also lie in the same half-plane. This means that $\dist(\A,\B)\ge\#\mathbf{S}_{\A,\B}$. Let us show that this inequality is indeed an equality.  

\begin{lemma}\label{lem:Dist-NSep}
	The distance between two fundamental domains $\A$ and $\B$ equals the number of geodesics in $\mathbf{S}_{\A,\B}$.
\end{lemma}

\begin{proof}
	From the consideration above we see that it is sufficient to construct a path of length $m=\#\mathbf{S}_{\A,\B}$ from $\A$ to $\B$. To do so choose points $a\in\Int\A$ and $b\in\Int\B$ so that the geodesic segment $I=ab$ does not pass through any vertex of $\TR$. 
	Then the geodesic lines from $\dd\TR$ crossed by $I$ are exactly the geodesics separating $a$ and $b$, or, equivalently, $\A$ and $\B$. The points of intersection of $I$ with these $m$ lines from $\mathbf{S}_{\A,\B}$ are different and by convexity $I$ cannot enter any domain twice. Hence $I$ traverses $(m+1)$ domains $\R_0=\A,\R_1,\dots,\R_m=\B$, and for any $i=0,\dots, m-1$ the domains $\R_i$ and $\R_{i+1}$ have a common side.
\end{proof}

\begin{remark}\label{rem:cross-once}
	Let us say that a path $\underline \R=(\R_0=\A,\dots,\R_n=\B)$ \emph{crosses a line $\gamma$} if $\gamma = \ell(s_i)$ for some $i$, where $s_i=\R_i\cap\R_{i+1}$. Then if $\gamma$ is the shortest path between $\A$ and~$\B$, it crosses every line $\gamma$ from $\mathbf{S}_{\A,\B}$ exactly once (i.e. $\gamma = \ell(s_i)$ for only one $i$) and does not cross any other lines.
\end{remark}

\begin{proof}
	We have seen above that every line from $\mathbf{S}_{\A,\B}$ appears in the sequence $\{\ell(s_i)\}_{i=0}^{n-1}$ \emph{at least} once. If some line from $\mathbf{S}_{\A,\B}$ appears twice in this sequence, or if any line outside of  $\mathbf{S}_{\A,\B}$ appears there, we have $n>\#\mathbf{S}_{\A,\B}$. On the other hand, for the shortest path we have $n=\dist(\A,\B)=\#\mathbf{S}_{\A,\B}$ by Lemma~\ref{lem:Dist-NSep}.
\end{proof}

The following proposition describes a thickened path as a convex set.

\begin{prop}\label{prop:NonSep}Let $\A$ and $\B$ be two fundamental domains in $\TR$. The  thickened path  from $\A$ to $\B$ is the  minimal convex union of fundamental domains that contains both $\A$ and $\B$.  If $\R$ is compact, then its boundary contains at least one side of $\A$ and  one  of $\B$.
\end{prop}

\begin{proof}
\noindent\textit{Step 1.} Denote by $\mathbf{NS}_{\A,\B}$ the set of all lines $\ell\subset\dd\TR$
 that do  not separate $\A$ and $\B$. For every $\ell\in\mathbf{NS}_{\A,\B}$ consider
the half-plane $H_\ell$ bounded by $\ell$ and containing both $\A$ and $\B$.
Set
\begin{equation*}
\G=\bigcap_{\ell\in\mathbf{NS}_{\A,\B}} H_\ell.
\end{equation*}
We claim that $\G$ is the thickened path from $\A$ to $\B$.
By the previous corollary, no shortest path from $\A$ to $\B$ intersects any line $\ell\in \mathbf{NS}_{\A,\B}$.
Therefore, the thickened path from $\A$ to $\B$ is contained in $\G$.

On the other hand, consider any fundamental domain $\C\subset\G$. As in the proof of Lemma~\ref{lem:Dist-NSep}, choose generic points $a\in\A$, $b\in B$, $c\in C$ so that
the segments $ac$ and $cb$ do not contain vertices. The sequences $\R_0=\A,\R_1,\dots,\R_k=\C$ and $\R_k=\C,\R_{k+1},\dots,\R_{k+l}=\B$ of fundamental domains
traversed respectively by $ac$ and $cb$ are shortest paths from $\A$ to $\C$ and $\C$ to $\B$ respectively. Since $\G$ is convex, the segments $ac$ and $cb$ lie in $\G$ and hence so do all the $\R_j$'s.

It remains to show that $\R_0 ,\R_1,\dots,\R_{k+l}$ is a shortest path. Let $\ell_j$ be the geodesic line separating $\R_j$ and $\R_{j+1}$. Then the lines $(\ell_j)_{j=0}^{k+l-1}$ are the consecutive geodesics intersected by the path $acb$. Every geodesic $\gamma$ in $\dd\TR$ intersects either none or two sides of the triangle $abc$. If $\gamma$ intersects the sides $ac$ and $cb$, then $\gamma$ does not separate $a$ and $b$. But then $c\in\C$ does not belong to $H_\gamma$, and hence to $\G$. On the other hand, $\gamma$ intersects $ab$ if and only if it belongs to $\mathbf{S}_{\A,\B}$, thus each line from $\mathbf{S}_{\A,\B}$ crosses the curve $acb$ exactly once. 
Therefore, $k+l=\#\mathbf{S}_{\A,\B}=\dist(\A,\B)$.

\smallskip

\noindent\textit{Step 2.} Assume that $\G'\subset\G$ is a smaller convex union of fundamental domains containing $\A$ and $\B$.
Note that $\G$ (hence $\G'$) contains only finite number of fundamental domains, since there are only finitely many domains at  distance not more than $\dist(\A,\B)$ from $\A$. Thus $\G'$ is a convex polygon and the supporting half-planes of its sides are $H_{\ell_k}$, $k=1,\dots, K$, for some lines $\ell_k\in\mathbf{NS}_{\A,\B}$.
Therefore,
\begin{equation*}
\G'=\bigcap_{k=1}^K H_{\ell_k} \supset \bigcap_{\ell\in\mathbf{NS}_{\A,\B}} H_{\ell}=\G.
\end{equation*}

\noindent\textit{Step 3.} To prove the final statement, choose generic points $a\in\A$, $b\in\B$ so that the line~$\ell(ab)$ does not contain any vertices of $\TR$. Denote by $\alpha$ the ray in $\ell(ab)$ starting from $a$ in the direction away from $b$. Let $a'$ be the first intersection of~$\alpha$ with $\dd\TR$ and let $\ell'\ni a'$ be the corresponding geodesic line from $\dd\TR$; $a'$ exists since $\A$ is compact. Thus $\ell'\in\mathbf{NS}_{\A,\B}$ as it does not cross the segment $ab$.
Therefore, no point on $\alpha$ after $a'$ belongs to $H_{\ell'}$ and hence to $\G$, and no point before $a'$ is separated from $a$ by any $H_\ell$. Thus $\alpha\cap\G=aa'$, so $a'\in\dd\G$ and the side of $\A$ that contains $a'$ belongs to $\dd\G$.
\end{proof}

\subsection{Levels of thickened paths}

Each level in a thickened path is a union of domains.  We now show that each level can contain at most two domains.  

Let $\underline\S=(\S_0=\A,\S_1,\dots,\S_n=\B)$ be a thickened path from $\A$ to $\B$. Consider its closure $\Clos_{\overline\DD}\underline\S$ and its boundary $\dd_{\overline\DD}\underline\S$ in $\overline\DD$. They are homeomorphic respectively to a closed disk and a circle.

Consider the intersection $\dd_{\overline\DD}\underline\S\cap \dd_{\overline\DD}\A$. It is nonempty: if $\R$ is compact, this is stated in Proposition~\ref{prop:NonSep}, otherwise any point of $\dd_{\overline\DD}\A\cap\dd\DD$ belongs to $\dd_{\overline\DD}\underline{\S}$. Moreover, it is connected.  Indeed, take any $p$ and $p'$ lying on different sides in this intersection and consider a segment $J\subset\A$ connecting $p$ and $p'$. Then $J$ separates $\Clos_{\overline\DD}\underline\S$ into two connected components. If both components contain fundamental domains (besides parts of $\A$),
choose any $\C\ne\A,\B$ such that $\B$ and $\C$ lie in different components. Then the path from $\C$ to $\B$ inside $\G$ must cross $J$ and hence $\A$, so $\C$ cannot belong to a shortest path from $\A$ to $\B$. Therefore, one connected component contains only a part of $\A$. But then the boundary of this connected component apart from $J$ is a segment of $\dd_{\overline\DD}\underline\S\cap \dd_{\overline\DD}\A$ that connects $p$ and $p'$.

We conclude that $\dd_{\overline\DD}\underline\S\setminus (\dd_{\overline\DD}\A\cup \dd_{\overline\DD}\B)$ consists of two arcs, which we call the \emph{left} and \emph{right} boundaries of $\underline\S$ and denote $\dd_{L,R}\underline\S$.
Namely, going clockwise around $\dd_{\overline\DD}\underline\S$ we pass through an arc of $\dd_{\overline\DD}\A$, then $\dd_{L}\underline\S$, then an arc of $\dd_{\overline\DD}\B$, and $\dd_{R}\underline\S$. Both $\dd_{L,R}\underline\S$ are oriented from $\dd_{\overline\DD}\A$ to $\dd_{\overline\DD}\B$. Sometimes we will use the same notation $\dd_{L,R}\underline\S$ for the parts of these boundaries that lie inside $\DD$.

\begin{prop}\label{prop:thick-path-LR} Let $\underline\S=(\S_0=\A,\S_1,\dots,\S_n=\B)$ be a thickened path from $\A$ to $\B$.
Consider the sequence of adjacent domains $\L_0=\A,\L_1,\dots,\L_m=\B$ which meet $\dd_L\underline\S$ in a point or side and the similar sequence of domains $\R_0=\A,\R_1,\dots,\R_{m'}=\B$ which meet $\dd_R\underline\S$.
Then:\\
1) both these sequences are shortest paths from $\A$ to $\B$, hence $m=m'=n$;\\
2) every domain in $\underline\S$ belongs to one of these two sequences, hence $\S_j=\{\L_j,\R_j\}$ for every $j=0,\dots,n$; it is possible that $\L_j=\R_j$.
\end{prop}

\begin{proof}
1. Consider the side $s_j$ between $\L_j$ and $\L_{j+1}$. Then $\ell(s_j)$ passes inside $\underline\S$ and hence separates $\A$ and $\B$ by Proposition~\ref{prop:NonSep}.
Moreover, every geodesic $\ell$ in $\TR$ separating $\A$ and $\B$ intersects $\underline{\S}$ in a segment $I(\ell)$ with one end on $\dd_L\underline{\S}$ and another one on  $\dd_R\underline{\S}$. Thus if $s_j\subset I(\ell)$ then $s_j$ is adjacent to the end of $I(\ell)$ lying on $\dd_L\underline{\S}$. Therefore, each of the $n$ geodesics separating $\A$ and $\B$ produces exactly one such $s_j$, so $m=n$.

2. This is Lemma 2.7 in \cite{BirSer}. We give a proof for the sake of completeness.

Assume that the fundamental domain~$\C$ lies strictly inside $\underline{\S}$. Then $\C$ is compact. Let $s_i$, $i=1,\dots,N(\R)$ be the consecutive sides of $\C$ and let $H_i$ be the half-plane bounded by $\ell(s_i)$ that does not contain $\C$.
By Lemma~\ref{lem:no-intersect} the lines $\ell(s_j)$ and $\ell(s_k)$ intersect only for adjacent sides $s_j$ and $s_k$. Therefore, $H_j\cap H_k\ne\varnothing$ also only for adjacent $s_j$ and $s_k$. This means that at most two lines $\ell(s_j)$ separate $\A$ from $\C$ and at most two   separate $\B$ from $\C$. Hence there are at least $N(\R)-4$ lines of the form $\ell(s_j)$ that do  not separate $\A$ and $\B$. In the case $N(\R)\ge 5$ we arrive at the contradiction with Step 1 in Proposition~\ref{prop:NonSep}: a line not separating $\A$ and $\B$ cannot enter the interior of $\underline{\S}$.

If $N(\R)=4$ the only remaining case (up to  renumbering the $s_j$) is that $\A\subset H_1\cap H_2$, $\B\subset H_3\cap H_4$. Then $p=s_2\cap s_3$, $q=s_4\cap s_1$ are opposite vertices of $\C$. Assumption~\ref{asm:R} states that $n(p)>2$ or $n(q)>2$. If, say, $n(p)>2$, there is a line~$\ell^*$ from $\dd\TR$ that intersects $\dd\C$ only at $p$. Denote the half-plane bounded by $\ell^*$ and not containing $\C$ by $H^*$.
Lemma~\ref{lem:no-trig} yields that $\ell^*$ does not intersect $\ell(s_1)$ and $\ell(s_4)$, thus $H_1\cap H^*=H_4\cap H^*=\varnothing$. Therefore, again using Step 1 in the proof of Proposition~\ref{prop:NonSep}, 
both $\A$ and $\B$ lie outside   $H^*$, so we arrive at the same contradiction for $\ell^*$.
\end{proof}

 \subsection{Bottles and bottlenecks}
 
Continuing our consideration of thickened paths, we call levels which contain only one domain \emph{bottlenecks}; for such levels we have  $\L_j = \R_j$. 
Thus a domain $\C\in\underline\S$ is a bottleneck if all shortest paths from $\S_0=\A$ to $\S_n=\B$ pass through $\C$.
The bottlenecks divide a thickened path into \emph{bottles}. Namely, if $\S_r$ and $\S_s$ are bottlenecks and $\S_k$, $r<k<s$ are not, then
$\underline\S_r^s=\bigcup_{k=r}^s\S_k$ is a \emph{bottle} (so each bottle has two bottlenecks, one at each end).

 We now focus on the structure of one bottle. Note that if in a thickened path $\underline\S$ the levels $\S_j=\{\A'\}$ and $\S_k=\{\B'\}$ contain one domain each, then $(\S_i)_{i=j}^k$ is the thickened path from $\A'$ to $\B'$. Therefore, every bottle is a thickened path between its two bottlenecks, and we assume for the rest of this section that $\underline\S=(\S_0,\S_1,\dots,\S_n)$ is a bottle: $\S_j$ contains two domains for every $j=1,\dots,n-1$ and one domain for $j=0$ and $j=n$.

Denote by $s^L_j$ (respectively, $s^R_j$) the common side of $\L_j$ and $\L_{j+1}$ (respectively, $\R_j$ and $\R_{j+1}$) and consider the set
\begin{equation*}
\mathbb A=\bigcup_{i=0}^{n}\bigl(\Int(\L_i)\cup \Int(\R_i)\bigr)\cup \bigcup_{i=0}^{n-1} \bigl(\Int(s^L_i)\cup\Int(s^R_i)\bigr).
\end{equation*}
This set is homeomorphic to an annulus and can be retracted to the boundary of the closed disk $\Clos_{\overline\DD}(\underline\S)$.
Now let
\begin{equation*}
\T=\Int\underline\S\setminus\mathbb A.
\end{equation*}
This set is a union of all vertices and sides of $\TR$ that lie inside $\underline\S$ and have no points in common with $\dd_{\overline\DD}\underline\S$;
note that $\T$ cannot contain fundamental domains by Proposition~\ref{prop:thick-path-LR}.
Thus set $\T$ is closed, connected and simply connected. In other words, the graph~$\T$ is a tree. We call $\T$ the \emph{core} of the bottle. 

\begin{prop}The  core $\T$ of a bottle is a linear graph, that is, its vertices can be enumerated as $v_0,v_1,\dots,v_r$ in such a way that for every $j=1,\dots, r$ the vertices $v_{j-1}$ and $v_j$ are adjacent and its edges are sides $v_{j-1}v_j$ for $j=1,\dots, r$.
\end{prop}

\begin{proof}
A tree is a linear graph if and only if it has not more than two \emph{leaves}, a leaf being a vertex with only one adjacent edge.  Let us check that the core $\T$ cannot have more than two leaves. Consider any leaf $v$, and let $s\subset\T$ be the only edge adjacent to $v$. Then all other sides in $\TR$ adjacent to $v$ must have their other end on $\dd\underline\S$, so going around $v$ starting from $s$, one traverses some interval $\mathbf{S}_v$, of length $2n(v)$,  of the circuit
\begin{equation*}
\dots-\L_1-\dots-(\L_n=\R_n)-\dots-\R_1-(\R_0=\L_0)-\dots
\end{equation*}
Note that all domains in the interval $\mathbf{S}_v$ except the first and the last meet $\T$ only in $v$. If all domains in $\mathbf{S}_v$ are $\L_j$'s: 
$\mathbf{S}_v=(\L_k,\dots,\L_{k+2n(v)-1})$, then the side $s$ separates $\L_k$ and $\L_{k+2n(v)-1}$, so the path $(\L_0,\dots,\L_n)$ is not shortest since it can be shortcut by crossing  $s$. Therefore, every interval $\mathbf{S}_v$ must contain some domains of type $\R_j$, so that either $\L_n$ or $\L_0$ is internal to the sequence. Since a domain cannot be internal to two $\mathbf{S}_v$'s there are at most two such intervals and hence at most two leaves.
\end{proof}

Continuing with the circuit of a bottle as above, the same argument that the paths $\underline\L=(L_0,\dots,\L_n)$ and $\underline\R=(R_0,\dots,\R_n)$ are shortest, gives the following.

\begin{lemma}
If $v_j$ is incident to the domains $\L_i,\dots,\L_{i+k-1}$ or $\R_i,\dots,\R_{i+k-1}$  then $k\le n(v_j)+1$.
\end{lemma}

\begin{proof} If this is not the case, then  the part of path $\underline\L$ between $\L_i$ and $\L_{i+k-1}$ can be replaces by a  shorter one which goes   around the other side of $v_j$, and similarly for $\underline\R$.
\end{proof}

Let $l_j'$ and $r_j'$ be the number of domains in $\underline\L$ and $\underline\R$ respectively that are incident to~$v_j$ so that $l_j'\le n(v_j)+1$,  and define $l_j=l'_j-n(v_j)$, $r_j=r'_j-n(v_j)$. 

The next lemma describes the shape of the core of a bottle. As above, the core is a union of geodesic segments joining the  vertices $v_0,   \dots,  v_r$. Items (1) and  (2) together say that the segments turn through at most one sector (petal)  at each vertex $v_{j}$. Item (3)  implies  that successive bends alternate between turns to the right and turns to the left.   

\begin{lemma}\label{lem:bottle-prop}
Let $\T$ be the core of a bottle, that is, the chain of sides between vertices $v_0,\dots,v_r$.\\
1) We have $l_j,r_j\le 1$ for all $j=0,\dots,r$.\\
2) If $0<j<r$ then $l_j+r_j=0$, hence $(l_j,r_j)\in\{(-1,1),(0,0),(1,-1)\}$.\\
If $j=0$ or $j=r$, where $r>0$, then $l_j+r_j=1$, hence $(l_j,r_j)\in\{(0,1),(1,0)\}$.\\
If $j=0=r$ then $l_j+r_j=2$, hence $(l_j,r_j)=(1,1)$.\\
3) The sequences $(l_j)_{j=0}^r$ and $(r_j)_{j=0}^r$  cannot contain a segment of the form $1,0,\dots,0,1$ 
(with $r \ge 0$  zeroes between the two $1$'s).\\
4) The side $v_kv_{k+1}$ separates either $\L_s$ and $\R_{s+1}$ if the last $(l_j,r_j)\ne(0,0)$ with $j\le k$ has $r_j=1$,
or $\L_{s+1}$ and $\R_s$ if this $(l_j,r_j)$ has $l_j=1$.
\end{lemma}

\begin{proof}
1) This follows from the previous lemma.

2) All $2n(v)$ domains around $v$ are counted either in $l'_j$ or in $r'_j$. The only domains that are counted both in $l'_j$ and $r'_j$ are $\L_0=\R_0$ and $\L_n=\R_n$.
As we have seen above, these domains are incident only to $v_0$ and $v_r$ respectively.

%%%%%%%%%%%%%%%%%%%%
\begin{figure}[bt]
\centering
\includegraphics{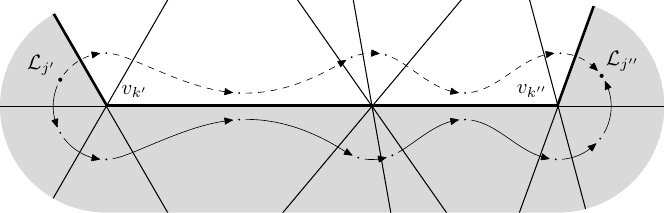}
\caption{The proof of Lemma~\ref{lem:bottle-prop}. The solid lines represent the subpath $(\L_{j'},\dots,\L_{j''})$ and the dashed ones indicate a shorter alternative.}
\label{fig:lem-bottle}
\end{figure}
%%%%%%%%%%%%%%%%

3) Assume that $l_{k'}=l_{k''}=1$, $l_i=0$ for $k'<i<k''$. Then the edges $v_{k'}v_{k'+1},\dots,\allowbreak v_{k''-1}v_{k''}$ form a geodesic segment $\g$. Let $\L_{j'}$ be the first domain in $\underline\L$ adjacent to $v_{k'}$ and $\L_{j''}$ be the last one adjacent to $v_{k''}$. The condition implies that 
the path from $\L_{j'}$ to $\L_{j''}$ crosses $\g$   both at $v_{k'}$ and $v_{k''}$. But this is impossible by Remark~\ref{rem:cross-once}, see Figure~\ref{fig:lem-bottle}.

4) We prove this by induction on $k$. The statement clearly holds for $k=0$.
Suppose the edge $v_{k-1}v_k$ separates $\L_s$ and $\R_{s'}$, while $v_kv_{k+1}$ separates $\L_t$ and $\R_{t'}$.
If $(l_k,r_k)=(0,0)$ then $t=s+n(v_k)-1$, $t'=s'+n(v_k)-1$, hence $t-t'=s-s'$, so the statement for $k-1$ implies it for $k$.

Similarly, if $(l_k,r_k)=(1,-1)$ then the previous nonzero $(l_j,r_j)$ should have $r_j=1$ by item 3, hence $s'=s+1$ by the induction assumption.
Also $t=s+n(v)$, $t'=s'+n(v)-2$, thus $t=t'+1$ and we have proved the statement for $k$.
\end{proof}

Besides describing the core of the bottle, Lemma~\ref{lem:bottle-prop} also allows us to describe how adjacent levels in the bottle are attached to one other.  
\begin{lemma}\label{lem:bottle-local}
Let $\underline\S=(\S_0,\S_1,\dots,\S_n)$ be a bottle.\\
1) Every level $\S_k$, $k=1,\dots,n-1$ contains two domains with a common vertex $v$.
The flower $\F_{v}$ is split by these two domains into two sectors, each containing an odd number of petals.
The sector bounded by the two radii of $\F_v$ which make up $\S_{k-1}\cap\S_{k}$ we call the ``past'' sector and that 
bounded by the  radii making up $\S_{k}\cap\S_{k+1}$ the ``future''  sector.\\
2a) If the future sector for $\S_k$ contains at least three petals, then $\S_{k+1}$ consists of the two petals in this sector adjacent to $\S_k$.\\
2b) If the future sector for $\S_k$ contains only one petal, there are the following possibilities: \\
\qquad (i) $k+1=n$ and $\S_n$ is the only domain in the future sector; or\\
\qquad (ii) let the boundary of the future sector be $v_Lvv_R$. Then $\S_{k+1}$ contains either the two domains from $\F_{v_L}$ adjacent to $\S_{k}$ (the ``left'' subcase),  or the two domains from $\F_{v_R}$ adjacent to $\S_{k}$ (the ``right'' subcase). \end{lemma}

\begin{proof}This is straightforward by induction on $k$. The situation in (2b) corresponds to the transition from the levels belonging to the flower $\F_{v_j}$ to the levels belonging to $\F_{v_{j+1}}$. The ``left'' (respectively, ``right'') subcase means that $v_jv_{j+1}$ separates $\L_k$ and $\R_{k+1}$,
(respectively, $\L_{k+1}$ and $\R_k$). Thus if both transitions $v_{j-1}\to v_j$ and $v_j\to v_{j+1}$ belong to the same (``left'' or ``right'') subcase, then the segment  $v_{j-1}v_jv_{j+1}$ of the core curve is straight by item~4 in~Lemma~\ref{lem:bottle-prop}. Similarly, if $v_{j-1}\to v_j$ is the ``left'' subcase and $v_j\to v_{j+1}$ is of ``right'' one, the core curve
$v_{j-1}v_jv_{j+1}$  bends to the right through one petal.
\end{proof}

\section{The Markov coding}
\label{sec:markov-coding}

As was explained in the introduction, states of our topological Markov chain should describe how the ``past'' level $\S_-=\S_k$ of a thickened path is attached to its ``future'' level $\S_+=\S_{k+1}$. More specifically, a state of the Markov chain should describe the arrangement of $\S_-$ and $\S_+$ up to the $G$-action. The set $\widehat \Xi$ of possible arrangements  is listed in Definition~\ref{defn:listingstates}.  However, to construct the actual states $ \Xi$ of our Markov chain we have to endow these arrangements with some additional data;  this is done in Subsection~\ref{subsec:states}.
In Subsection~\ref{subsec:transitionmatrix} we list the admissible transitions  between states,  thus defining the transition matrix $\Pi$. In Subsection~\ref{subsec:correspondence} we
prove the important result that there is a bijective correspondence between admissible sequences  and thickened paths.  
Finally, Subsection~\ref{subsec:involution} presents a time-reversing involution on the set of states.

\subsection{States of the Markov chain}
\label{subsec:states}

\subsubsection{Types of adjacency} As we have seen in Lemma~\ref{lem:bottle-local}, the adjacency graph for the domains in $\S_-\cup \S_+$ has one of the following types as illustrated in Figure~\ref{fig:configs} below:  
\begin{enumerate}[itemsep=0pt,label={$\Alph*$.}]
	\item $\#\S_-=1$, $\#\S_+=1$, and the graph contains the only possible edge from $\S_-$ to $\S_+$. This corresponds to a trivial bottle.
	\item $\#\S_-=1$, $\#\S_+=2$, and the graph contains both edges from $\S_-$ to $\S_+$. This state starts a bottle.
	\item $\#\S_-=2$, $\#\S_+=2$, and the edges join the left domain in $\S_-$ to the left domain in $\S_+$ and the right domain in $\S_-$ to the right domain in $\S_+$ (this is the case from item~2a of Lemma~\ref{lem:bottle-local}).
	\item  $\#\S_-=2$, $\#\S_+=1$, and the graph contains both edges from $\S_-$ to $\S_+$. This state ends a bottle (the case 2b (i) of the lemma).
	\item $\#\S_-=2$, $\#\S_+=2$, and the graph contains three edges, the two described for type $C$, and one more. This type is subdivided into the type $E_L$, where the third edge goes from the left domain in $\S_-$ to the right domain in $\S_+$, and the type $E_R$, where it goes from the right domain in $\S_-$ to the left domain in $\S_+$. The states of type $E$ correspond to the transitions from one flower to the next one inside a bottle (the ``left'' and ``right'' subcases in
case 2b (ii) of the lemma).
\end{enumerate}

\subsubsection{Labelling} \label{subsubsec:labelling}

The notation for each state of our Markov chain includes the type $A\dots E$ of the state from the list above and the label(s)  corresponding to generators  on the sides separating $\S_-$ and $\S_+$. More precisely, these separating sides form a polygonal curve, which is co-oriented from $\S_-$ to $\S_+$ and thus oriented from left to right when looking from $\S_-$ to $\S_+$. The notation for a state is found by  recording in order from left to right the labels on the $\S_+$-side of the separating sides, see Definition~\ref{defn:listingstates} below.

Importantly, the same orientation on this separating curve allows us to define ``left'' and ``right'' domains in each of $\S_\pm$, namely, $\L_\pm$ (respectively, $\R_\pm$) is the only domain in $\S_\pm$ that borders the leftmost (respectively, the rightmost) edge of the separating curve. As before, if $\S_\pm$ contains only one domain we have $\L_\pm=\R_\pm$.

%%%%%%%%%%%%%%%%%%%%
\begin{figure}[hbt]
\centering
\includegraphics{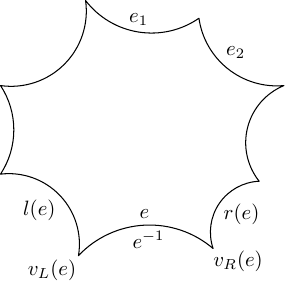}
\caption{The definitions of vertices $v_L(e)$, $v_R(e)$ and labels $l(e)$, $r(e)$. Labels $e_1,e_2\in G_0$ are adjacent.}
\label{fig:notation-l-r}
\end{figure}
%%%%%%%%%%%%%%%%

Clearly, the labels which appear in the notation of the state must satisfy some restrictions. To express these we introduce some notation regarding vertices, sides, and labels as shown in Figure~\ref{fig:notation-l-r}.
For any $e\in G_0$ consider the side $s_e$ of $\R$ so that its label inside $\R$ is $e$.
We co-orient this side from the outside to the inside of $\R$, and the corresponding orientation of $s_e$ allows us to define  its \emph{left vertex} $v_L(e)$ and the \emph{right vertex} $v_R(e)$. Note that $v_L(e)$ or $v_R(e)$ is undefined if the corresponding end of~$s_e$ lies on $\dd\DD$. The same notation $v_{L,R}(s)$ will be used for the ends of a co-oriented side~$s$ of the tessellation~$\TR$.

\begin{definition}\label{defn:adjacent} The labels $e_1$ and $e_2$ are called \emph{adjacent} if the sides of $\R$ with these \emph{outgoing} labels have a common vertex, i.~e.\ either $e_1=e_2$, or $v_L(e_1^{-1})=v_R(e_2^{-1})$, or vice versa, see Figure~\ref{fig:notation-l-r}.
\end{definition}
We now define maps $l$ and $r$ on the set of labels as shown in Figure~\ref{fig:notation-l-r}.
Informally speaking, we do the following: for $e\in G_0$ we go around $v_L(s_e)$ in the  counterclockwise direction, then the next side we cross after $s_e$ has the label $l(e)$ outside $\R$. Similarly, going clockwise around $v_R(s_e)$ we obtain $r(e)$.
Formally we define $l(e)$ and $r(e)$ as the labels such that $v_R(l(e)^{-1})=v_L(e)$, $v_L(r(e)^{-1})=v_R(e)$. Note that $l(e)$ or $r(e)$ is undefined if the corresponding end of $s_e$ lies on $\dd\DD$.

%%%%%%%%%%%%%%%%%%%%
\begin{figure}[tb]
\centering
\begin{tabular}{rcrc}
a)&\raisebox{5ex}{\raisebox{-\height}{\includegraphics{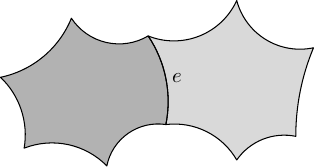}}}&\qquad b)&\raisebox{5ex}{\raisebox{-\height}{\includegraphics{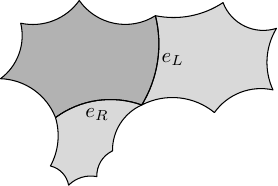}}}\\
c)&\raisebox{5ex}[10ex]{\raisebox{-\height}{\includegraphics{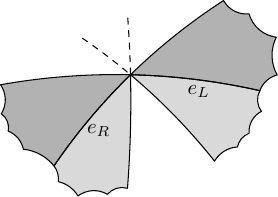}}}&\qquad d)&\raisebox{5ex}{\raisebox{-\height}{\includegraphics{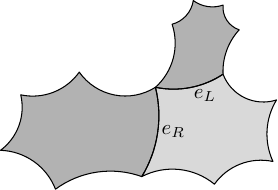}}}\\
\multicolumn{4}{c}{e)\hspace{\arraycolsep}\raisebox{5ex}[10ex]{\raisebox{-\height}{\includegraphics{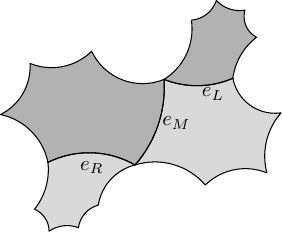}}}}
\end{tabular}
\caption{Configurations for states of the Markov coding:\protect\\ a)~$A(e)$, b)~$B(e_L,e_R)$, c)~$C_2(e_L,e_R)$, d)~$D(e_L,e_R)$, e)~$E_R(e_L,e_M,e_R)$.\protect\\
The domains in $\S_-$ and $\S_+$ are indicated respectively by the dark and the light shades of gray.}
\label{fig:configs}
\end{figure}
%%%%%%%%%%%%%%%%

\subsubsection{The possible arrangements $\widehat{\Xi}$} The following definition specifies the set $\widehat{\Xi}$ of  all possible arrangements  of $\S_-$ and $\S_+$ up to the action of $G$. Later, we will refine this in order to  list the actual states $\Xi$ of the Markov chain.

\begin{definition}\label{defn:listingstates}
The set $\widehat{\Xi}$ 
consists of the following elements (see Figure~\ref{fig:configs}):
\begin{itemize}
	\item $A(e)$: $\#\S_-=\#\S_+=1$, and $e$ is the label on the $\S_+$-side of the common side of $\S_-$ and $\S_+$.
	\item $B(e_L,e_R)$: $\#\S_-=1$, $\#\S_+=2$, $e_L$ and $e_R$ are $\S_+$-labels on the common sides of $\S_-$ with the left and the right domains in $\S_+$ respectively. Since these sides of $\S_-$ are adjacent, we have that $v_L(e_L^{-1})=v_R(e_R^{-1})$.
	\item $C_k(e_L,e_R)$: $\#\S_-=\#\S_+=2$, and all four domains in $\S_\pm$ share a common vertex~$v$. The label $e_L$ (respectively, $e_R$) is the  $\S_+$-label on the common side of the left (respectively, right) domains in $\S_-$ and $\S_+$, and the sector of the flower at~$v$ between these two sides that contains $\S_-$ consists of $2k+1$ petals.
	Denoting $n(e_L,e_R)=n(v)=n(v_R(e_L))=n(v_L(e_R))$, then $1\le k\le n(e_L,e_R)-2$ and we have $l^{2k+1}(e_L^{-1})=e_R$.
	\item $D(e_L,e_R)$: $\#\S_-=2$, $\#\S_+=1$, $e_L$ and $e_R$ are $\S_+$-labels on the common sides of the left and the right domain in $\S_-$ with the domain $\S_+$. The adjacency condition gives
	$v_R(e_L)=v_L(e_R)$.
	\item $E_{L,R}(e_L,e_M,e_R)$: $\#\S_-=\#\S_+=2$. The four domains in $\S_-$ and $\S_+$ do not have a common vertex, and there are three sides separating them. The state $E_L$ represents the case when these sides form an \textsf{N}-shaped line, that is, the left past domain borders both future domains via sides with the $\S_+$-labels $e_L$ and $e_M$, and the right past domain borders only the right future domain via the side with the label~$e_R$. Thus we have $v_L(e_L^{-1})=v_R(e_M^{-1})$ and $v_R(e_M)=v_L(e_R)$. The state $E_R$ is the same with left and right inverted: the boundary is \scalebox{-1}[1]{\textsf{N}}-shaped, and $v_R(e_L)=v_L(e_M)$, $v_L(e_M^{-1})=v_R(e_R^{-1})$.
\end{itemize}
\end{definition}

\subsubsection{Refining the arrangements.}

It is clear that every configuration of adjacent levels in a thickened path belongs to the set $\widehat{\Xi}$. On the other hand, the set of all possible sequences of configurations cannot be generated by a Markov chain. For example, for a vertex $v$ with $n(v)\ge 3$ it is allowed that $\S_i$, $\S_{i+1}$, $\S_{i+2}$ are consecutive petals around $v$, say, in the counterclockwise direction. Then if $e$ is the label on the future side of $\S_i\cap \S_{i+1}$, the label on the future side of $\S_{i+1}\cap \S_{i+2}$ is $l(e)$, and we have that the transition $A(e)\to A(l(e))$ is admissible.
On the other hand, a long sequence $A(e)\to A(l(e))\to A(l(l(e)))\to\dots$ is not admissible, since the respective sets $\S_i$ are still the consecutive petals around $v$, and a thickened path cannot have $v$ on its boundary and contains more than $n(v)$ petals around~$v$.

To solve this problem we endow the states of type $A$ with some additional information  based on the following statement.

\begin{prop}\label{prop:conv-boundary}Let $\underline\S$ be a thickened path. Suppose a vertex $v\in\partial\underline\S$ belongs to the boundary of $\S_k$ for $k=i,\dots,j+1$, where $j>i$. Then either
\begin{enumerate}[itemsep=0pt]
	\item $j=i+2$ and both pairs $(\S_i,\S_{i+1})$, $(\S_{i+1},\S_{i+2})$ represent $E$-states (see Figure~\ref{fig:EE-case}) or
	\item for all $k=i+1,\dots,j-1$ the pair $(\S_k,\S_{k+1})$ represents a state of type $A$, for $k=i$ it represents a state of type $A$ or $D$, and for $k=j$ it represents a state of type $A$ or $B$, depending on whether $\S_i$ or $\S_{j+1}$ respectively contains two domains. Moreover $j-i \leq n(v)$.
\end{enumerate}
\end{prop}

\begin{proof} Suppose first that a vertex $v\in\DD$ belongs to three consecutive levels $\S_k$, $\S_{k+1}$, $\S_{k+2}$ of the thickened path, and $\#\S_{k+1}=2$. If, say, $v$ belongs to $\dd_L\underline{\S}$, then $\dd_L\S_{k+1}$ consists of the  vertex $v$ only. Then $\R$ must be compact and hence $N(\R)\ge 4$.

The two domains in $\S_{k+1}$ must be  petals of $\F_u$ for some vertex $u \neq v$. 
The left domain $\L_{k+1}$ in the level $\S_{k+1}$ must meet  $\S_{k+2}$ along at least two sides, one emanating from $v$ and one from $u$. Moreover these sides, being common sides of one domain in two levels, must themselves be adjacent, hence meet in a vertex $w$ say. By the same argument,  $\L_{k+1}$ meets $\S_{k}$ along two adjacent sides which meet in a vertex $w'$ say. 
Hence $N(\R)=4$, and each of $\L_{k+1}\cap \S_{k}$ and $\L_{k+1}\cap \S_{k+2}$ contains two sides, see Figure~\ref{fig:EE-case}. We deduce that  the states representing the pairs $(\S_k,\S_{k+1})$, $(\S_{k+1},\S_{k+2})$ are of types $E_R$ and $E_L$ respectively. In particular, this means that $v$ cannot belong to four consecutive levels of the thickened path, and we see that were are in case (1)  of the proposition.   

%%%%%%%%%%%%%%%%%%%%
\begin{figure}[hbt]
\centering
\includegraphics{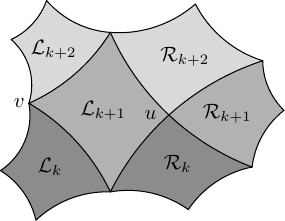}
\caption{Two consecutive $E$ states, see Proposition~\ref{prop:conv-boundary}.}
\label{fig:EE-case}
\end{figure}
%%%%%%%%%%%%%%%%

It remains to consider the case when $\#\S_k=1$ for all $k=i+1,\dots,j-1$. If $\S_i $ contains two domains and $\S_{i+1}$ contains one, then $(\S_i , \S_{i+1})$  must be of type $D$; otherwise $\S_i$  contains one domain and   $(\S_i , \S_{i+1})$  must be of type $A$, with a similar argument for the transition $(\S_j , \S_{j+1})$, which implies case (2)  of the proposition.   
\end{proof}

\begin{remark}\label{rem:EE-case}Assumption~\ref{asm:R} yields that in the first case in this proposition we have $n(v)\ge 3$. Indeed, $N(\R)=4$, $\R$ is compact, and for the common vertex $u$ of the two domains in $\S_{k+1}$ we have $n(u)=2$, see Figure~\ref{fig:EE-case}.  
\end{remark}  

Let $(\S_k,\S_{k+1})$ form a configuration $A(e)$ and $s_k$ be the common side of $\S_k$ and $\S_{k+1}$.
Then one can define four numbers $i_{\pm,L}$ and $i_{\pm,R}$ as follows:
$i_{-,\alpha}$ (resp., $i_{+,\alpha}$), $\alpha\in\{L,R\}$, is the number of $m\le k$ (resp., $m\ge k+1$) such that $\S_m$ contains $v_\alpha(s_k)$. If the vertex $v_\alpha(s_k)$ is not defined, we set $i_{\pm,\alpha}=1$.

Note that it is not possible to have $i_{-,L}>1$ and $i_{-,R}>1$ simultaneously: these conditions mean that both $\partial_L\S_k$ and $\partial_R\S_k$ each consist of a vertex only, say $v$ and $v'$. Indeed  since the sides  of $\S_k$ adjacent to $v$ and $v'$ must both be in common with $\S_{k+1}$, and since by assumption the transition is of type $A$,  this would mean that $vv'$ is the common side $s$ of $\S_k, \S_{k+1}$. But the sides $t, t'$ of $\S_k$ adjacent to $s$  must also both be adjacent to $\S_{k-1}$. Now $\ell(t), \ell(t')$ cannot meet, otherwise we have a triangle in $\TR$, so they separate $\S_{k-1}$ into two disconnected components which is impossible. The same argument applies  to $i_{+,L}$ and $i_{+,R}$.

The convexity of $\underline\S$ at $v_\alpha(s_k)$, $\alpha=L,R$, implies that $i_{-,\alpha}+i_{+,\alpha}\le n(v_\alpha(s_k))$. It follows that the configuration $A(e)$ can be subdivided as follows, see Figure~\ref{fig:type-A}:

%%%%%%%%%%%%%%%%%%%%
\begin{figure}[tb]
\centering
\begin{tabular}{lll}
a)\enskip $A_0(e)$ &b)\enskip $A_L[2,1](e)$&c)\enskip $A_{LR}[3,2](e)$\\[4pt]%
\includegraphics{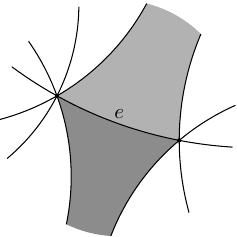}\hspace*{1em}&%
\includegraphics{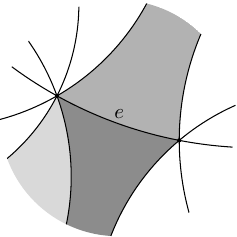}\hspace*{1em}&%
\includegraphics{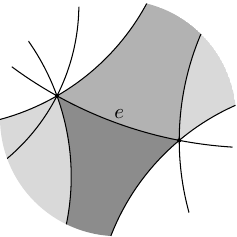}\\
\end{tabular}\\[\bigskipamount]
\begin{tabular}{cc}
\multicolumn{1}{b{5cm}}{\raggedright d)\enskip Impossible ``$A_R[2,2](e)$'':\\ $i_-+i_+>n(v_R(e))$, thus\\
convexity in $v_R(e)$ fails}&%
\multicolumn{1}{b{5cm}}{\raggedright e)\enskip Impossible subtype:\\ both $i_{+,L}$ and $i_{+,R}$\\ are greater than~$1$}\\[4pt]%
\includegraphics{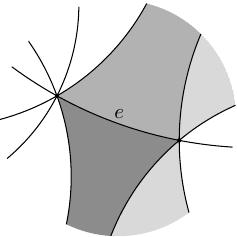}\qquad&%
\includegraphics{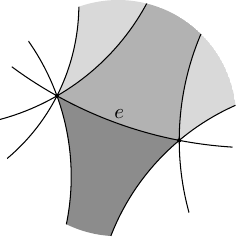}\\%
\end{tabular}
\caption{Possible (a--c) and impossible (d--e) subtypes for type $A$ states.
The dark and medium gray domains are respectively the past and the future domain for the current state, light gray domains are other domains from the thickened path. In the figure, $n(v_L(e))=4$, $n(v_R(e))=3$.}
\label{fig:type-A}
\end{figure}
%%%%%%%%%%%%%%%%

\begin{itemize}
\item $A_0(e)$: all four $i_{\pm,L/R}$ equal one.
\item $A_L[i_-,i_+](e)$: here $i_{-,L}=i_-$, $i_{+,L}=i_+$, $i_{-,R}=i_{+,R}=1$, and the indices $i_\pm$ should satisfy $3\le i_-+i_+\le n(v_L(e))$.
\item $A_R[i_-,i_+](e)$: symmetric to the previous case; here $3\le i_-+i_+\le n(v_R(e))$.
\item $A_{LR}[i_-,i_+](e)$: here $i_{-,L}=i_-$, $i_{+,R}=i_+$, and $i_{+,L}=i_{-,R}=1$. The conditions on the indices $i_\pm$ are $2\le i_-\le n(v_L(e))-1$, $2\le i_+\le n(v_R(e))-1$.
\item $A_{RL}[i_-,i_+](e)$: symmetric to the previous case; here $2\le i_-\le n(v_R(e))-1$, $2\le i_+\le n(v_L(e))-1$.
\end{itemize}

\begin{remark}\label{rem:specialcase}
If $N(\R)=3$ and $\R$ has a compact side (as is the case for example for the classical fundamental domain for the group $SL(2,\ZZ)$), some of these states may be absent. Namely, let $s$ be the only compact side of $\R$ and $g$ be its label outside of~$\R$.
If $(\S_k,\S_{k+1})$ has the form $A(g)$, then $\S_{k+2}$ must contain at least one of the domains adjacent to the sides of $\S_{k+1}$, hence either $i_{+,L}$ or $i_{+,R}$ is greater than one so there are no states $A_0(g)$.   Similarly, either $i_{-,L}>1$ or $i_{-,R}>1$. This case needs special consideration in several statements below, and we usually refer to it as ``the special case from Remark~\ref{rem:specialcase}''.

Note that even in this case the list of $A_{\dots}(g)$-states is not completely empty. Indeed, since $s$ is the only compact side, it must be paired to itself: $g=g^{-1}$. Therefore, the ends of $s$ are swapped by the action of $g$, hence $n(v_L(g))=n(v_R(g))=n$. Let $\alpha$ and $\beta$ be the angles of $\R$ at the ends of $s$. Consider the flower around a vertex $v\in\DD$. Note that the sides incident to $v$ are alternately compact and non-compact, and the angles between these sides are alternately $\alpha$ and $\beta$. Therefore, $n\alpha+n\beta=2\pi$. On the other hand, the sum of angles in the hyperbolic triangle $\R$ is $\alpha+\beta<\pi$. Consequently, $n\ge 3$, and for example, the state $A_{LR}[2,2](g)$ is allowed.
\end{remark}

\subsubsection{The states of the Markov chain.} Finally we are able to list the states ${\Xi}$.
\begin{definition}
The \emph{set of states} $\Xi$ of our Markov chain is the set of all states of types $B,C,D,E$ from the set $\widehat{\Xi}$ and of all subtypes of type $A$ states   enumerated in the previous list. We denote the projection from $\Xi$ to $\widehat{\Xi}$ by $\pi$.
\end{definition}

Finally, let us define sets $\Xi_S, \Xi_F\subset\Xi$ as follows:
\begin{align*}
\Xi_S&{}=\{A_0(e), A_L[1,i_+](e), A_R[1,i_+](e), B(e_L,e_R)\},\\
\Xi_F&{}=\{A_0(e), A_L[i_-,1](e), A_R[i_-,1](e), D(e_L,e_R)\},
\end{align*}
where the parameters $i_{\pm}$, $e$, $e_L$, $e_R$ admit all possible values. In the special case from Remark~\ref{rem:specialcase} these definitions are amended as follows: if $g=g^{-1}$ is the label on the compact side of $\R$, we include $A_{LR}[2,i_+](g), A_{RL}[2,i_+](g)$ in $\Xi_S$
and $A_{LR}[i_-,2](g), A_{RL}[i_-,2](g)$ in $\Xi_F$ in place of the $A$ states listed above.

\subsection{The admissible transitions}\label{subsec:transitionmatrix}

\begin{definition}\label{def:transitions}
The set of \emph{admissible transitions} in our Markov coding is enumerated in the following list. We denote by $\Pi$ the $\Xi\times\Xi$ adjacency matrix for the corresponding topological Markov chain and write $j\to j'$ if the transition from $j$ to $j'$ is admissible according  to this list.
(Recall that the adjacency of labels was defined near the beginning of this section, see Figure~\ref{fig:notation-l-r}.)
\begin{itemize}
	\item $A_0(e)\to
	\begin{cases}
	A_0(e')&\text{for $e'$ non-adjacent to $e^{-1}$},\\
	A_L[1,i_+](e')&\text{for $e'$ non-adjacent to $e^{-1}$, any admissible $i_+$},\\
	A_R[1,i_+](e')&\text{for $e'$ non-adjacent to $e^{-1}$, any admissible $i_+$},\\
	B(e_L,e_R)&\text{for any $e_L,e_R$ non-adjacent to $e^{-1}$}.
	\end{cases}$
	\item If $i_+>1$ then \\
	$A_L[i_-,i_+](e)\to
	\begin{cases}
	A_L[i_-+1,i_+-1](l(e)),&\\
	A_{LR}[i_-+1,j_+](l(e)),&\llap{$\biggl|\,$}\parbox{0.3\textwidth}{if $i_+=2$,\newline for any admissible $j_+$,}\\
	B(l(e),r(l(e))^{-1}),&\text{if $i_+=2$}.
	\end{cases}$
	\item $A_L[i_-,1](e)\to (\text{the same cases as for $A_0(e)$})$.
	\item $A_{RL}[i_-,i_+](e)\to (\text{the same cases as for $A_L[1,i_+](e)$})$.
	\item The transitions for the $A_R$- and $A_{LR}$-states are similar with the exchange of left and right.
	\item $B(e_L,e_R)\to C_1(r(e_L),l(e_R))$ if $n(e_L,e_R)\ge 3$,\\
	if $n(e_L,e_R)=2$ the transitions for $B(e_L,e_R)$ are the same as for $C_{n(e_L,e_R)-2}(e_L,e_R)$ below.
	\item $C_i(e_L,e_R)\to C_{i+1}(r(e_L),l(e_R))$, for $i<n(e_L,e_R)-2$,
	\item $C_{n(e_L,e_R)-2}(e_L,e_R)\to
	\begin{cases}
	D(r(e_L),l(e_R)),&\\
	E_L(r(r(e_L)^{-1}),r(e_L),l(e_R)),&\\
	E_R(r(e_L),l(e_R),l(l(e_R)^{-1})).&
	\end{cases}$
	\item
	$D(e_L,e_R)\to\begin{cases}
	A_0(e')&\text{for $e'$ non-adjacent to $e_L^{-1},e_R^{-1}$},\\
	 \begingroup\arraycolsep=0pt\begin{array}{l}A_L[1,i_+](e')\\A_R[1,i_+](e')\end{array}\Biggr\}\endgroup&\llap{$\biggl|\,$}\parbox{0.5\textwidth}{for $e'$ non-adjacent to $e_L^{-1},e_R^{-1}$,\newline any admissible $i_+$,}\\
	B(e_L',e_R')&\llap{$\left|\vbox to 24pt{}\right.\mathsurround=0pt\nulldelimiterspace=0pt\,$}\parbox{0.5\textwidth}{for $e_L',e_R'$ either not adjacent to $e_L^{-1},e_R^{-1}$,\newline or adjacent via a vertex $v$ with $n(v)>2$\newline or adjacent via two such vertices,}\\
    \begingroup\arraycolsep=0pt\begin{array}{l}A_L[2,i_+](l(e_L))\\A_{LR}[2,i_+](l(e_L))\end{array}\Biggr\}\endgroup&\llap{$\biggl|\,$}\parbox{0.5\textwidth}{if $n(v_L(e_L))>2$,\newline for any admissible $i_+$,}\\
    \begingroup\arraycolsep=0pt\begin{array}{l}A_R[2,i_+](r(e_R))\\A_{RL}[2,i_+](r(e_R))\end{array}\Biggr\}\endgroup&\llap{$\biggl|\,$}\parbox{0.5\textwidth}{if $n(v_R(e_R))>2$,\newline for any admissible $i_+$.}\\
	\end{cases}$
	\item $E_L(e_L,e_M,e_R)$ has the same set of transitions as $B(e_L,e_M)$.
	\item $E_R(e_L,e_M,e_R)$ has the same set of transitions as $B(e_M,e_R)$.
\end{itemize}
\end{definition}

\subsection{Correspondence with thickened paths}\label{subsec:correspondence}

We now come to the important result that admissible sequences of  states do indeed  correspond to thickened paths.  
Precisely, define
\begin{multline}\label{eq:PathsSF} \Paths_{N-1}^{S\to F}
= \\ \{(j_0,\dots,j_{N-1})\subset\Xi^N: j_0\in\Xi_S, j_{N-1}\in\Xi_F,  \Pi_{j_k,j_{k+1}}=1\text{ for }k=0,\dots,N-2\}.
\end{multline}

We will show that $\Paths_{N-1}^{S\to F}$ is in 1:1-correspondence with the set of the thickened paths of length $N$.

\begin{definition} \label{defn:generates}Say that a sequence $\underline j$ of states \emph{generates} a sequence of domains $\underline\S$ if for each $k$ the pair $(\S_k,\S_{k+1})$ represents the configuration $\pi(j_k)$. 
\end{definition} 
\begin{thm}\label{thm:th-path-vs-sphere}
Let $\underline\S=(\S_0,\dots,\S_N)$ be a thickened path starting at $\R$. Then there exists a unique sequence of states $\underline j\in\Paths_{N-1}^{S\to F}$ which generates the sequence $\underline\S$.
Moreover, this mapping of thickened paths of length $N$ starting in $\R$ to the set $\Paths_{N-1}^{S\to F}$ is a bijection.
\end{thm}

\begin{proof}
The proof is split into two parts. First, we show that for every thickened path~$\underline\S$ there exists a unique sequence $\underline j\in\Paths_{N-1}^{S\to F}$ that generates $\underline\S$. Second, we prove that the sequence $\underline\S$ of domains generated by any sequence $\underline j\in\Paths_{N-1}^{S\to F}$ is the thickened path between its ends.

\medskip

\noindent\textbf{Part 1.} Every thickened path $\underline\S=(\S_0,\dots,\S_N)$ can be generated by a unique sequence $\underline j\in\Paths_{N-1}^{S\to F}$.

\smallskip

\noindent\textit{Step 1.} By hypothesis, each pair $(\S_k,\S_{k+1})$ in a thickened path $\underline\S$ represents a unique configuration $\hat{\jmath}_k\in\widehat{\Xi}$. Further, for every configuration of type $A$ one can recover the  indices $i_{\pm,L/R}$ as described above, thus arriving at the states $j_k$ with $\pi(j_k)=\hat{\jmath}_k$. Note that if $\hat{\jmath}_0=A(e)$ then the state $j_0$ has $i_{-,L}=i_{+,R}=1$, so $j_0\in\Xi_S$.
In the special case from Remark~\ref{rem:specialcase} we need to amend these indices as follows: if $\pi(j_0)=A(g)$, where $g$ is the label on the compact side, then either $i_{+,L}\ge 2$ or $i_{+,R}\ge 2$. In the former case we then set $i_{-,L}=1$, $i_{-,R}=2$ and in the latter we set $i_{-,L}=2$, $i_{-,R}=1$; this corresponds to the addition of the ``virtual domain'' $\S_{-1}$ to our thickened path. Note that the state $A_{RL}[2,i_+](g)$ has the same set of allowed transitions as the non-existent ``state $A_L[1,i_+](g)$''.

Now we have to check that all transitions $j_k\to j_{k+1}$ are admissible. There are three types of restrictions on the pair of states $(j_k,j_{k+1})$ in the list of Definition~\ref{def:transitions}.

First, there are restrictions on the configurations $\hat{\jmath}_k,\hat{\jmath}_{k+1}$, $\pi(j_k) = \hat {\jmath}_k$. For example, if $\hat{\jmath}_k=C_l(e_L,e_R)$, then
$\S_+=\S_{k+1}$ is a pair of petals meeting at a vertex $v$ with $2n(v)-2l-1$ petals in the ``future'' sector in $\F_v$. Therefore, if $l<n(v)-2$ by Lemma~\ref{lem:bottle-local} we see that $\S_{++}=\S_{k+2}$ is the pair of petals in the ``future'' sector adjacent to $\S_+$, hence $\hat{\jmath}_{k+1}=C_{l+1}(e'_L,e'_R)$ with $e'_L=r(e_L)$, $e'_R=l(e_R)$.

Second, there are restrictions on the indices $i_\pm$ of $A$-states.
If $j_k$ is an $A$-state with $i_{+,L}>1$ then $\S_{k+2},\dots, \S_{k+i_{+,L}+1}$ should contain the consecutive petals around the vertex $v_L(s_k)$ going from $\S_{k+1}$ in the counterclockwise direction.
Hence by Proposition~\ref{prop:conv-boundary}
the sequence $(j_{k+1},\dots,j_{k+i_{+,L}})$ is either of type $(A,\dots,A)$ or
$(A,\dots,A,B)$. Therefore, if $i_{+,L}>2$ then $\hat{\jmath}_{k+1}=A(l(e))$,
and for its indices $i'_{\pm,L}$ we have $i'_{-,L}=i_{-,L}+1$,
$i'_{+,L}=i_{+,L}-1$. Similarly, if $i_{+,L}=2$ we have either $\hat{\jmath}_{k+1}=B(l(e),r(l(e))^{-1})$ or $\hat{\jmath}_{k+1}=A(l(e))$ with the same relations for $i'_{\pm,L}$. In the latter case $i'_{+,L}=1$ so we have two subcases: either $i'_{+,R}=1$ or  $i'_{+,R}>1$, which correspond in Definition~\ref{def:transitions} to the transitions to $A_L$- and $A_{LR}$-states respectively.

Finally, there are restrictions related to the convexity of $\dd\underline\S$ (see Proposition~\ref{prop:NonSep}) which we need to check for the boundary vertices~$v$ that are incident to at least three levels in~$\underline{\S}$. These cases are enumerated in Proposition~\ref{prop:conv-boundary}. In the cases when the corresponding sequence of states contains $A$-states, the convexity is guaranteed by the inequalities on the indices $i_\pm$ for these states, so we need to consider only the cases when $(j_k,j_{k+1})$ have types $(E_L, E_R)$, $(E_R,E_L)$, and $(D,B)$. In the first two cases the convexity at $v$ holds, see Remark~\ref{rem:EE-case}.
The remaining case $(D,B)$ is specially mentioned in Definition~\ref{def:transitions}: if $v$ is a common vertex of $\S_{k-1}\cap\S_{k}$ and $\S_{k}\cap\S_{k+1}$, we require that $n(v)>2$.

\smallskip

\noindent\textit{Step 2.}  It remains to prove that the above-constructed sequence~$\underline j$ is the only one in $\Paths_{N-1}^{S\to F}$ that generates $\underline\S$. Namely, we have to check that the indices $i_{\pm,L/R}$ for $A$-states cannot be chosen in a different way. One can see that the ``past'' indices $i_{-,L/R}$ for the state $j_k$ are uniquely defined by the configurations $\pi(j_{k-1}),\pi(j_k)$ and by the past indices for the state $j_{k-1}$ (assuming $j_{k-1}$ has type $A$). The past indices for $j_0\in\Xi_S$ are $i_{-,L/R}(j_0)=1$, hence one can successively find these indices for all successive states $j_1,\dots,j_{N-1}$.
Similarly, the ``future'' indices $i_{+,L/R}(j_k)$ are successively found starting from the end of the sequence: $i_{+,L/R}(j_{N-1})=1$.

The special case from Remark~\ref{rem:specialcase} again needs separate consideration if $\hat{\jmath}_0=A(g)$. Here $\hat\jmath_1=A(e)$, where $e$ is a label on a non-compact side of $\R$, and either $e=l(g)$ or $e=r(g)$.
In the former case, say,  this yields $i_{+,L}(j_0)\ge 2$, thus $j_0\in\Xi_S$ implies
$j_0=A_{RL}[2,i_+](g)$ with some $i_+$. Therefore we find the past indices for $j_0$ and can now proceed as in the general case.

\medskip
 
\noindent\textbf{Part 2.} Every sequence $\underline{j}\in\Paths_{N-1}^{S\to F}$ generates a thickened path. 

\smallskip

\noindent\textit{Step 1.} We begin by constructing the $\S_k$'s inductively, starting with $\S_0=\R$. To define $\S_{k+1}$ we take a configuration $(\S_-,\S_+)$ representing $\pi(j_k)$,
choose $h$ such that $h\S_-=\S_k$ and define $\S_{k+1}=h\S_+$. The choice of such an $h$ is possible for $k=0$ since all states from $\Xi_S$ have only one ``past'' domain,
and for $k\ge 1$ this is possible since if $j\to j'$ is admissible, and $(\S_-,\S_+)$, $(\S_-',\S_+')$ are configurations
representing $\pi(j)$ and $\pi(j')$ respectively, then $\S_+$ and $\S_-'$ can be translated to each other by an element of $G$: $\S_-'=h\S_+$.
Moreover, as  described in~\ref{subsubsec:labelling}, one can define  left  domains $\L_\pm$ and  right domains $\R_\pm$ in $\S_\pm$, as well as
$\L_\pm',\R_\pm'$ in $\S'_\pm$, and these definitions agree: $\L_-'=h\L_+$, $\R_-'=h\R_+$. Thus we have defined all levels $\S_k$ and the domains $\L_k,\R_k\subset\S_k$ for all $k=0,\dots, N$. Moreover $\S_N$ contains only one domain since $j_{N-1}\in\Xi_F$.

We need to show that $\G = \bigcup\underline\S=\bigcup_{k=0}^N \S_k$ is the thickened path from $\S_0$ to $\S_N$ and that the $\S_k$'s are its levels.
 This is obtained from the following statement, which will be established below.
 
\begin{claim}\label{clm:GenerThP}
1.~The intersection $\S_l\cap \S_m$ with $l\ne m$ contains no fundamental domains, and contains sides only if $|l-m|=1$.\\
2.~The set $\bigcup\underline\S$ is convex.
\end{claim}

To prove the result given Claim~\ref{clm:GenerThP}, let $\underline\T$ be the actual thickened path from~$\S_0$ to~$\S_N$.
The second item in the claim together with Proposition~\ref{prop:NonSep} gives $\bigcup\underline\T\subset \bigcup\underline\S$.
Now consider some shortest path $\underline\A=(\A_0=\S_0,\A_1,\dots,\A_M=\S_N)$ from $\S_0$ to $\S_N$. Then $\A_m\subset\bigcup\underline\T\subset \bigcup\underline\S$. Define $d_m$ by $\A_m\subset\S_{d_m}$ so that in particular $d_M = N$. The first item of the claim  implies that $|d_m-d_{m-1}|=1$ hence $d_M\le M$. On the other hand, the two length $N$ paths $\underline \L = (\L_0,\dots,\L_N)$ and $\underline \R = (\R_0,\dots,\R_N)$
connect  $\S_0$ and $\S_N$ and hence $M \leq N$.
 Therefore  $M=N$ and  $\underline \L, \underline \R$ are shortest paths connecting $\S_0$ and $\S_N$. Thus $\underline \L, \underline \R$  are contained  in $\bigcup\underline\T$.
But $\bigcup\underline\S = \underline \L \cup  \underline \R$ and  hence $\bigcup\underline\S\subset\bigcup\underline\T$.

\smallskip

\noindent\textit{Step 2.} To verify Claim~\ref{clm:GenerThP}, we will construct a nested sequence of convex regions $  \D_k \supset \D_{k+1}  $ whose intersection is $\underline \S$, with some further properties listed in Claim~\ref{clm:FkThPath}.

We begin with  some notation. From here to the end of the proof we consider  all domains, curves, vertices, etc. to lie in the closure $\overline{\DD}$ of the hyperbolic plane.

 Let $s_k=\S_k\cap\S_{k+1}$ be the curve separating $\S_k$ and $\S_{k+1}$. This curve is the union of no more than three sides of domains in $\TR$, and, as one can see from the adjacency matrix, the curves $s_{k-1}$ and $s_k$ contain no common sides of $\dd\S_k$. Moreover, $\dd\S_k\setminus(s_{k-1}\cup s_k)$ is a union of two curves, one joining the left ends of $s_{k-1}$ and $s_k$, and the other joining the  right ends; it is possible that one or both of these curves consists of a single vertex and no sides, see the discussion following Remark~\ref{rem:EE-case}. We denote these curves by $\dd_L\S_k$ and $\dd_R\S_k$ respectively. For $k=0, N$ there is only one $s_j$ to remove, so we define $\dd_O\S_0=\dd\S_0\setminus s_0$,
$\dd_O\S_N=\dd\S_N\setminus s_{N-1}$. For $k=1,\dots,N-1$ we also denote $\dd_O\S_k=\dd_L\S_k\cup\dd_R\S_k$.  

Next, let us orient the curves $\dd_{\dots}\S_k$ in such a way that $\S_k$ lies locally to the left of the curve when moving in positive direction.
Then all these curves can be joined into one closed oriented curve in the following order, so that the end of each curve coincides with the beginning of the next:
\begin{multline*}
\dots,\dd_O\S_N,\dd_L\S_{N-1},\dots,\dd_L\S_{k+1},
\underbrace{\dd_L\S_k,\dots,\dd_L\S_1,\dd_O\S_0,\dd_R\S_1,\dots,\dd_R\S_k}_{\dd_O\S_0^k},\\
\dd_R\S_{k+1},\dots,\dd_R\S_{N-1},\dd_O\S_N,\dots
\end{multline*}
Denote the part of this curve bracketed in the formula by $\dd_O\S_0^k$.

\smallskip

\noindent\textit{Step 3.} Since our object is to show that the region $\underline \S$  is convex, we need to check that no more than $n(v)-2$ consecutive boundary curves $\dd_\alpha\S_j$  consist of only one (and hence the same) vertex and no sides.
 
Let $j_k$ be an $A$-state with indices $i_{\pm,L/R}$.
Denote by $d_{-,\alpha}$  ($\alpha=L,R$) the maximal $d$ such that each of $\dd_\alpha\S_{k-d+1},\dots,\dd_\alpha\S_k$ contain only one (and hence the same) vertex,
and by $d_{+,\alpha}$ the maximal $d$ such that each of $\dd_\alpha\S_{k+1},\dots,\dd_\alpha\S_{k+d}$ contains only one vertex.
Then one can check from the adjacency matrix that $d_{\pm,\alpha}=i_{\pm,\alpha}-1$.

Now from the adjacency matrix we obtain the following converse of Proposition~\ref{prop:conv-boundary}: if $\dd_\alpha\S_k$ contains no sides,
then the pair of states $(j_{k-1},j_k)$ can be of the following types: $(A,A)$, $(D,A)$, $(A,B)$, $(D,B)$, $(E_R,E_L)$.
Moreover in all  cases the transition rules ensure that if
\begin{equation*}
v=\dd_\alpha\S_k=\dd_\alpha\S_{k+1}=\dots=\dd_\alpha\S_{k+m-1},\quad \alpha\in\{L,R\},
\end{equation*}
then $n(v)-2\ge m$.
 
\smallskip

\noindent\textit{Step 4.} We now construct the sets $\D_k$ referred to above.
 Let $\mathcal H_k$ be the collection of all half-planes $H$ such that $\dd H$ contains a side of $\dd_O\S_k$ and $H$ contains $\S_k$.
Denote
\begin{equation*}
\mathcal H_0^k=\bigcup_{j=0}^k \mathcal H_j, \qquad
\D_k=\bigcap_{H\in\mathcal H^0_k} H.
\end{equation*}
By construction, $\D_k$ is a union of fundamental domains and is convex, moreover clearly $\D_{k}\subset\D_{k-1}$. In more detail:
\begin{claim}\label{clm:FkThPath}
For $k=0,\dots, N-1$ we have the following (see Figure~\ref{fig:markov-to-thpath}):\\
(i) The curve $s_k$ lies in the interior of $\D_k$, joins two points on its boundary, and divides $\D_k$ into two parts.\\
(ii) One of these parts is the union of all $\S_j$ with $j=0,\dots, k$. We denote this part by $\D^-_k$ and the other part by $\D^+_k$.\\
(iii) $\dd\D^-_k=s_k\cup\dd_O\S_0^k$, while $\dd\D^+_k$ consists of $s_k$ and two rays $\gamma_{k,L}$, $\gamma_{k,R}$ that are continuations of the first and the last sides in $\dd_O\S_0^k$ beyond the ends of $s_k$. These rays do not intersect inside $\DD$.\\
(iv) $\D_{k}\subset\D_{k-1}$, or, more precisely, $\S_{k}\cup \D^+_{k}\subset\D^+_{k-1}$.\\
Finally, for $k=N$ we have $\D_N=\bigcup\underline{\S}$.
\end{claim}

To deduce Claim~\ref{clm:GenerThP} from this statement note that if $l=m+b$, $b>0$, then
\begin{equation*}
\S_{m+b}\subset\D^+_{m+b-1}\subset \D^+_{m+b-2}\subset \dots\subset \D^+_m,
\end{equation*}
hence $\S_l$ has no domain in common with $\S_m\subset\D^-_m$. Moreover, if $\S_m\cap\S_{m+b}$ contains a common side $s$, then
$s\subset \D^+_m\cap\D^-_m=s_m$, and by construction one of the two domains adjacent to $s_m$ belongs to $\S_m$ and the other does to $\S_{m+1}$.
This proves the first statement in Claim~\ref{clm:GenerThP}. The second is immediate from $\D_N=\bigcup\underline{\S}$.

\smallskip

\noindent\textit{Step 5.} Finally, let us verify Claim~\ref{clm:FkThPath} by induction on $k$.
The base, $k=0$, is clear.

Let us assume that this claim holds for some $k$ and check it for $k+1$.
Since $s_k$ lies in the interior of $\D_k$, the points of $\S_{k+1}$ that are close to $s_k$ lie inside $\D_k^+$, and since $\D_k^+$ is the union of fundamental domains, $\S_{k+1}\subset \D^+_k$.

To construct $\D_{k+1}$ we need to add to the intersection defining $\D_k$ the half-planes $H\in\H_{k+1}$. Assume first that $k+1<N$.

If $\dd_\alpha\S_{k+1}$ for  $\alpha=L,R$ consists of a single vertex, there are no new half-planes to be added on $\dd_\alpha\S_{k+1}$.
Otherwise,  $\dd_\alpha\S_{k+1} $ consists of sides joining a sequence of vertices $u_0^\alpha u_1^\alpha\dots u_{m_{\alpha}}^\alpha$. Let $H^\alpha_j$ be the half-plane in $\H_{k+1}$ with $\dd H^\alpha_j= \ell(u_{j-1}^{\alpha}u_{j}^\alpha)$ and let $\beta^\alpha_j\subset\ell(u_{j-1}^{\alpha}u_{j}^\alpha)$ be the ray starting from $u_{j-1}^{\alpha}$ and passing through $u_{j}^\alpha$,
see Figure~\ref{fig:markov-to-thpath}.

%%%%%%%%%%%%%%%%%%%%
\begin{figure}[ht]
\centering
\includegraphics{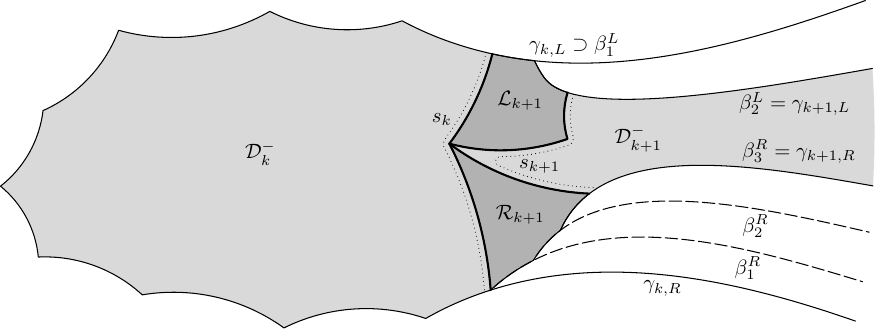}
\caption{Illustrating the proof of Claim~\ref{clm:FkThPath}.}
\label{fig:markov-to-thpath}
\end{figure}
%%%%%%%%%%%%%%%%

Add the half-planes $H^\alpha_j$ to the intersection one by one with $j$ increasing. When $H^\alpha_j$ is added,  we cut the intersection $\D_k$ along the ray $\beta^\alpha_j$ and remove the part that does not contain $\S_{k+1}$. The removed regions are shown in white in Figure~\ref{fig:markov-to-thpath}. Note that it is possible that the ray $\beta^\alpha_1$ is contained in $\gamma_{k,\alpha}$, in which case we just skip it.

Note that $\beta^L_j$ and $\beta^R_{j'}$ do not intersect for any $j, j'$. This follows from Lemma~\ref{lem:no-intersect} applied to the domain $P=\S_{k+1}$ if $\S_{k+1}$ contains only one fundamental domain, and to $P=\F_u$ if both domains in $\S_{k+1}$ share a vertex~$u$.

Let us now check that $s_{k+1}$ lies in the interior of $\D_{k+1}$.
Let $v_{j,L}$, $v_{j,R}$ be the left and right ends of $s_j$ and let $v_{j,L}u_{j,L}$ and $u_{j,R}v_{j,R}$ be the sides of~$s_j$ adjacent to these ends. We will show that $v_{k+1,L}u_{k+1,L}$ lies inside $\D_{k+1}$.

Consider two cases. First, if $v:=v_{k+1,L}$ coincides with $v_{k,L}$ then by Step 3, for some $r\le n(v)$ the domains $\L_{k+1-r},\dots,\L_{k+2}$ are  consecutive petals in the flower $\F_v$ while $\L_{k+1-r}$ is not in $\F_v$, and the ray $\gamma_{k+1,L}=\gamma_{k,L}$ contains the side $wv\subset\dd_L\S_{k+1-r}$. But then the angle $wvu_{k+1,L}$ contains at most $n(v)-1$ petals, so $vu_{k+1,L}$ is not the continuation of $wv$ and hence is not contained in $\gamma_{k+1,L}$.

Similarly, if $v\ne v_{k,L}$, then $\gamma_{k+1,L}$ is the continuation of the side $wv=u^L_{m_L-1}u^L_{m_L}$ of $\dd_L\S_{k+1}$ adjacent to $v$, and the angle $wvu_{k+1,L}$ contains only one sector, namely, $\L_{k+1}$. Thus again $vu_{k+1,L}$ lies inside $\D_{k+1}$.

This proves item (i) of Claim~\ref{clm:FkThPath} if $s_{k+1}$ contains one or two sides. If it contains three sides: $s_{k+1}=v_{k+1,L}u_{k+1,L}u_{k+1,R}v_{k+1,R}$, it remains to rule out the possibility that $u_{k+1,L}\in\gamma_{k+1,R}$.
But in this case  the triangle $u_{k+1,L}u_{k+1,R}v_{k+1,R}$ has all its sides lying in $\dd\TR$, and this is impossible by~Lemma~\ref{lem:no-trig}. Therefore, $u_{k+1,L}$ and $u_{k+1,R}$ lie inside $\D_{k+1}$ and
item (i) is fully established.

Further, by construction $\D^-_{k+1}\setminus\D^-_k$ is bounded by $\dd_O\S_{k+1}\cup s_k\cup s_{k+1}=\dd\S_{k+1}$, hence $\D^-_{k+1}\setminus\D^-_k=\S_{k+1}$, and item (ii) holds. Items (iii) and (iv) for $k+1$ are also now clear.

In the case $k+1=N$ we likewise consecutively cut $\D_{N-1}$ along the rays through the segments of $\dd_O\S_N$; on the last step we cut along the last segment of this boundary. Therefore, $\D_N\setminus\D^{-}_{N-1}$ is bounded by
$s_{N-1}\cup\dd_O\S_N=\dd\S_N$, hence $\D_N\setminus\D^{-}_{N-1}=\S_N$,  verifying
the final statement of  Claim~\ref{clm:FkThPath}.
\end{proof}

\subsection{The time-reversing involution}
\label{subsec:involution}

The Markov coding defined above has the following property: the Markov chain with time reversed, that is, the Markov chain with the matrix $\Pi^T$, is the same as the initial one with the states renamed. This is possible precisely because the thickened path between domains $\B$ and $\A$ is exactly the thickened path from $\A$ to $\B$ read in the opposite direction. Thus we obtain an  involution which inverts arrangements and states; informally it swaps the past and the future domains for each state.
Precisely, we define the involution $\iota\colon \Xi\to\Xi$  by:
\begin{gather*}
A_0(e)\leftrightarrow A_0(e^{-1}),\qquad
A_{\alpha}[i_-,i_+](e)\leftrightarrow A_{\alpha}[i_+,i_-](e^{-1})\quad (\alpha=LR, RL),\\
A_L[i_-,i_+](e)\leftrightarrow A_R[i_+,i_-](e^{-1}),\qquad
B(e_L,e_R)\leftrightarrow D(e_R^{-1},e_L^{-1}),\\
C_k(e_L,e_R)\leftrightarrow C_{n(e_L,e_R)-k-1}(e_R^{-1},e_L^{-1}),\\
E_\alpha(e_L,e_M,e_R)\leftrightarrow E_\alpha(e_R^{-1},e_M^{-1},e_L^{-1}) \quad(\alpha=L,R).
\end{gather*}

\begin{prop} The involution $\iota$ maps the topological Markov chain with the adjacency matrix $\Pi$
	to the same chain with reversed time, that is, $\Pi_{\iota(j)\iota(k)}=\Pi_{kj}$. Also, $\iota(\Xi_S)=\Xi_F$ and vice versa.
\end{prop}

\begin{proof}
This follows directly from the definitions.
\end{proof}

\section{Operations with thickened paths}
\label{sec:operations}

In this section we develop some techniques for manipulating thickened paths which will be used in 
Section~\ref{subsec:strconn-aper} to establish strong connectivity and aperiodicity of the Markov chain. The same techniques will also be used in Section~\ref{subsec:proof-ineq} to verify the convergence conditions for Theorem~\ref{thm:main}.

\subsection{Adjusting the labels of states}

Recall from Definition~\ref{defn:generates} that a sequence $\underline j$ of states  generates  a sequence of domains $\underline\S$ if for each $k$ the pair $(\S_k,\S_{k+1})$ represents the configuration $\pi(j_k)$.  In certain circumstances, we will need to adjust the 
sequence $\underline j$ while leaving the  sequence it generates, and hence the sequence $\pi(j_k)$, unchanged. Such an adjustment is achieved by the following technical lemma. It will be crucial later to note that the required changes do not propagate beyond a definite bounded distance which depends only on the tessellation $\TR$. 

As we have noted in the proof of Theorem~\ref{thm:th-path-vs-sphere}, the sequence $\pi(\underline j)$ is uniquely determined by $\underline{\S}$, while the indices $i_{-,L/R}$ of any $A$-state $j_k$ are defined uniquely from the corresponding indices for the state $j_{k-1}$, and the indices $i_{+,L/R}$ are similarly defined in the backwards direction.

\begin{lemma}\label{lem:change-states}
	Suppose that the admissible sequence $\underline j=(j_0,\dots,j_{m-1})$ generates a sequence of domains $\underline\S=(\S_0,\dots,\S_m)$. Assume that $\pi(j_{m-1})=A(e)$. Let $j'_{m-1}$ be any $A$-state such that $\pi(j'_{m-1})=A(e)$ and $i_{-,\alpha}(j_{m-1})=i_{-,\alpha}(j'_{m-1})$ for $\alpha=L,R$. Then there exists an admissible sequence $\underline j'=(j'_0,\dots,j'_{m-1})$ generating the same sequence $\underline{\S}$. Moreover, if $s_i=\S_i\cap\S_{i+1}$ then $j_i=j'_i$ whenever  either $j_i$ is not of type $A$ or  $s_i$ and $s_{m-1}$ have no common points. In particular, $j_i = j'_i$ for $i \leq m-n_0 +1$ where $n_0 = \max \{ n(v): v \text{ is a vertex of }\TR\}$.  (While the lemma is trivial if $\R$ has no vertices inside $\DD$, we set $n_0=2$ in this case. This will be used below.) 
\end{lemma}

We remark that the same result holds \emph{mutatis mutandi} adjusting states forwards from~$j_0$.

\begin{proof}
	First assume that $i_{-,L}(j_{m-1})=i_{-,R}(j_{m-1})=1$. Then the states $j_{m-1}$ and $j'_{m-1}$ belong to the set $\{A_0(e)$, $A_L[1,i_+](e)$, $A_R[1,i_+](e)\}$. One can see from Definition~\ref{def:transitions} that all these states have the same set of allowed preceding states, so the sequence $\underline j'=(j_0,\dots,j_{m-2},j'_{m-1})$ is admissible.
	
	Now assume that, say, $i_{-,L}(j_{m-1})=k>1$. Then the states $j_{m-1}$ and $j'_{m-1}$ belong to the set $\{A_L[k,i_+](e)$, $A_{LR}[k,i_+](e)\}$. Let $l=i_{+,L}(j_{m-1})$, i.~e.~$l=i_+$ if $j_{m-1}=A_L[k,i_+](e)$ and $l=1$ if $j_{m-1}=A_{LR}[k,i_+](e)$; let $l'$ be defined in the same way for $j'_{m-1}$. Then the suffix $(j_{m-k},\dots,j_{m-1})$ of the sequence $\underline j$ has the following form:
	\begin{equation}\label{eq:last-turn}
		\begingroup\arraycolsep=0pt
		\left.
        \begin{array}[c]{c}
		A_L[1,l+k-1](e_1)\\
		A_{RL}[i_-,l+k-1](e_1)\\
		D(e_1,\tilde e_1)
		\end{array}
		\right\},
		A_L[2,l+k-2](e_2),\dots,A_L[k-1,l+1](e_{k-1}),
		\left\{
		\begin{array}[c]{c}
		A_L[k,i_+](e)\\
		A_{LR}[k,i_+](e).
		\end{array}
        \right.
		\endgroup
	\end{equation}
	Here $e_{i+1}=l(e_i)$ for $i=1,\dots,k-1$, where $e_k=e$. Formally speaking, it is possible that the whole sequence~$\underline j$ is only a suffix of the sequence in~\eqref{eq:last-turn}; the proof for this case is the same.

	Define the sequence $\underline j'$ as follows: $j'_i=j_i$ for $i=0,\dots,m-k-1$, for $i=m-k,\dots,m-2$ define $j'_i$ by the formula \eqref{eq:last-turn} with $l'$ in place of $l$ (that is, $j'_{m-k}=j_{m-k}$ if $j_{m-k}=D(e_1,\tilde e_1)$).  	Observe that these states are allowed: it is clear that all indices are positive and
	\begin{multline}\label{eq:last-turn-est}
	i_{-,L}(j'_{m-i})+i_{+,L}(j'_{m-i})=(k+1-i)+(l'+i-1)\\
    {}=k+l'=i_{-,L}(j'_{m-1})+i_{+,L}(j'_{m-1})\le n(v_L(e)),
	\end{multline}
	and since $s_{m-i}$ and $s_{m-1}$ have the same left end, we may replace $e$ by $e_{k+1-i}$ in the right-hand side of this formula. Also it is clear that $\underline j'$ is an admissible sequence.
	
	To prove the last statement observe that if $n_0=2$, then $\pi(\hat{\jmath})=A(e)$ implies $\hat{\jmath}=A_0(e)$, hence $\underline j =\underline j'$, while for $n_0\ge 3$ we have the following two cases. If $l=l'$, then $j'_i=j_i$ for all $i\le m-2$; otherwise $\max(l,l')\ge 2$, so \eqref{eq:last-turn-est} for $\underline j$ and $\underline j'$ yields $k\le n_0-2$, and $j'_i=j_i$ for $i\le m-k-1$. Hence in all cases $j'_i=j_i$ for $i\le m-n_0+1$. 
	\end{proof}

\subsection{Narrowing}
The next lemma shows how a sequence $\underline{\S}$ can be ``narrowed''  by reducing its final level  from two domains to  one. 
This  will be useful, for example, when  $\underline{\S}=(\S_0,\dots,\S_N)$ is a thickened path between its ends  and we need 
to find a thickened path $\underline{\S'}$ between $\S_0$ and  some intermediate domain $\L_k$ for which  $\S_k$ contains two domains.

 One way to  deal with this situation is to consider the  thickened path as a minimal convex union of fundamental domains. Then one can see that the desired thickened path is a subset of $(\S_0,\dots,\S_k)$ obtained from  $\bigcup_{j=0}^k\S_j$ by cutting along a line $\ell\subset\dd\TR$ incident to a common vertex of $\L_k$ and $R_k$ as shown in Figure~\ref{fig:narrowing}. 
 
However, since below we are mostly interested in the corresponding sequences of states, from now on we will consider not only thickened paths but any sequence  $(\S_0,\dots,\S_N)$ of domains generated by admissible sequences of states $(j_0,\dots,j_{N-1})$, in other words, we drop the conditions  that $j_0\in\Xi_S$, $j_{N-1}\in\Xi_F$.

Suppose the sequence $\underline\S=(\S_0,\dots,\S_m)$ is generated by an admissible sequence of states $\underline j=(j_0,\dots,j_{m-1})$. By the \emph{directed adjacency graph} associated to $\underline\S$ we mean the graph whose vertices are the domains in $\underline\S$ with a directed edges going from a domain $\A\in\S_k$ to $\B\in\S_{k+1}$ whenever $\A$ and $\B$ share a common side. 

\begin{lemma}\label{lem:narrowing}  
Suppose that the sequence $\underline\S=(\S_0,\dots,\S_m)$ is generated by an admissible sequence of states $\underline j=(j_0,\dots,j_{m-1})$. 
Assume that $\S_m$ contains two domains $\L_m,\R_m$  and let $\S'_m$ be one of them, for definiteness $\L_m$. Then there is a  narrowed sequence $\underline\S'$ from $\S_0$ to $\S'_m $ which can be described as follows.

1. Let $\S'_l$ be the set of domains $\A\subset\S_l$ such that there exists a path in the directed graph from $\A$ to $\S'_m$.
Let $k$ be the maximal number such that there is an edge $\R_k\to\L_{k+1}$; we set $k=-1$ if there is no such edge. Then $\S'_l=\S_l$ for $l\le k$ and $\S_l=\L_l$ for $l \ge k+1$.

2. The sequence $\underline\S'$ can be generated by an admissible sequence $\underline j'=(j'_0,\dots,j'_{m-1})$,  where one can assume that $j_i=j_i'$ if $\pi(j_i)=\pi(j'_i)$ and either $j_i$ is not of type $A$ or $s_i$ has no common vertex with $s_k$. In particular, $j_0=j_0'$ if $\S'_{n_0-1}=\S_{n_0-1}$, where $n_0$ is defined in Lemma~\ref{lem:change-states}. 
\end{lemma}

\begin{proof}
1. This is illustrated in  Figure~\ref{fig:narrowing}. 
 Let us go backward from $l=m$ to $l=0$. Every path from $\A\in\S'_l$ to $\S'_m$ should pass through $\B\in\S'_{l+1}$. For $l\ge k+1$ there is only one edge, $\L_l\to\L_{l+1}$ that ends in $\S'_{l+1}=\L_{l+1}$, hence $\S'_l=\L_l$. Similarly, for $l\le k$ we obtain $\S'_l=\S_l$.

2. Assume that $\S'_m=\L_m$ and let $k$ be the same as in the first statement. Then, following Lemma~\ref{lem:bottle-prop} (4) and still referring to Figure~\ref{fig:narrowing},  the state~$j_k$ is of type either $B$ or $E_R$ (if $k\ge 0$), and the states $j_{k+1},\dots,j_{m-1}$ are of types $C$ and $E_L$. 

Let us define the states $j'_l$ for $l=k,\dots,m-1$ as follows (see Figure~\ref{fig:narrowing}a):
\begin{center}
\begin{tabular}{|c|c|}
\hline
$j_l$&$j'_l$\\
\hline
$B(e_L,e_R)$&$A_R[1,n(e_L,e_R)-1](e_L)$ or $A_{LR}[i_-,n(e_L,e_R)-1](e_L)$\\
$E_R(e_L,e_M,e_R)$&$D(e_L,e_M)$\\
$C_i(e_L,e_R)$&$A_R[i+1,n(e_L,e_R)-i-1](e_L)$\\
$E_L(e_L,e_M,e_R)$&$A_R[1,n(e_L,e_M)-1](e_L)$\\
\hline
\end{tabular}
\end{center}
In the first line there are several options, we will choose one of them later.
Note also that if the table suggests $A_R[1,1](e)$, it should be replaced by $A_0(e)$.
It is clear by the construction that for all $l\ge k$ the state $j'_l$ is well-defined, $\pi(j'_l)$ is represented by the pair $(\S'_l,\S'_{l+1})$ and the transition
$j'_l\to j'_{l+1}$ is allowed.

%%%%%%%%%%%%%%%%%%%%
\begin{figure}[bt]
	\centering
	a)\enskip\raisebox{5ex}{\raisebox{-\height}{\includegraphics{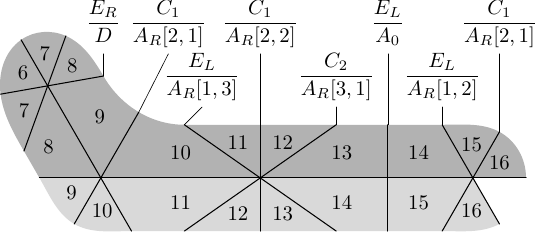}}}\\[12pt]
	\raisebox{7ex}{b)}\enskip{\includegraphics{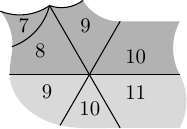}}\qquad
	\raisebox{7ex}{c)}\enskip{\includegraphics{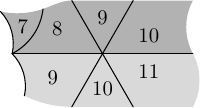}}
	\caption{The proof of Lemma~\ref{lem:narrowing}.
	The original sequence $\underline\S$ contains all shaded domains, while $\underline\S'$ contains only those shaded dark gray. Numbers inside the domains indicate their levels.	The states $j_i$ and $j_i'$ are shown above the curves $s_i$. In all figures $k=8$.}
	\label{fig:narrowing}
\end{figure}
%%%%%%%%%%%%%%%%%%%%
If $k=-1$, this ends the proof; otherwise we have to define $j'_l$ for $l=0,\dots,k-1$, as well as to choose $j'_k$ from the set given above in such a way that the transitions $j'_0\to\dots\to j'_k$ are allowed.
By default we set $j'_l=j_l$, although sometimes a correction is needed: this will be stated explicitly and only occurs in the final paragraph of the proof.

If $j_k=E_R(e_L,e_M,e_R)$, then either
$n(e_L,e_M)>2$ and $j_{k-1}=C_{n(e_L,e_M)-2}(\hat e_L,\hat e_R)$, or
$n(e_L,e_M)=2$ and
 $j_{k-1}\in\{B(\hat e_L,\hat e_R), E_R(\tilde e,\hat e_L,\hat e_R), E_L(\hat e_L,\hat e_R,\tilde e)\}$. 
In all these cases $e_L=r(\hat e_L)$, $e_M=l(\hat e_R)$ and the transition $j_{k-1}\to j'_k=D(e_L,e_M)$ is allowed.

If $j_k=B(e_L,e_R)$, there are several cases. As above, let $s_j=\S_j\cap\S_{j+1}$.

First of all, suppose that $s_k$ and $s_{k-1}$ have no common vertices,
i.~e.\ $j_{k-1}$ equals $A_0(\hat e)$, $A_L[i_-,1](\hat e)$, $A_R[i_-,1](\hat e)$, or $D(\hat e_L,\hat e_R)$ with $e_L,e_R$ non-adjacent to $\hat e^{-1}$ (or $\hat e_L^{-1},\hat e_R^{-1}$).
Then we set $j'_k=A_R[1,n(e_L,e_R)-1](e_L)$, so the transition $j_{k-1}\to j'_k$ is allowed.

Now assume that the left ends of $s_k$ and $s_{k-1}$ coincide (see Figure~\ref{fig:narrowing}b), i.~e.\ $j_{k-1}$ equals
$A_L[i_-,2](\hat e)$, $A_{RL}[i_-,2](\hat e)$ with $e_L=l(\hat e)$, or  $j_{k-1}$ equals $D(\hat e_L,\hat e_R)$ with $e_L=l(\hat e_L)$.
In the first of these three subcases we set $j'_k=A_{LR}[i_-+1,n(e_L,e_R)-1](e_L)$, and in the remaining two
we set $j'_k=A_{LR}[2,n(e_L,e_R)-1](e_L)$.

It remains to consider the case when the right (but not left) ends of $s_k$ and $s_{k-1}$ coincide (see Figure~\ref{fig:narrowing}c), hence $s_{k-1}$ and $s'_k=\S'_{k}\cap\S'_{k+1}$ have no common vertices.
Then either $j_{k-1}$ equals $A_R[i_-,2](\hat e)$ or $A_{LR}[i_-,2](\hat e)$ with $e_R=r(\hat e)$, or  it equals $D(\hat e_L,\hat e_R)$ with $e_R=r(\hat e_R)$.
Let $j'_k=A_R[1,n(e_L,e_R)-1](e_L)$, so in the last subcase the transition $j_{k-1}\to j'_k$ is admissible.
In the first two subcases apply Lemma~\ref{lem:change-states} to define the sequence $(j_0',\dots,j_{k-1}')$: if $j_{k-1}=A_R[i_-,2](\hat e)$, let $j'_{k-1}=A_R[i_-,1](\hat e)$,
and if $j_{k-1}=A_{LR}[i_-,2](\hat e)$, let $j'_{k-1}=A_L[i_-,1](\hat e)$.  

The last part of the statement, which estimates the common part of $\underline j$ and $\underline j'$, follows from the corresponding part in Lemma~\ref{lem:change-states}: $\S'_{n_0-1}=\S_{n_0-1}$ yields $k\ge n_0-1$, so when we apply that lemma with $m=k\ge n_0-1$, we can have $j_i\ne j'_i$ only for $i>m-n_0+1\ge 0$.
\end{proof}
 
\subsection{Joining}

The next lemma deals with how to join  two sequences of domains, $\underline\S^-=(\S_{-m},\dots,\S_0)$ and $\underline\S^+=(\S_0,\dots,\S_n)$  which have a common end $\S_0$  containing only one domain. 
The result  of such a join is not necessarily a thickened path between its ends, since  the union $\U=\bigcup\underline\S^-\cup\bigcup\underline\S^+$ may fail to be convex at the vertices of $\S_0$.

If the union is convex, we show that it can be generated by an admissible sequence of states. If convexity fails at some vertex $v$, there are two cases: either $v$ is incident  to $n(v)+1$ domains in the union, or to more than $n(v)+1$.  In the former case we will show that  the union can be enlarged to a thickened path,  while in 
 the latter case the union cannot be so enlarged: the part of $\U$ between the first and the last domain adjacent to $v$ can be shortcut by a path that goes ``the other way around~$v$''.  A possible enlargement in the first case is illustrated in Figure~\ref{fig:convexification}.

%%%%%%%%%%%%%%%%%%%%
\begin{figure}[hbt]
	\centering
	\includegraphics{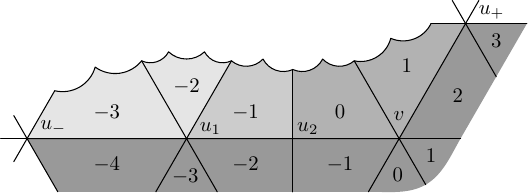}
	\caption{Joining two paths at a minimally concave vertex, see Definition~\ref{defn:min-concave}. The union of the two thickened paths shown in dark gray  is minimally non-convex at the vertex $v$. The lighter colours show the consecutive flowers we add to the path:	we first add the flower $\F_v$, so the adjacent vertex $u_2$ on $\gamma$ becomes minimally concave, then we add $\F_{u_2}$, making $u_1=w_-$ minimally concave, and then $\F_{u_1}$.
	The vertices $u_\pm$ are non-concave since they are incident to at most $n(u_\pm)-1$ domains in $\U$ and the addition of the flowers increases this number by one. For details see Lemma~\ref{lem:joining} (3).}
	\label{fig:convexification}
\end{figure}
%%%%%%%%%%%%%%%%%%%%

Consider the case in which  convexity fails  because $v$ is incident  to $n(v)+1$ domains in the union $\U$. Then
one can add to $\U$ the remaining domains in the flower $\F_v$, and these domains can be assigned levels in such a way that the levels of adjacent domains differ by one, see Figure~\ref{fig:convexification}. However, this adds one domain adjacent to a vertex $u$ next to $v$ in $\dd\U$, so if $\dd\U$ had a straight angle at $u$, now $u$ is adjacent to $n(u)+1$ domains in $\U'=\U\cup\F_v$. Thus we need to add $\F_u$, and so on until we arrive to the ends $u_\pm$ of the maximal geodesic segments  in $\dd\U$ starting from $v$ in both directions. Since the vertices $u_\pm$ are incident to less than $n(u_\pm)$ domains in $\U$,  the addition of one more domain does not destroy convexity at these points.
It can be shown that the resulting union of fundamental domains can be generated by an admissible sequence of states,  and hence  is a thickened path between its ends.

In fact, the whole analysis of thickened paths presented in this paper can be performed using this ``convexification'' technique, i.~e.\ the successive addition of flowers to vertices of a collection of domains, where the convexity failed; one can find a detailed exposition of this approach in the first version of this preprint~\cite{BKS-arXv1}. Rather than doing this, however, we keep track of states  using  Lemma~\ref{lem:joining} below.  

 \begin{definition} \label{defn:min-concave} Suppose given sequences of domains $\underline\S^-=(\S_{-m},\dots,\S_0)$ and $\underline\S^+=(\S_0,\dots,\S_n)$ such that $\S_0$ contains only one domain, and as usual denote $s_l=\S_l\cap\S_{l+1}$. We call the common vertex $v$ of $s_{-1}$ and $s_0$  \emph{concave}, if it is incident to more than $n(v)$ domains in the sequence $\underline\S=(\S_{-m},\dots,\S_n)$. It is \emph{minimally concave} it is incident to exactly $n(v)+1$ domains in  $\underline\S=(\S_{-m},\dots,\S_n)$.
\end{definition}

\begin{lemma}\label{lem:joining}
1. Let sequences of levels $\underline\S^-=(\S_{-m},\dots,\S_0)$ and $\underline\S^+=(\S_0,\dots,\S_n)$ be generated
by sequences of states $\underline j^-=(j_{-m},\dots,j_{-1})$ and $\underline j^+=(j_0,\dots,j_{n-1})$.
Assume that $\S_0$ contains only one domain
and that the curves $s_{-1}$ and $s_0$ have no common sides. Then at most one vertex in $s_{-1}\cap s_0$ is concave.

2a. Assume that there are no concave common vertices. Then the sequence $\underline\S = (\S_{-m},\dots,\S_0,\dots,\S_n)$ can be generated by a sequence of states $\underline{\hat\jmath}$.

2b. (See Figure~\ref{fig:convexification}.) Assume that there is a minimally concave common vertex $v$.
Let $u_-v$ and $vu_+$ be the maximal geodesic segments in $\dd\underline\S^-\setminus\dd\S_0$ and $\dd\underline\S^+\setminus\dd\S_0$ respectively.
Let $w_\pm$ be the internal vertices of the curve $u_-vu_+$ that are closest to $u_\pm$; it is possible that one or both of $w_\pm$ coincide with $v$. Define the curve $\gamma=w_-vw_+$. Then there exists a sequence $\underline{\hat\S}=(\hat{\S}_{-m},\dots,\hat{\S}_{n})$
generated by an admissible sequence of states $\underline{\hat\jmath}$ such that
\begin{enumerate}[label={(\roman*)}]
	\item $\S_i\subset \hat{\S}_i$ for all $i$;
	\item $\hat{\S}_i=\S_i$ if $\S_i$ has no common vertex with $\gamma$;
	\item $\bigcup\underline{\hat\S}$ is the union of $\bigcup\underline{\S}$ and of all flowers $\F_w$, where $w$ is a vertex of $\gamma$;
\end{enumerate}
3.  Let $\delta=s_{-1}\cup s_0$ in the case from 2a and $\delta=u_-vu_+$ in the case from 2b. Then one can assume that $\hat\jmath_t=j_t$ if $\pi(\hat \jmath_t)=\pi(j_t)$ and either $j_t$ is not of type $A$ or $s_t$ has no common points with $\delta$. In particular, $j_{n-1}=\hat\jmath_{n-1}$ if $n\ge n_0-1$ and (in the case from 2b) $\S_{n-n_0+1}=\hat{\S}_{n-n_0+1}$, where $n_0$ is defined in Lemma~\ref{lem:change-states}. 
\end{lemma}

\begin{proof}
1. Assume the contrary: both ends $v_L,v_R$ of the curves $s_{-1}$ and $s_0$ coincide. Then $\S_0$ is compact, so $N(\R)=4$, and $j_{-1}$ is of type $D$ and $j_0$ is of type $B$.
Therefore, both $v_L$ and $v_R$ are incident only to $\S_{-1},\S_0,\S_1$, but at least one of $v_L,v_R$ has $n(v)\ge 3$ by Assumption~\ref{asm:R}.

\smallskip

2a. Since $\S_0$ contains one domain, the state $j_{-1}$ is of type $D$ or $A$, and the state $j_0$ is of type $A$ or $B$.

First suppose that $s_{-1}$ and $s_0$ have no common vertices. Then if $j_{-1}$ is of type $A$, apply Lemma~\ref{lem:change-states} for $\hat\jmath_{-1}$ with $i_{+,L/R}(\hat\jmath_{-1})=1$, otherwise let $\underline{\hat\jmath}^-=\underline j^-$. Similarly, if $j_0$ is of type $A$, apply the analogue of Lemma~\ref{lem:change-states} with time inverted for the state $\hat\jmath_{0}$ with $i_{-,L/R}(\hat\jmath_0)=1$, otherwise let $\underline{\hat\jmath}^-=\underline j^-$.

The case when $s_{-1}$ and $s_0$ share both ends is trivial: here $j_{-1}$ is of type $D$, $j_0$ is of type $B$, and the convexity in the common vertices $v_{L,R}$ means that $n(v_{L,R})\ge 3$, so the transition $j_{-1}\to j_0$ is admissible.

Finally, suppose that $s_{-1}$ and $s_0$ share their left end only: $s_{-1}\cap s_0=v_L$.
Let $n_+$ (respectively, $n_-$) be the number of levels $\S_k$ with $k\ge 0$ (respectively, $k\le 0$) that are adjacent to $v_L$. By our assumption,   $n_-+n_+-1\le n(v_L)$, $n_\pm\ge 2$.

Let us construct the sequence $\underline{\hat\jmath}^-=(\hat\jmath_{-m},\dots,\hat\jmath_{-1})$ as follows. If $j_{-1}$ is of type $D$ we set $ \underline{\hat\jmath}^-=\underline j^-$; otherwise the construction is performed in two steps. First we have to ensure that $i_{-,L}(\hat{\jmath}_{-1})=n_--1$. This is in fact true for $j_{-1}$ unless all domains in $\underline{\S}^-$ are consecutive petals in the same flower. Namely, let us say that the state $j_l$ (with $l<0$) is ``correct'', if either it is not of type $A$, or it is of type $A$ and the index $i_{-,L}(j_l)$ coincides with the number of domains in $(\S_{-m},\dots,\S_l)$ that are incident to the left end of~$s_l$. One can check that if the transition $j_l\to j_{l+1}$ is not of type $A\to A$,  or if it is of type $A\to A$ and $\dd_L\S_{l+1}$ contains at least one side, then $j_{l+1}$ is correct. Also, if $j_l\to j_{l+1}$ is of type $A\to A$ \emph{and} the state $j_l$ is correct, then $j_{l+1}$ is correct.

Thus the state $j_{-1}$ is correct unless the states $j_{-m},\dots,j_{-1}$ are of type $A$ and the curves $s_{-m},\dots,s_{-1}$ are segments with the same left end $v_L$.
If $j_{-1}$ is correct, let $\underline{\tilde\jmath}^-=\underline j^-$, otherwise use Lemma~\ref{lem:change-states} to obtain the sequence $\underline{\tilde\jmath}^-$ with $i_{-,L}(\tilde\jmath_{-m})=1$; note that since $i_{-,L}(\tilde\jmath_{-m})\le i_{-,L}(j_{-m})$, the state $\tilde\jmath_{-m}$ is allowed. Now $\tilde\jmath_{-m}$ and hence $\tilde\jmath_{-1}$ are correct.

Apply Lemma~\ref{lem:change-states} again to transform $\underline{\tilde\jmath}^-$
into $\underline{\hat\jmath}^-$ with $i_{+,L}(\hat\jmath_{-1})=n_+$ and $i_{+,R}(\hat\jmath_{-1})=1$.
Note that the state $\hat\jmath_{-1}$ is allowed, since $i_{-,L}(\hat\jmath_{-1})=i_{-,L}(\tilde\jmath_{-1})=n_{-}-1$ due to its ``correctness'', so $i_{-,L}(\hat\jmath_{-1})+i_{+,L}(\hat\jmath_{-1})=n_-+n_+-1\le n(v_L)$.

Likewise, we construct a sequence $\underline{\hat\jmath}^+=(\hat\jmath_{0},\dots,\hat\jmath_{n-1})$ with $i_{-,L}(\hat\jmath_{0})=n_-$, $i_{+,L}(\hat\jmath_{0})=n_+-1$ that generates the same domains $\underline\S^+$. It remains to check that the transition $\hat{\jmath}_{-1}\to\hat{\jmath}_0$ is admissible. Clearly it belongs to one of the four types, $D\to B$, $D\to A$, $A\to B$, and $A\to A$. In the first case the argument from part (1) above works,  while the next two cases correspond to the transitions
$D(e_L,e_R)\to A_{L,LR}[2,\dots](l(e_L))$ and $A_{L,RL}[\dots,2](e)\to B(l(e),\dots)$, which are admissible. Finally, the case $A\to A$ splits into four subcases:
\begin{equation*}
\begin{array}{l}
A_L[i_-,i_+](e)\\
A_{RL}[i_-,i_+](e)
\end{array}
\Biggr\}\to \Biggl\{
\begin{array}{l}
A_L[i'_-,i'_+](l(e))\\
A_{LR}[i'_-,i'_+](l(e)).\\
\end{array}
\end{equation*}
Definition~\ref{def:transitions} reduces the existence of such a transition to the subcase $A_L\to A_L$, replacing $A_{RL}[i_-,i_+](e)$ by $A_L[1,i_+](e)$ and
$A_{LR}[i'_-,i'_+](l(e))$ by $A_{L}[i'_-,1](l(e))$. In the resulting transition
\begin{equation}
\label{eq:reduc-case}
A_L[\tilde{\imath}_-,\tilde{\imath}_+](e)\to
A_L[\tilde{\imath}'_-,\tilde{\imath}'_+](l(e))
\end{equation}
one can express the indices in terms of $n_\pm$,
so that in all four subcases the transition \eqref{eq:reduc-case} is of the form $A_L[n_--1,n_+](e)\to A_L[n_-,n_+-1](l(e))$, which is allowed.

\smallskip

2b. Let $k$ be such that $\dd_L\S_{-k}\supset u_-w_-$, see Figure~\ref{fig:convexification}. Let us show that the states $j_{-k+1},\dots,j_{-1}$ are of type $A$, while $j_{-k}$ is of type $A$ or $D$.
Indeed, let $u_0=u_-,u_1,\dots,u_p=v$ be the consecutive vertices on $u_-v$, thus $u_-w_-=u_0u_1$.
Then there are $-k=i_0<i_1<\dots<i_{p-1}\le -1$ such that $\dd_L\S_{i_t}=u_tu_{t+1}$ for $t=0,\dots,p-1$, and for $t=0,\dots, p-1$ and $i_t<i<i_{t+1}$ we have $\dd_L\S_i=u_{t+1}$, here $i_p=0$.
By Proposition~\ref{prop:conv-boundary} each segment $\mathbf J_t=(j_{i_t},\dots,j_{i_{t+1}-1})$ has the form
\begin{equation*}
(E_R,E_L)\quad\text{or}\quad(A\text{ or }D,A,\dots,A,A\text{ or }B).
\end{equation*}
Since $\S_0$ contains one domain, the segment $\mathbf J_{p-1}$ has the form
$(A\text{ or }D,A,\dots,A)$. Assume that $p>1$ and $\mathbf J_{p-1}$ starts with $D$. Then $\mathbf J_{p-2}$ ends with $B$ or $E_L$, hence $\L_{i_{p-1}}$ intersects each of $s_{i_{p-1}-1}$ and $s_{i_{p-1}}$ in one side and these sides meet in a vertex of $\L_{i_{p-1}}$. Thus if $\dd_L\S_{i_{p-1}}=u_{p-1}u_p$, the domain $\L_{i_{p-1}}$ would be a compact triangle, which is not allowed. Hence $\mathbf J_{p-1}$ has the form $(A,\dots,A)$ and $\mathbf J_{p-2}$ has the form $(A\text{ or }D,\dots,A)$. Repeating this argument $p$ times, we obtain the desired statement.

Moreover, let $e_l$ be the internal label on $u_{l-1}u_l$ and let $n_l=n(u_l)$.
Then one can see that the sequence $j_{-k},\dots, j_{-1}$ has the following structure:
\begin{gather*}
\left.
\begin{array}{r}
A_L[1,n_1-1](l(e_1))\\
A_{RL}[i_-,n_1-1](l(e_1))\\
D(l(e_1),r(l(e_1))^{-1})
\end{array}
\right\}
,A_L[2,n_1-2](l^2(e_1)),\dots,
A_L[n_1-1,1](l^{n_1-1}(e_1)),\\
\dddots,A_L[1,n_t-1](l(e_t)),A_L[2,n_t-2](l^2(e_t)),\dots,
A_L[n_t-1,1](l^{n_t-1}(e_1)),\dddots\\
A_L[1,m_--1](l(e_p)),\dots,
A_L[n_--2,m_--n_-](l^{n_--2}(e_p)),\hspace*{0.3\textwidth}\\
\hspace*{0.5\textwidth}
\left\{
\begin{array}{l}
A_L[n_--1,m_--n_-+1](l^{n_--1}(e_p))\\
A_{LR}[n_--1,i_+](l^{n_--1}(e_p))^{\dagger}
\end{array}
\right.
\end{gather*}
where $n_-$ is the number of levels in $\underline{\S}^-$ adjacent to $v$, and the number $m_-$ can be chosen arbitrarily satisfying
the inequality $n_-\le m_-\le n(v)$;
the case marked with $\dagger$ is allowed only if $m_-=n_-$.
Here every segment separated by ``$\dddots$'' is one of the $\mathbf J_t$'s. (Note that here and below we somewhat abuse the notation for states: thus for example, $A_{LR}[1,i_+](e)$ means $A_R[1,i_+](e)$, $A_L[1,1](e)$ means $A_0(e)$, etc.)

Similarly, let $K$ be such that $\dd_L\S_{K+1}\supset w_+u_+$, let $U_0=v,U_1,\dots,U_q=u_+$ be the consecutive vertices on $vu_+$, let $E_l$ be the external label on $U_{l-1}U_{l}$ and let $N_l=n(U_{l-1})$.
Then the sequence $(j_0,\dots, j_{K})$ has the form
\begin{gather*}
\left.
\begin{array}{l}
A_L[m_+-n_++1,n_+-1](l^{1-n_+}(E_1))\\
A_{RL}[i_+,n_+-1](l^{1-n_+}(E_1))^{\dagger}
\end{array}
\right\},\hspace*{0.5\textwidth}\\
\hspace*{0.2\textwidth}A_L[m_+-n_++2,n_+-2](l^{2-n_+}(E_1)),\dots
,A_L[m_+-1,1](l^{-1}(E_1)),\\
\dddots,A_L[1,N_t-1](l^{1-N_t}(E_t)),A_L[2,N_t-2](l^{2-N_t}(E_t)),\dots,
A_L[N_t-1,1](l^{-1}(E_t)),\dddots\\
A_L[N_q-1,1](l^{1-N_q}(E_q)),\dots,A_L[N_q-2,2](l^{-2}(E_q)),
\left\{
\begin{array}{l}
A_L[N_q-1,1](l^{-1}(E_q))\\
A_{LR}[N_q-1,i_+](l^{-1}(E_q))\\
B(l^{-1}(E_q),r(l^{-1}(E_q)^{-1})),
\end{array}
\right.
\end{gather*}
where $n_+$ is the number of levels in $\underline{\S}^+$ adjacent to $v$, the number $m_+$ should satisfy $n_+\le m_+\le n(v)$,
and the case marked with $\dagger$ is allowed only if $m_+=n_+$.

%%%%%%%%%%%%%%%%%%%%
\begin{figure}[hbt]
\centering
\includegraphics{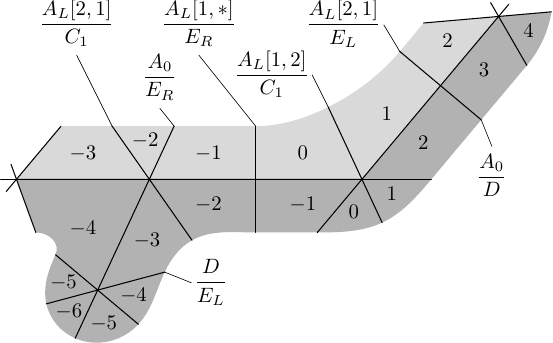}
\caption{The proof of Lemma~\ref{lem:joining}. The original sequence $\underline\S^-\cup\underline\S^+$ contains only dark gray domains, while $\underline{\hat\S}$ contains both light and dark gray ones. Numbers inside the domains indicate their levels. The states $j_i$ and $\hat\jmath_i$ are shown above or below the curves $s_i$. A
The asterisk in $j_{-1}$ represents $1$ or $2$. It is also possible that $j_{-1}=A_R[1,i_+]$ and/or $j_0=A_{RL}[i_-,2]$.}
\label{fig:joining}
\end{figure}
%%%%%%%%%%%%%%%%

Let us now define the states $\hat{\jmath}_{-k},\dots,\hat{\jmath}_{K}$, see Figure~\ref{fig:joining}.
In some sense this is opposite to the transformation in the proof of Lemma~\ref{lem:narrowing}.
For $l=-k,\dots,-1$ we use the following substitutions:
\begin{align*}
  \begingroup
  \left.
  \def\arraystretch{1.2}%
  \begin{array}{@{}l@{}}
  A_L[1,\dots](l(e))\\
  A_{RL}[\dots](l(e))\\
  \end{array}
  \right\}\endgroup&{}\mapsto
  \begin{cases}
  B(e^{-1},l(e)),&\text{if $l=-k$},\\
  E_R(l(e^{-1})^{-1},e^{-1},l(e)),&\text{otherwise},
  \end{cases}\\
  A_{L,LR}[i,\dots](l^i(e))&{}\mapsto C_{i-1}(r^{i-1}(e^{-1}),l^i(e)),\quad i\ge 2,\\
  D(l(e),\tilde e)&{}\mapsto E_L(e^{-1},l(e),\tilde e),
\end{align*}
while for $l=0,\dots,K$ we use
\begin{align*}
  \begingroup
  \left.
  \def\arraystretch{1.2}%
  \begin{array}{@{}l@{}}
  A_{L}[\dots,1](l^{-1}(f))\\
  A_{LR}[\dots](l^{-1}(f))\\
  \end{array}
  \right\}\endgroup&{}\mapsto
  \begin{cases}
  D(f^{-1},l^{-1}(f)),&\text{if $l=K$},\\
  E_L(r(f),f^{-1},l^{-1}(f)),&\text{otherwise},
  \end{cases}\\
  A_{L,RL}[\dots,i](l^{-i}(f))&{}\mapsto C_{n(v_L(f))-i}(r^{-i+1}(f^{-1}),l^{-i}(f)),\quad i\ge 2,\\
  B(l^{-1}(f),\tilde e)&\mapsto E_R(f^{-1},l^{-1}(f),\tilde e).
\end{align*}
It is straightforward to check that all transitions within
$(\hat{\jmath}_{-k},\dots,\hat{\jmath}_{-1})$ and
$(\hat{\jmath}_{0},\dots,\hat{\jmath}_{K-1})$ are admissible.
Let us verify that $\hat\jmath_{-1}\to\hat\jmath_0$ is also admissible.

Indeed, if $n_->2$, we have $\hat\jmath_{-1}=C_{n_--2}(r^{n_--2}(e_1^{-1}),l^{n_--1}(e_1))$, and if $n_-=2$, this state is of type $B$ or $E_R$
and has the same allowed next states as  ``the state $C_0(e_1,l(e_1))$''.
Similarly, $\hat\jmath_0=C_{n(v)-n_++1}(r^{2-n_+}(f_1^{-1}),l^{1-n_+}(f_1))$; for $n_+=2$ this formula means that $\hat\jmath_0$
has the same allowed previous states as this ``state $C_{n(v)-1}(\dots)$''.

By the assumptions of the lemma, $v$ is adjacent to $n(v)+1$ domains in $\underline\S$, hence
\begin{equation*}
n(v)=n_-+n_+-2,\quad
r^{n(v)-1}(e_1^{-1})=f_1^{-1},\quad
l^{n(v)+1}(e_1)=f_1.
\end{equation*}
Indeed, going around $v$ from $u_{p-1}v$ to $vu_1$ in an anti-clockwise direction, one crosses all these $n(v)+1$ domains in $\underline\S$,
while going clockwise one passes through the remaining $n(v)-1$ domains in $\F_v$.
Therefore, the transition $\hat\jmath_{-1}\to\hat\jmath_{0}$ takes the form $C_s(e_L,e_R)\to C_{s+1}(l(e_L),r(e_R))$, which is admissible.

It remains to define the states $\hat\jmath_{l}$ for $l<-k$ and $l>K$.
To do so, we join the sequences $\underline\S^{--}=(\S_{-m},\dots,\S_{-k})$, and $\underline\S^{++}=(\S_{K+1},\dots,\S_{n})$ with the 
respective sequences of states $\underline j^{--}=(j_{-m},\dots,j_{-k-1})$ and $\underline j^{++}=(j_{K+1},\dots,j_{n-1})$ to the constructed sequences
$\underline{\hat\S}^\diamond=(\hat\S_{-k},\dots,\hat\S_{K+1})$
and $\underline{\hat\jmath}^\diamond=(\hat\jmath_{-k},\dots,\hat\jmath_{K})$.

If $\S_{-k}$ contains two domains, then $j_{-k}$ is of type $D$, $\hat\jmath_{-k}$ is of type $E_L$, and they have the same set of admissible previous states, thus the joining is just a concatenation.

Otherwise let us show that the joining can be done by item 2a of this lemma.
Compare the conditions there for the joining of $\underline\S^{--}$ and $\underline{\hat\S}^\diamond$ and 
for that of $\underline\S^{--}$ and $(\S_{-k},\dots,\S_0)$, which is clearly admissible. 
Since $\hat s_{-k}=s_{-k}\cup u_0u_1$ and $u_0u_1\subset\dd\underline\S^-$, the curves $s_{-k-1}$ and $\hat s_{-k}$ have no common sides.
If $s_{-k-1}$ and $s_{-k}$ share their right end $w$, no domains adjacent to $w$ are added in $\underline{\hat\S}^\diamond$, so $w$ is convex for both joinings.
Now consider $u_0$. By  definition the angle of $\dd_L\underline\S^-$ at $u_0$ is less than $\pi$, hence it is adjacent to less than $n(u_0)$ domains in the first joining.
The only domain adjacent to $u_0$ that is added in  $\underline{\hat\S}^\diamond$ is the one adjacent to $\S_{-k}$ via the side $u_0u_1$.
Therefore, $u_0$  is adjacent to at most $n(u_0)$ domains in the second joining.
Note also that since $\hat\jmath_{-k}$ is of type $B$, the sequence $\underline{\hat\jmath}^\diamond$ stays the same after this joining.

Therefore, both sequences can be joined to $\underline{\hat\S}^\diamond$ at the same time. This operation can change only those states in $\underline j^{--}$ and $\underline j^{++}$ where 
the corresponding levels are adjacent to an end of $\hat s_{-k}$ or $\hat s_{K}$. Moreover, this end should be adjacent to different numbers of domains in $\underline{\hat\S}^\diamond$ and in $(\S_{-k},\dots,\S_{K+1})$, hence it is $u_-$ or $u_+$.

\smallskip

 3. This statement is again a direct consequence of the estimates in Lemma~\ref{lem:change-states}. 
In the case from 2a we replace $j_0$ by $\hat{\jmath}_{0}$ representing the same configuration and adjust the next states by that lemma. Hence we can get $j_{t}\ne \hat{\jmath}_t$ only for $t\le n_0-3$, and not for $t=n-1\ge n_0-2$.
Similarly, in the case 2b the equality $\S_{n-n_0+1}=\hat{\S}_{n-n_0+1}$ means that $K<n-n_0+1$, hence the modification of $\underline j^{++}=(j_{K+1},\dots,j_{n-1})$ reaches to at most $j_{(K+1)+n_0-3}$, and $(K+1)+n_0-3<n-1$, whence again $\hat j_{n-1}=j_{n-1}$. 
\end{proof}

The following corollary describes a precise sense in which the operations of narrowing and joining are mutually inverse. It is needed in the proof of Lemma~\ref{lem:wye} which is crucial for verifying Assumption~\ref{asm:ineq} for convergence in Section~\ref{subsec:thm-convergence}.  
\begin{cor}\label{cor:join-narrow}
Let $\underline{\hat\S}$ be constructed from $\underline\S^-$ and $\underline\S^+$ as in Lemma~\ref{lem:joining}. Then applying Lemma~\ref{lem:narrowing} to the sequence $\underline{\hat\S}^-=(\hat\S_i)_{i=-m}^0$ or $\underline{\hat\S}^+=(\hat\S_i)_{i=0}^n$ and the domain $\hat\S'_0=\S_0$, we arrive at the original sequences $\underline\S^\pm$.
\end{cor}

\begin{proof}
If there are no concave vertices in $\underline\S$, then $\underline{\hat\S}=\underline\S$, $\hat\S_0$ contains only one domain, hence no domains are removed when applying Lemma~\ref{lem:narrowing}. Let us now assume that the vertex $v$, which is the common left end of $s_0$ and $s_{-1}$,  is minimally concave.
Using the notation of Lemma~\ref{lem:joining}, one can see that any side in $u_-v$ separates $\hat\L_{l+1}$ and $\hat\R_l=\S_l$ for some $l<0$, while sides in $vu_+$ separate $\hat\R_l=\S_l$ and $\hat\L_{l-1}$ with $l>0$. Hence it is impossible to connect the domain $\B$ in $\underline{\hat\S}$ to $\hat\R_0=\S_0$ by a path of adjacent domains in $\underline{\hat\S}$ with monotonic indices if and only if $\B$ is one of the domains
$\hat\L_{-k+1},\dots,\hat\L_{K}$, which comprise $\bigcup\underline{\hat\S}\setminus\bigcup\underline{\S}$. Therefore, these are the domains removed when we apply Lemma~\ref{lem:narrowing}.  
\end{proof}

\section{Strong connectivity and aperiodicity}
\label{subsec:strconn-aper}

 In this section we will establish strong connectivity and aperiodicity of the Markov chain $(\Xi, \Pi)$ constructed in Section~\ref{sec:markov-coding}.
Together these properties yield the existence of $N>0$ such that all entries of the matrix $\Pi^{N}$ are positive, where $\Pi$ is the adjacency matrix as defined in~\ref{subsec:transitionmatrix}.  
\begin{definition}
	Let $X$ and $M$ be respectively the set of states and the adjacency matrix of a topological Markov chain.\\
	1. $(X,M)$ is   \emph{strongly connected} if for any $x,y\in X$ there exists a sequence  $z_0=x,z_1,\dots,\allowbreak z_k=y$ such that $z_j\to z_{j+1}$ is an admissible transition for all $j=0,1,\dots,k-1$.\\
	2. $(X,M)$ is   \emph{aperiodic} if there does not exist $c>1$, together with a   map $\tau\colon X\to \mathbb{Z}/c\mathbb{Z}$,
	such that for every admissible transition $x\to y$ one has $\tau(y)=\tau(x)+1$.
\end{definition}

\subsection{Head and tail paths} We start by constructing some special short admissible sequences, of length depending only on the geometry of $\TR$, which will be useful in enabling us to pass from one state to another. 
Since these sequences can  be attached  either to the beginning or end of an admissible sequence, we  call them ``head'' and ``tail'' sequences, whence the notation, $\underline\H^e$ and $\underline\T^e$.

\begin{prop}\label{prop:tail-paths}
	There exists $M\le 4$, depending only on $\TR$, such that for every $e\in G_0$ there is an admissible sequence of states $\underline{i}^e=(i^e_0\to\dots\to i^e_{M-1}=\mathbf{t}(e))$  with the following properties.  First, $i^e_0=A_0(e)$ unless $\R$ belongs to the special case of Remark~\ref{rem:specialcase} and $e$ is the label on the compact side, in which case $i^e_0$ can be chosen to be either $A_{LR}[2,2](e)$ or $A_{RL}[2,2](e)$.  Let   $\underline\T^e=(\T_0^e=\R,\T_1^e,\dots,\T_M^e)$  be the sequence of levels generated by $\underline{i}^e$ and let $\gamma_{L}, \gamma_R$ be the maximal geodesic segments of the curve $\dd\underline\T^e\setminus\dd\T^e_0$ adjacent respectively to the left or right end of $s_0=\T^e_0\cap\T^e_1$ provided that this end is in $\DD$, otherwise let $\gamma_{L,R}$ be empty. Then  $\gamma_{L,R}$ have no common points with $s_{M-1}=\T^e_{M-1}\cap\T^e_M$.
	
	Similarly, there exists a sequence of states $\underline{j}^e=(\mathbf{h}(e)=j^e_{-M}\to\dots\to j^e_{-1})$ 	 so that  if $\underline{j}^e$ generates a sequence of levels $\underline\H^e=(\H_{-M}^e,\H_{-M+1}^e,\dots,\H_0^e=\R)$, then the maximal geodesic segments of the curve
	$\dd\underline\H^e\setminus\dd\H^e_0$ starting at its ends lying inside~$\DD$ have no common points with $\H^e_{-M}\cap\H^e_{-M+1}$,  and $j^e_{-1}=A_0(e)$ except for the label on the compact side in the special case of Remark~\ref{rem:specialcase}, where $j^e_{-1}$ is any of $A_{LR}[2,2](e)$, $A_{RL}[2,2](e)$. 
\end{prop}

%%%%%%%%%%%%%%%%%%%%
\begin{figure}[tbp]
	\centering
	\begingroup\tabcolsep=0pt
	a)\enskip\raisebox{5ex}{\raisebox{-\height}{\includegraphics{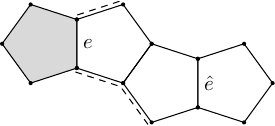}}}\qquad b)\enskip\raisebox{5ex}{\raisebox{-\height}{\includegraphics{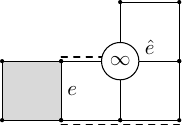}}}\\[12pt]
	c)\enskip\raisebox{5ex}{\raisebox{-\height}{\includegraphics{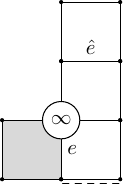}}}\qquad
	\begin{tabular}[t]{l}
		d)\enskip\raisebox{5ex}{\raisebox{-\height}{\includegraphics{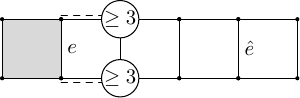}}}\\
		e)\enskip\raisebox{5ex}{\raisebox{-\height}{\includegraphics{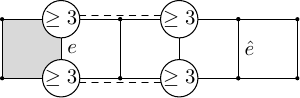}}}\\
	\end{tabular}\\[12pt]
	f)\enskip\raisebox{5ex}{\raisebox{-\height}{\includegraphics{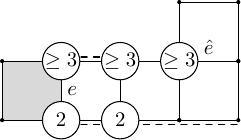}}}\qquad g)\enskip\raisebox{5ex}{\raisebox{-\height}{\includegraphics{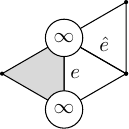}}}\qquad h)\enskip\raisebox{5ex}{\raisebox{-\height}{\includegraphics{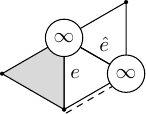}}}\\[12pt]
	i)\enskip\raisebox{5ex}{\raisebox{-\height}{\includegraphics{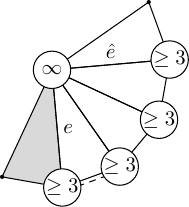}}}\qquad j)\enskip\raisebox{5ex}{\raisebox{-\height}{\includegraphics{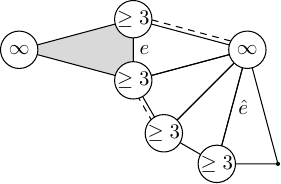}}}
	\endgroup
	\caption{``Tail'' paths $\underline\T^e$ from Proposition~\ref{prop:tail-paths}. 
	The domains $\T^e_0$ are shaded gray.
	Dashed lines show the maximal possible extent of the segments $\gamma_{L,R}$.
	Numbers in circles indicate $n(\dots)$ for the vertex, $\infty$ meaning that the vertex lies on $\dd\DD$.
	Note that $\mathbf{t}(e)=A_0(\hat e)$  in all cases except f),  in which case $\mathbf{t}(e)=A_L[2,1](\hat e)$.
	}
	\label{fig:tail-paths}
\end{figure}
%%%%%%%%%%%%%%%%
\begin{proof}
	The proposition and proof is illustrated in Figure~\ref{fig:tail-paths}. We will construct the paths $\underline{\T}^e$ in such a way that all states $i^e_k$ are of type $A$. 	
	
	If $N(\R)\ge 5$, take $M=3$ and consider a path $\underline{\T}^e$ such that $\dd_L\T_1^e$ and $\dd_R\T_2^e$ contain at least two sides each, see Fig.~\ref{fig:tail-paths}a). Then the geodesic segments $\gamma_{L,R}$ cannot pass through the internal vertices of $\dd_L\T_1^e$ and $\dd_R\T_2^e$, hence they cannot touch $\T_3^e$.
	
	If $N(\R)=4$ and $\R$ is non-compact, one can construct $\underline{\T}^e$ as shown in Fig.~\ref{fig:tail-paths}b) or~c).
	
	Suppose $N(\R)=4$ and $\R$ is compact and has no opposite vertices with $n(v)=2$. Let $s_1$ be the side of $\T_1^e=e\R$ opposite to $s_0$, and choose $\T_2^e$ to be the domain on the other side of $s_1$. Then there are the following three cases. If both ends of $s_1$ have $n(\dots)\ge 3$, we use the path shown in Fig.~\ref{fig:tail-paths}d). Alternatively, if both ends of $s_0$ have $n(\dots)\ge 3$, the same holds
	for both ends of $s_2$, the side of $\T_2^e$ opposite to $s_1$, and we use the path from Fig.~\ref{fig:tail-paths}e).
	The remaining case is when both $s_0$ and $s_1$ have ends with $n(\dots)=2$. Then these ends lie on the same (say, right) boundary,
	and the left ends of $s_0$, $s_1$, and $s_2$ all have $n(\dots)\ge 3$. Then we construct our path as shown in Fig.~\ref{fig:tail-paths}f).
	
	If $N(\R) = 3$ and all of  the sides are non-compact, we can use the paths shown in Fig.~\ref{fig:tail-paths}g),~h).
	
	Finally, in the special case from Remark~\ref{rem:specialcase} the paths $\underline{\T}^e$ are defined as shown in
	Fig.~\ref{fig:tail-paths}i),~j). If $e$ is the label on the compact side then as in  (\ref{fig:tail-paths}j) we have $i_0^e=A_{LR}[2,2](g)$, while the analogous symmetrical path gives the sequence with $i_0^e=A_{RL}[2,2](g)$.
	
	As for the second statement concerning the head paths $\underline\H^e$, one can set $j^e_{-k}=\iota(i^{e^{-1}}_{k-1})$, $\H^e_{-k}=\T^{e^{-1}}_k$, $k=0,\dots,M$. In particular, we have $\mathbf{h}(e)=\iota(\mathbf{t}(e^{-1}))$.
\end{proof}

The next proposition describes  two important combinations of  the ``head'' and ``tail'' sequences, which will be directly used to prove connectivity and aperiodicity.

\begin{prop}\label{prop:HT-combin} 

	1. Take any $e,\hat e$ such that $\hat e\ne e^{-1}$ and let $\underline\H^e$, $\underline\T^e$ be the paths from Proposition~\ref{prop:tail-paths}.
	 Then Lemma~\ref{lem:joining} can be applied to join $\underline\H^{\hat e}$ and~$\underline\T^e$. The resulting path is generated by an admissible sequence $\underline k$ of states with $k_{-M}=\mathbf{h}(\hat e)$, $k_{M-1}=\mathbf{t}(e)$.\\
	2. Take any $e$, $\hat e$ such that $\hat e\ne e^{-1}$ and denote $\S_0=\R$, $\S_1=e\R$. At most one vertex~$w$ of $\S_1$ is shared with $\H^{\hat e}_{-1}$.  Choose any side of $\S_1$ that is not incident to~$w$ and let $\tilde e$ be the label on this side outside   $\S_1$. Denote $\S_{l+1}=e\T^{\tilde e}_l$ for $l=0,\dots,M$. Then one can apply Lemma~\ref{lem:joining} twice to join the  three sequences, $\underline\S^-=\underline\H^{\hat e}$, $\underline{\S}^{\diamond}=(\S_l)_{l=0}^1$, and $\underline{\S}^+=(\S_l)_{l=1}^{M+1}$.
	The sequence of states $\underline k$ corresponding to this joining starts with $\mathbf{h}(\hat e)$ and ends with $\mathbf{t}(\tilde e)$.
\end{prop}

\begin{proof}
	1. We use the notation of Lemma~\ref{lem:joining}. The sides $s_{-1}=\S_{-1}\cap\S_0$ and $s_0=\S_0\cap\S_1$ do not coincide since $\hat e\ne e^{-1}$.
	Assume that they have a common vertex $v$. In the non-special case $v$ is incident to only two domains, $\S_{-1}$ and $\S_0$, in $\underline\H^{\hat e}$ since $i^{\hat e}_{-1}=A_0(\hat e)$.
	Similarly, $v$ is incident to only two domains in $\underline\T^e$. Therefore, $v$ is incident only to $\S_{-1},\S_0,\S_1$ in~$\underline\S$, hence $v$ is either convex or minimally concave.
	In the special case from Remark~\ref{rem:specialcase} at most one of the sides $s_{-1}$, $s_0$ is compact, hence by the same argument $v$ is adjacent to at most four domains in $\underline\S$, and we use that $n(v)\ge 3$.
	The last part of the statement follows directly from Proposition~\ref{prop:tail-paths}.
	
	2. The paths $\underline{\S}^-$ and $\underline{\S}^+$ are generated by the sequences $\underline j^{\hat e}$ and $(i^{\tilde e}_{k-1})_{k=1}^M$ respectively, while $\underline{\S}^\diamond$ is generated by any state of the form $A_{\dots}(e)$.
	As in the first part of the proof, we can apply Lemma~\ref{lem:joining} to the paths $\underline{\S}^-$ and $\underline{\S}^\diamond$ and obtain the path $\underline{\tilde\S}$ and the sequence of states $(\tilde\jmath_l)_{l=-M}^0$ such that $s_{-M}=\S_{-M}\cap\S_{-M+1}$ have no common points with the domains added in this step and $\tilde\jmath_{-M}=\mathbf{h}(\hat e)$.
	
	Let us now consider two cases. Assume first that no domains are added at the first step. Then $\tilde s_0=s_0$ and $s_1=\S_1\cap\S_2$ are different sides of $\S_1$, and if they have a common vertex $u$, then $u\ne w$, hence $u$ is incident only to $\S_0$ and $\S_1$ in $\underline{\tilde\S}$, that is, $u$ is either convex or minimally concave in $\underline{\tilde\S}\cup \underline{\S}^+$. Therefore, we can apply Lemma~\ref{lem:joining} to $\underline{\tilde\S}$ and $\underline{\S}^+$ and obtain a path~$\underline{\hat\S}$ and corresponding  sequence $\underline k$ of states. As above, $k_M=i^{\tilde e}_{M-1}=\mathbf{t}(\tilde e)$. To show that $k_{-M}=\tilde\jmath_{-M}$, we need to check that the maximal geodesic segment $\gamma$ in $\dd\underline{\tilde\S}\setminus s_1$ starting at $u$ does not reach $s_{-M}$. Clearly, if $\gamma$ does not go past an end $z$ of $s_0$, it does not reach $s_{-M}$. On the other hand, if $\gamma$ goes past $z$, then it ends at the same point as the maximal geodesic segment of $\dd\underline\H^{\hat e}\setminus s_0$ starting from $z$, hence $\gamma$ does not reach $s_{-M}$ in this case also.
	
	Now assume that we have added some domains in the first step. This means that the vertex~$w$ exists, the curve $\tilde s_0=\tilde{\S}_0\cap\S_1$ consists of the two sides of $\S_1$ that are adjacent to $w$, and $\tilde{\S}_0=\S_0\cup\A$ consists of two domains adjacent to these sides. Therefore, the curves $\tilde s_0$ and $s_1$ have no common sides, and if they have a common vertex $\tilde w$, then $\tilde w$ is the end of $\tilde s_0$, hence it is adjacent only to $\tilde{\S}_0$ and $\tilde{\S}_1$ in $\underline{\tilde\S}$. Therefore, we may apply Lemma~\ref{lem:joining} and get $k_M=\mathbf{t}(\tilde e)$ as in the previous case. To show that $k_{-M}=\mathbf{h}(\hat e)$, it remains to check that the domains added in this step have no common points with $s_{-M}$. Indeed, no domains are added to $\tilde{\S}_l$ with $l\le 0$ since $\tilde{\S}_0$ already contains two domains. On the other hand, if $s_{-M}$ has a common point with $\hat{\S}_l$ for $l\ge 1$, then it has a common point with~$\hat\S_0=\tilde{\S}_0=\S_0\cup\A$. The side $s_{-M}$ has no common points with $\S_0$ by the construction of~$\underline\H^{\hat e}$, and if $s_{-M}$ has a common point with $\A$, we arrive at a contradiction with the first application of Lemma~\ref{lem:joining}: the addition of $\A$ then changes the indices $i_{+,L/R}$ for the state corresponding to $s_{-M}$, i.~e.\ $\tilde\jmath_{-M}\ne j^{\hat e}_{-M}$.
	\end{proof}

\subsection{Connectivity and aperiodicity}

\begin{prop}\label{lem:str-conn}     
	The topological Markov chain $(\Xi,\Pi)$ introduced in Definition~\ref{def:transitions} is strongly connected.
\end{prop}

\begin{proof}
	The scheme of the proof is the following. We will consider several cases. In each case we choose the set $\Omega\subset\Xi$ with $\iota(\Omega)=\Omega$ and prove two properties:
	\begin{enumerate}[itemsep=0pt,label={(\roman*)}]
		\item\label{item:path-to-Omega} for every $j\in\Xi$ there exists a path along the arrows in the adjacency graph of the Markov chain from $j$ to some state $k\in \Omega$;
		\item\label{item:paths-in-Omega} for every $k_1,k_2\in \Omega$ there exists a path from $k_1$ to $k_2$.
	\end{enumerate}
	Let us write $j\rightsquigarrow k$  to indicate that there exists a path going from $j$ to $k$.
	Observe that properties~\ref{item:path-to-Omega} and \ref{item:paths-in-Omega} imply strong connectivity. Namely, from \ref{item:path-to-Omega} we have that for any states $j,j'\in\Xi$ there exist $k,k'\in\Omega$ such that $j\rightsquigarrow k$ and $\iota(j')\rightsquigarrow k'$. Applying the involution $\iota$ to the second of these relations, we get $\iota(k')\rightsquigarrow j'$. Finally,  \ref{item:paths-in-Omega} yields $k\rightsquigarrow \iota(k')$, and we have $j\rightsquigarrow k\rightsquigarrow \iota(k')\rightsquigarrow j'$.
	
	Property~\ref{item:path-to-Omega} is proven in the same way in all cases.
	Namely, if $\R$ does not belong to the special case of Remark~\ref{rem:specialcase}, let $\Omega=\{A_0(e):e\in G_0\}$. First of all, we can reach an $A$-state from a state $j$ via a series of states of the form $C\dots CDA$. Then we transform an $A$-state to a state with smaller index $i_+$, arriving eventually at a state $k$ with $\pi(k)=A(\hat e)$ and $i_{+,L/R}(k)=1$. Then the state $k$ can be followed by a state $A_0(e)$ where $e$ is any label not adjacent to $\hat e^{-1}$.
	
	In the special case of Remark~\ref{rem:specialcase} we set
	\begin{equation}\label{eq:Omega0-specialcase}
		\Omega=\{A_0(e):s_e\text{ is not compact}\}.
	\end{equation}
	The procedure given above brings us either to $\Omega$ or to an $A$-state with $i_+=2$, say, $\pi(k)=A(\hat e)$, $i_{+,L}(k)=2$, $i_{+,R}(k)=1$.
	If $s_{\hat e}$ is compact, we have $k\to A_L[i_-,1](\tilde e)\to A_0(\tilde e)$, where $\tilde e=l(\hat e)$ is the label on a non-compact side, and $i_-=i_{-,L}(k)+1$.
	Similarly, if $s_{\hat e}$ is non-compact, then $s_{\tilde e}$ is compact, $k\to A_{LR}[i_-,2](\tilde e)$, and we reduce this case to the previous one.
	
	It remains to check property~\ref{item:paths-in-Omega}.
	\smallskip
	
	1. Assume $N(\R)\ge 5$. 
	Let us construct the tail paths $\underline{\T}^e$ from Proposition~\ref{prop:tail-paths} and shown in Figure~\ref{fig:tail-paths}a)  in a uniform manner, namely, we choose the domains $\T^e_{2,3}$ in such a way that $\dd_R\T^e_1$ and $\dd_L\T^e_2$ contain one side each. Denote by $t(e)$ the label such that $\mathbf{t}(e)=A_0(t(e))$; $t(e)$ is shown as $\hat e$ in Figure~\ref{fig:tail-paths}a).
	Then $t\colon G_0\to G_0$ is a bijection. Indeed, for any $e'\in G_0$,  a path $\underline\T=(\T_i)_{i=0}^3$ such that $(\T_2,\T_3)$ represents the configuration $A(e')$ and $\dd_L\T_2$ and $\dd_R\T_1$ contain one side each is unique up to the group action. Hence if we require $\T_0=\R$,  we have $\underline\T=\underline\T^e$ for some $e\in G_0$, thus $e'=t(e)$. Note also that for the map $h$ defined by $\mathbf h(e)=A_0(h(e))$ we have $h(e)=(t(e^{-1}))^{-1}$ by the construction of the paths $\underline\H^e$.
	
	Now take any $f,\hat f\in G_0$ such that $\hat f\ne f^{-1}$ and denote
	\begin{equation*}
		e=t^{-1}(f),\quad \hat e=h^{-1}(\hat f)=(t^{-1}(\hat f^{-1}))^{-1}.
	\end{equation*}
	Hence $e\ne\hat e^{-1}$, and we may consider the path from the first part of Proposition~\ref{prop:HT-combin} for these $e,\hat e$. This shows that $A_0(\hat f)\rightsquigarrow A_0(f)$ for any $f,\hat f$ such that $\hat f\ne f^{-1}$. Finally, to verify $A_0(f)\rightsquigarrow A_0(f^{-1})$ choose any $e\in G_0\setminus\{f,f^{-1}\}$ and observe that
	$A_0(f)\rightsquigarrow A_0(e)\rightsquigarrow A_0(f^{-1})$.
	
	\smallskip
	
	2. Let us assume that $N(\R)=4$ and $\R$ is compact. Define the bijection $\tau\colon G_0\to G_0$ as follows. For $e\in G_0$ consider the side of $\R$ with the inside label~$e$, then the outside label on the opposite side of $\R$ equals $\tau(e)$. Then $A_0(e)\to A_0(\tau(e))$. Choose $m$ such that $\tau^m=\mathrm{id}$.
	
	Assumption~\ref{asm:R} implies that there is a label $f\in G_0$ such that $n(v_L(f))\ge 3$ and $n(v_R(f))\ge 3$. Then we have
	\begin{equation*}
	A_0(f)\rightsquigarrow A_0(\tau^{m-1}(f))\to A_L[1,2](f)\to A_L[2,1](l(f))\to A_0(\tau(l(f)))\rightsquigarrow A_0(l(f)).
	\end{equation*}
	Similarly, $A_0(f)\rightsquigarrow A_0(r(f))$, thus $A_0(f)\rightsquigarrow A_0(g)$ for all $g\ne f^{-1}$. By the involution, we get $A_0(g)\rightsquigarrow A_0(f^{-1})$ for $g\ne f$.
	The inequalities $n(v_{L,R}(f))\ge 3$ imply that $n(v_{R,L}(f^{-1}))\ge 3$,
	so by the same argument we have $A_0(f^{-1})\rightsquigarrow A_0(g)$ for all $g\ne f$, $A_0(g)\rightsquigarrow A_0(f)$ for $g\ne f^{-1}$. Finally,
	$A_0(f)\leftrightsquigarrow A_0(f^{-1})$ since for $h\ne f,f^{-1}$ we have
	\begin{equation*}
	A_0(f)\leftrightsquigarrow A_0(h)\leftrightsquigarrow A_0(f^{-1}),
	\end{equation*}
	hence property (ii) holds.
	
	\smallskip
	
	3. The argument from the previous case works for a non-compact $\R$ with $N(\R)=4$ if there exists a label $f\in G_0$ such that for both $\alpha=L,R$ we have that either $v_\alpha(f)$ is undefined or $n(v_\alpha(f))\ge 3$. If, say, the left end of $s_f$ lies on $\dd\DD$ then one can simply write $A_0(f)\to A_0(l(f))$. (Here we use an extended definition of $l(f)$:  the left end of $s_f$ and right end of $s_{l(f)^{-1}}$ are either the same point of $\dd\DD$ or they are joined by an arc from $\dd_{\overline{\DD}}\R\cap\dd\DD$.)  
	
	\smallskip

%%%%%%%%%%%%%%%%%%%%
\begin{figure}[tb]
	\centering
	\begin{tabular}[t]{l}
		a)\enskip\raisebox{5ex}{\raisebox{-\height}{\includegraphics{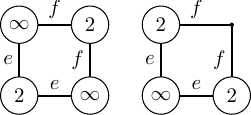}}}\\
		\quad\\
		b)\enskip\raisebox{5ex}{\raisebox{-\height}{\includegraphics{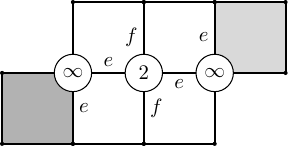}}}\\
	\end{tabular}\quad
	\begin{tabular}[t]{l}
		\quad\\[-1ex]
		c)\enskip\raisebox{5ex}{\raisebox{-\height}{\includegraphics{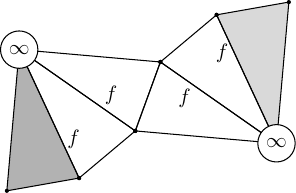}}}\\
	\end{tabular}
	\caption{The proof of Proposition~\ref{lem:str-conn} (4) and (5). The notation is similar to that of~Figure~\ref{fig:tail-paths}. The first and the last domains in the paths are shaded dark gray and light gray respectively.}
	\label{fig:str-conn}
\end{figure}
%%%%%%%%%%%%%%%%

	4. Consider now a non-compact $\R$ with $N(\R)=4$  which is not covered by the previous case. Then $\R$ has either one vertex/arc on $\dd\DD$ or two such vertices in opposite corners. It is easily checked that the only possible arrangements of side pairings are as shown in Figure~\ref{fig:str-conn}a). In both cases $\tau$ has a pair of $2$-cycles and we have
	$A_0(e)\to A_0(f) \to A_0(e)$,
	$A_0(e^{-1})\to A_0(f^{-1}) \to A_0(e^{-1})$.
	
	In the left case we have $A_0(e)\to A_0(f^{-1})$,
	$A_0(e^{-1})\to A_0(f)$ while in the right case the path shown in Figure~\ref{fig:str-conn}b) yields $A_0(e)\rightsquigarrow A_0(e^{-1})$:
	\begin{equation*}
	A_0(e)\to B(e,f)\to D(f^{-1},e^{-1})\to A_0(e^{-1}),
	\end{equation*}
	and similarly $A_0(e^{-1})\rightsquigarrow A_0(e)$.
	Thus in both cases the cycles of $\tau$ are linked.
		
	\smallskip
	
	5. Finally, consider the special case from Remark~\ref{rem:specialcase}, where
	$\Omega$ is defined by~\eqref{eq:Omega0-specialcase}.
	Let a side $s_f$ be non-compact, say $g=r(f)$, where $s_g$ is the compact side. Then property~\ref{item:paths-in-Omega}, namely, that $A_0(f)\rightsquigarrow A_0(f^{-1})$, is shown in Figure~\ref{fig:str-conn}c):
	\begin{equation*}
	A_0(f)\to A_R[1,2](f)\to A_{RL}[2,2](g)\to A_L[2,1](f^{-1})\to A_0(f^{-1}).
	\qedhere
	\end{equation*}
\end{proof}

\begin{prop}\label{lem:aperiodic}
	The topological Markov chain $(\Xi,\Pi)$ defined in Definition~\ref{def:transitions} is aperiodic.
\end{prop}

\begin{proof}
	Suppose that our Markov chain has a period $c$, that is, an index $\delta(i)\in\mathbb Z/c\mathbb Z$ can be assigned to every state $i\in\Xi$ in such a way that all allowed transitions $i\to j$ satisfy $\delta(j)=\delta(i)+1$.
	
	Take any $e_1$, $e_2$ and choose $\hat e\ne e_1^{-1},e_2^{-1}$. Then using the paths from the first part of Proposition~\ref{prop:HT-combin} we have $\delta(\mathbf{t}(e_s))=\delta(\mathbf{h}(\hat e))+2M-1$, $s=1,2$,  whence $\delta(\mathbf{t}(e_1))=\delta(\mathbf{t}(e_2))$. Therefore,  $\delta(\mathbf{t}(e))$ is the same for all $e$, we denote it by $\delta_{\mathbf{t}}$. Similarly, $\delta(\mathbf{h}(e))$ equals the same number $\delta_{\mathbf{h}}$ for all $e$, and $\delta_{\mathbf{t}}=\delta_{\mathbf{h}}+2M-1$.
	On the other hand, any path from the second part of Proposition~\ref{prop:HT-combin} yields $\delta_{\mathbf{t}}=\delta_{\mathbf{h}}+2M$.
	Therefore, $2M-1\equiv 2M\pmod{c}$, whence $c=1$.
\end{proof}

\begin{cor}\label{cor:PiN-positive}By Propositions~\ref{lem:str-conn} and~\ref{lem:aperiodic} the Markov chain $(\Xi,\Pi)$ is strongly connected and aperiodic, hence there exists $N_0>0$ such that all entries of the matrix $\Pi^{N_0}$ are positive.
\end{cor}

\begin{proof}That this result is implied by strongly connectivity and aperiodicity is well known, see for example~\cite{KemSn}.   
\end{proof}
 
\section{Spherical sums and Markov operator}
\label{sec:sph-avg}

In this section we express  the spherical averages from Equation~\eqref{eq:sph-avg} in terms of powers of a Markov operator, see   Lemma~\ref{lem:sph-avg-PU}, and obtain an identity relating this Markov operator with its adjoint, see Lemma~\ref{lem:adjoint}.

\subsection{Thickened paths and spheres in the group graph}\label{subsec:g-and-u}

Consider a state $k\in\Xi$ and let $(\S_-,\S_+)$ be a representation of the configuration $\pi(k)$. Let $\L_\pm, \R_\pm$ be the left and the right domains in $\S_\pm$; as usual, if the state has only one past (respectively, future) domain then $\L_-=\R_-$ (respectively, $\L_+=\R_+$).

Define the maps $\g,\u\colon\Xi\to G$ as follows: let $\L_-=h\R$, then
\begin{equation}\label{eq:g-and-u}
 \L_+=h\g(k)^{-1}\R,\quad \R_+=h\u(k)^{-1}\R.
\end{equation}
Clearly, these definitions do not depend on the choice of a representation for $k$.

\begin{lemma}\label{lem:g-and-u}
The maps defined above   $\g$ and $\u$ satisfy the following identities:\\
1) $\u(\iota(k))=\u(k)^{-1}$ for any $k\in\Xi$,\\
2) $\g(k)=\u(j)^{-1}\g(\iota(j))^{-1} \u(k)$ for any $j,k\in\Xi$ such that $k\to j$ is an admissible transition.
\end{lemma}

\begin{proof}
1. Consider a representation $(\S_-,\S_+)$ of the state $k$ as above. Then $(\widetilde{\S}_-,\widetilde{\S}_+)=(\S_+,\S_-)$ is a representation of the state $\iota(k)$, and $\widetilde{\L}_\pm=\R_\mp$, $\widetilde{\R}_\pm=\L_\mp$. Hence if $\L_-=h\R$, $\R_+=g\R$, then $\u(k)^{-1}=h^{-1}g$, $\u(\iota(k))^{-1}=g^{-1}h$.

2. If $k\to j$ is admissible, one can consider the sets $\S_-$, $\S_+$, and $\S_{++}$ such that $(\S_-,\S_+)$ represents $k$ and $(\S_+,\S_{++})$ represents $j$.
Define $\L_\alpha,\R_\alpha$ as above for $\alpha\in\{{-},{+},{++}\}$. Let $\L_-=h\R$. Then
\begin{equation*}
\L_+=h\g(k)^{-1}\R,\quad \R_+=h\u(k)^{-1}\R,\quad \R_{++}=h\g(k)^{-1}\u(j)^{-1}\R.
\end{equation*}
On the other hand, $(\widetilde{\S}_-,\widetilde{\S}_+)=(\S_{++},\S_+)$ represents the state $\iota(j)$, hence the formula for
$\widetilde{\L}_-=\R_{++}$ yields that $\R_+=\widetilde{\L}_+=h\g(k)^{-1}\u(j)^{-1}\g(\iota(j))^{-1}\R$.
\end{proof}

Recall that the set $\Paths_{n-1}^{S\to F}$ defined by Equation~\eqref{eq:PathsSF}  is the set of all admissible sequences of length $n$ which begin and end  in  start and end states respectively. By Theorem~\ref{thm:th-path-vs-sphere}, $\Paths_{n-1}^{S\to F}$ is in bijective correspondence with the set of thickened paths of length $n$ and hence to the sphere of radius $n$ in the graph of $G$. More precisely:

\begin{lemma}Consider the map $\Phi\colon \Paths_{n-1}^{S\to F}\to G$, where
	\begin{equation*}
	\Phi(j_0\to\dots\to j_{n-1})=\u(j_{n-1})\g(j_{n-2})\dots \g(j_0).
	\end{equation*}
	Then $\Phi$ is a bijection of $\Paths_{n-1}^{S\to F}$ onto the set $S_{n}(G)=\{g\in G:|g|_{G_0}=n\}$.
\end{lemma}

\begin{remark}Note that for $j_{n-1}\in \Xi_F$ there is \emph{only one} future fundamental domain,  hence $\u(j_{n-1})=\g(j_{n-1})$.
	The reason for using $\omega$ rather than $\gamma$ in the final step will become apparent later, see Lemma~\ref{lem:sph-avg-PU}.
\end{remark}

\begin{proof}
	As observed above, sequences from $\Paths_{n-1}^{S\to F}$ bijectively correspond to thickened paths $\underline{\S}$ from $\R$ to $g\R$ with $g\in S_n(G)$.
	Take $\underline{j}\in\Paths_{n-1}^{S\to F}$. Let the sequence $\underline{\S}=(\S_0=\R,\dots,\S_n)$ be generated by $\underline j$ and let $\L_k$, $\R_k$ be the left and right domains in $\S_k$. Define $h_k\in G$ so that $\L_k=h_k\R$. Then
	\begin{equation*}
		g=h_{n-1}\u(j_{n-1})^{-1}=h_{n-2}\g(j_{n-2})^{-1}\u(j_{n-1})^{-1}=\dots=[\u(j_{n-1})\g(j_{n-2})\dots \g(j_0)]^{-1},
	\end{equation*}
	and it remains to use that $g\mapsto g^{-1}$ is a bijective map of the sphere $S_{n}(G)$.
\end{proof}

\subsection{Parry measure}

Let $\Pi$ be the adjacency matrix of the topological Markov chain described in Definition~\ref{def:transitions}. By Corollary~\ref{cor:PiN-positive} for some $N_0$ all elements of the matrix $\Pi^{N_0}$ are positive. The Perron--Frobenius theorem then yields that the matrix $\Pi$ has a unique (up to a scaling) eigenvector~$h$ with nonnegative coordinates and that all coordinates of~$h$ are positive:
\begin{equation*}
\sum_{j} \Pi_{ij}h_j=\lambda h_i\text{ and }h_i>0\quad \text{for all $i\in\Xi$}.
\end{equation*}
Moreover the corresponding eigenvalue~$\lambda>0$ has multiplicity one and is larger than the absolute value of any other eigenvalue of~$\Pi$.
The eigenvalue $\lambda$ is called the \emph{Perron--Frobenius} (PF) eigenvalue and $h$ is called the \emph{right Perron--Frobenius} eigenvector.
The matrix $\pi=(p_{ij})$ with entries
\begin{equation}\label{eq:Parry-pij}
p_{ij}=\frac{h_j}{\lambda h_i}\Pi_{ij}
\end{equation}
is stochastic and the corresponding Markov chain has the following property: the probability of an admissible sequence of transitions
depends only on the initial and the final states in this sequence and the number of steps:
\begin{equation}\label{eq:Parry-prod}
p_{i_0i_1}\dots p_{i_{n-1}i_n}=\frac{h_{i_n}}{\lambda^n h_{i_0}} \Pi_{i_0i_1}\dots \Pi_{i_{n-1}i_n}=\frac{h_{i_n}}{\lambda^n h_{i_0}}.
\end{equation}
The Markov measure defined by the matrix $\pi=(p_{ij})$ is called the \emph{Parry measure}. Its stationary distribution is
\begin{equation}\label{eq:StatDist}
p_i=\alpha_i h_i,
\end{equation}
where $\alpha$ is the left PF eigenvector of $\Pi$: $\alpha \Pi=\lambda\alpha$, normalized by $\alpha h=\sum_i \alpha_ih_i=1$.

The time-reversing involution on the set of states implies certain symmetries for the Parry measure.

\begin{prop}\label{prop:Parry-inv} Suppose given an involution $\iota\colon\Xi\to\Xi$  such that $\Pi_{\iota(j)\iota(k)}=\Pi_{kj}$ for all $j,k\in\Xi$.
	Then the transition probability matrix $(p_{ij})$ and the stationary distribution $(p_i)$ of the Parry measure corresponding to the matrix $\Pi$ satisfy the following equations:
	\begin{equation*}
	p_{\iota(j)}=p_{j},\qquad
	p_{\iota(j)\iota(k)}=\frac{p_k p_{kj}}{p_j}\qquad\text{for all $j,k\in\Xi$}.
	\end{equation*}
\end{prop}

\begin{proof}Let $J$ be the matrix for the substitution $\iota$. Then $J=J^T=J^{-1}$, $J\Pi J=\Pi^{T}$.
	As above, let $\lambda$ be the Perron--Frobenius (PF) eigenvalue for $\Pi$ and let $\alpha$ and $h$ be its left and right PF eigenvectors, normalized by $\alpha h=1$.
	Then $\alpha J$ is a left PF eigenvector for $J\Pi J=\Pi^T$, whence $(\alpha J)^T=J\alpha^T$ is a right PF eigenvector for $\Pi$.
	Therefore, $J\alpha^T$ is proportional to $h$: $\alpha_{\iota(k)}=ch_k$.
	Now
	\begin{equation*}
	p_{\iota(j)}=\alpha_{\iota(j)}h_{\iota(j)}=ch_j\cdot \frac{1}{c}\alpha_{j}=p_j
	\end{equation*}
	and
	\begin{equation*}
	p_{\iota(j)\iota(k)}=\frac{\Pi_{\iota(j)\iota(k)}h_{\iota(k)}}{\lambda h_{\iota(j)}}=\frac{\Pi_{kj}c^{-1}\alpha_k}{\lambda c^{-1}\alpha_j}=
	\frac{\Pi_{kj}h_j}{\lambda h_k}\frac{h_k\alpha_k}{h_j\alpha_j}=\frac{p_k p_{kj}}{p_j}.\qedhere
	\end{equation*}
\end{proof}

\subsection{The Markov operator}

Recall that the group $G$ acts on a Lebesgue probability space $(X,\mu)$ by measure-preserving maps $T_g$.
We denote $T_g f:=f \circ T_g^{-1}$ for any function $f\in L^p(X,\mu)$.
Denote
\begin{equation*}
\widetilde{\mathbf{S}}_n(f)=\sum_{|g|=n} T_g^{-1} f,\quad\text{then}\quad \mathbf{S}_n(f)=\frac{\widetilde{\mathbf{S}}_n(f)}{\widetilde{\mathbf{S}}_n(1)}=\frac{\sum_{|g|=n} T_g^{-1} f}{\#\{g:|g|=n\}},
\end{equation*}
where $\mathbf{S}_n(f)$ is defined by~\eqref{eq:sph-avg}.

Consider the probability space $Y=\Xi\times X$ with the product measure $\nu=p\times\mu$. Here $p(\{i\})=p_i$, where $p_i$ is defined by \eqref{eq:StatDist}.
It is convenient to identify a function $\varphi\in L^1(Y,\nu)$ with a tuple of functions $(\varphi_i)_{i\in\Xi}$, where $\varphi_i({}\cdot{})=\varphi(i,{}\cdot{})$.

Define the following operators $P,U\colon L^1(Y,\nu)\to L^1(Y,\nu)$:
\begin{equation}\label{eq:PandU}
(P\varphi)_i=\sum_j p_{ij} T^{-1}_{\g(i)}\varphi_j, \qquad
(U\varphi)_j=T^{-1}_{\u(j)}\varphi_{\iota(j)}.
\end{equation}
It is clear that $P$ and $U$ are measure-preserving Markov operators, meaning that both preserve $L^p$-norms for any $p \in [1,\infty]$ and both map the positive cone into itself.

\begin{lemma}\label{lem:sph-avg-PU}For any function $f\in L^1(X,\mu)$ define a function $\varphi^{(f)}\in L^1(Y,\nu)$ by
	\begin{equation*}
	(\varphi^{(f)})_j=
	\begin{cases}
	\dfrac{1}{h_{\iota(j)}}f,&j\in \Xi_S,\\
	0,&\text{otherwise}.
	\end{cases}
	\end{equation*}
	Then
	\begin{equation}\label{eq:Sn-as-PkU}
	\widetilde{\mathbf{S}}_n(f)=\lambda^{n-1}\sum_{j\in \Xi_S} h_j(P^{n-1}U\varphi^{(f)})_j.
	\end{equation}
\end{lemma}

\begin{proof}Indeed,
	\begin{multline*}
	\widetilde{\mathbf{S}}_n(f)=
	\sum_{\substack{i_0\in \Xi_S,i_{n-1}\in \Xi_F,\\i_1,\dots,i_{n-2}\in\Xi}} \Pi_{i_0i_1}\dots \Pi_{i_{n-2}i_{n-1}} T^{-1}_{\u(i_{n-1})\g(i_{n-2})\dots \g(i_0)}f={}\\
	\lambda^{n-1}\sum_{\substack{i_0\in \Xi_S,i_{n-1}\in \Xi_F,\\i_1,\dots,i_{n-2}\in\Xi}} h_{i_0}p_{i_0i_1}\dots p_{i_{n-2}i_{n-1}}\frac{1}{h_{i_{n-1}}}
	T_{\g(i_0)}^{-1}\dots T_{\g(i_{n-2})}^{-1}T_{\u(i_{n-1})}^{-1}f={}\\
	\lambda^{n-1}\sum_{i_0\in \Xi_S} h_{i_0} \biggl(\sum_{i_1} p_{i_0i_1}T_{\g(i_0)}^{-1}\biggl(\dots
	\biggl(\sum_{i_{n-1}} p_{i_{n-2}i_{n-1}}T_{\g(i_{n-2})}^{-1}
	 \underbrace{T_{\u(i_{n-1})}^{-1}\biggl(\frac{\chi_{\Xi_F}(i_{n-1})}{h_{i_{n-1}}}f\biggr)}_{(U\varphi^{(f)})_{i_{n-1}}}
	\biggr)\dots\biggr)\biggr)=\\
	\lambda^{n-1}\sum_{i_0} h_{i_0} (P^{n-1}U \varphi^{(f)})_{i_0}.\qedhere
	\end{multline*}
\end{proof}

\subsection{The dual (adjoint) operator}
Let us recall that for $\varphi,\psi\in L^2(Y,\nu)$ we have
\begin{equation*}
\langle \varphi,\psi\rangle=\sum_{k\in\Xi}p_k\langle \varphi_k,\psi_k\rangle.
\end{equation*}
A short computation shows that if an operator $Q$ has the form
\begin{equation*}
(Q\varphi)_i=\sum_{j\in\Xi} p_{ij} T_{ij} \varphi_j,\quad\text{then its dual satisfies}  \quad (Q^*\psi)_j=\sum_{k\in\Xi} \frac{p_kp_{kj}}{p_j} T^*_{kj} \psi_k.
\end{equation*}
Therefore, for $P$ defined by \eqref{eq:PandU} we have
\begin{equation}\label{eq:P-adjoint}
(P^*\psi)_j=\sum_{k\in\Xi} \frac{p_kp_{kj}}{p_j}T_{\g(k)}\psi_k.
\end{equation}

\begin{lemma}\label{lem:adjoint}
	The Markov operators $P$ and $U$ defined by~\eqref{eq:PandU} satisfy the following identities:
	\begin{equation*}
	U=U^{-1}=U^*,\qquad P^*=UPU.
	\end{equation*}
\end{lemma}

\begin{proof}
These identities follow from Lemma~\ref{lem:g-and-u} and Proposition~\ref{prop:Parry-inv}. For example, let us prove the second:
	\begin{multline*}
	(UPU\psi)_j=T^{-1}_{\u(j)}(PU\psi)_{\iota(j)}=T^{-1}_{\u(j)}\sum_l p_{\iota(j),l}T^{-1}_{\g(\iota(j))}(U\psi)_l=\\
	\sum_l p_{\iota(j),l}T^{-1}_{\u(j)}T^{-1}_{\g(\iota(j))}T^{-1}_{\u(l)}\psi_{\iota(l)}=
	\sum_k p_{\iota(j),\iota(k)}T_{\u(j)^{-1}\g(\iota(j))^{-1}\u(k)}\psi_k.
	\end{multline*}
	For the last equality we substitute $l=\iota(k)$ and use the first identity in Lemma~\ref{lem:g-and-u}.
	Now using Proposition~\ref{prop:Parry-inv}, the second identity in  Lemma~\ref{lem:g-and-u} and formula \eqref{eq:P-adjoint} one can see that the right-hand side equals $(P^*\psi)_j$.
\end{proof}

\section{Proof of the main theorem}
\label{sec:proof-main}

To prove our main result, Theorem~\ref{thm:main}, we use a new theorem on pointwise convergence for powers of a Markov operator, Theorem~\ref{thm:convergence}, which is stated in  Subsection~\ref{subsec:thm-convergence}; its proof is postponed to Section~\ref{sec:proof-convergence}. The result is an elaboration  of that used in~\cite{Buf-Annals} under weaker assumptions.
  In order to apply  Theorem~\ref{thm:convergence} to the operators defined by~\eqref{eq:PandU} some work is needed to check that these assumptions hold.  
 
In  Subsection~\ref{subsec:solve-eqns} we check Assumptions~\ref{asm:Qn}, \ref{asm:QstarmQm} in the ergodic case, that is,  when the sigma-algebra~$\I_{G_0^2}$ is trivial.  The remaining Assumption~\ref{asm:ineq} is checked in Subsection~\ref{subsec:proof-ineq}, concluding the proof of Theorem~\ref{thm:main} in the ergodic case. Finally in Subsection~\ref{subsec:conclusion} we deal with the general, non-ergodic, case.

\subsection{General theorem on pointwise convergence}
\label{subsec:thm-convergence}

Let $(Z, \eta)$ be a Lebesgue probability space, and let $Q$ be a measure-preserving Markov operator on $L^1(Z,\eta)$. In order to state our convergence result, we need the following assumptions.

\begin{asm}\label{asm:adjoint}
	There exists a decomposition $Q=VW$, where $V$ and $W$ are measure-preserving Markov operators, so that $Q^*=WV$.
\end{asm}

\begin{asm}\label{asm:Qn}
	For every $n\in\mathbb{N}$ the equation $Q^n\psi=\psi$ has only constant solutions in $L^2(Z,\eta)$.
\end{asm}

\begin{asm}\label{asm:QstarmQm}
	There exists $m\in\mathbb N$ such that the equation $(Q^*)^m Q^m\psi=\psi$ has only constant solutions in $L^2(Z,\eta)$.
\end{asm}

\begin{asm}\label{asm:ineq}
	There exists a sequence of operators $A_n$ and constants $C,K>0$ and $a,b\in\mathbb N$ so that for all $n\ge m_0:=\lceil a/2\rceil$ the following inequality holds for any nonnegative $\varphi\in L^1(Z,\eta)$:
	\begin{equation}\label{eq:asm-ineq}
	WQ^{2n-a}\varphi\le C\sum_{j=-b}^b (Q^*)^n Q^{n+j}\varphi+A_n\varphi.
	\end{equation}
Here $W$ is the operator from Assumption~\ref{asm:adjoint}. The operators $A_n\colon L^1(Z,\eta)\to L^1(Z,\eta)$ map nonnegative functions into nonnegative ones, and for any $p\in[1,\infty]$ map $L^p(Z,\eta)$ to itself. Moreover  $\|A_n\|_{L^p}\le \alpha_n$, with $\sum_{n=m_0}^\infty \alpha_n\le K$.
\end{asm}

\begin{remark}
	Applying $V'=QV$ to both sides of \eqref{eq:asm-ineq}, we arrive at the inequality
	\begin{equation}\label{eq:var-asm-ineq}\tag{\ref{eq:asm-ineq}${}'$}
		Q^{2n-a'}\varphi\le CV'\sum_{j=-b}^b (Q^*)^n Q^{n+j}\varphi+A'_n\varphi
	\end{equation}
	with the same estimates on the norms of the operators $A'_n$. We will use both \eqref{eq:asm-ineq} and \eqref{eq:var-asm-ineq} below.
\end{remark}

\begin{thm}\label{thm:convergence}
	Let $Q\colon L^1(Z,\eta)\to L^1(Z,\eta)$ be a measure-preserving Markov operator acting on a Lebesgue probability space $(Z,\eta)$ and satisfying Assumptions \ref{asm:adjoint}--\ref{asm:ineq}.
	Then for every function $\varphi\in L\log L(Z,\eta)$ the sequence $Q^n\varphi$ converges almost surely and in $L^1$ to $\int_Z\varphi\,d\eta$ as $n\to\infty$.
\end{thm}
As remarked above, the proof of this theorem is deferred to Section~\ref{sec:proof-convergence}.

We now proceed to check that the above assumptions hold in our case.

\subsection{Checking Assumptions~\ref{asm:Qn} and~\ref{asm:QstarmQm}.}
\label{subsec:solve-eqns}

Let $P, U$ be the Markov operators defined  in~\eqref{eq:PandU} and define 
\begin{equation}\label{eq:QVW}
Q=P^2,\quad V=PU,\quad\text{and}\quad W=UP.
\end{equation}
 In this subsection we  check Assumptions~\ref{asm:Qn} and~\ref{asm:QstarmQm} of Theorem~\ref{thm:convergence} for these $Q, U, V$ in the case in which  the sigma-algebra $\I_{G_0^2}$ is trivial.
To do this we express the equations from these assumptions in terms of the components $\varphi_j$, $j\in\Xi$, of a function $\varphi \in L^2(Y,\nu)$.

\begin{prop}\label{prop:eqns-to-systems}Let $P$ be the Markov operator defined by \eqref{eq:PandU}. Then the following hold.\\
	1. A function $\varphi\in L^2(Y,\nu)$ is a solution to the equation $P^k\varphi=\varphi$ if and only if for any admissible sequence $i_0\to i_1\to\dots\to i_k$ of states we have
	\begin{equation}\label{eq:asmQn-comp}
	\varphi_{i_0}=T_{\g(i_0)}^{-1}\dots T_{\g(i_{k-1})}^{-1}\varphi_{i_k}.
	\end{equation}
	2. For $k\ge N_0$, where $N_0$ is defined in Corollary~\ref{cor:PiN-positive}, a function $\varphi\in L^2(Y,\nu)$ is a solution to $(P^*)^kP^k\varphi=\varphi$ if and only if for any admissible sequences $i_0\to i_1\to\dots\to i_k$ and $j_0\to j_1\to\dots\to j_k$ with $i_0=j_0$ we have 
	\begin{equation*}
	T_{\g(i_1)}^{-1}\dots T_{\g(i_{k-1})}^{-1}\varphi_{i_k}=T_{\g(j_1)}^{-1}\dots T_{\g(j_{k-1})}^{-1}\varphi_{j_k}.
	\end{equation*}
\end{prop}

\begin{proof}
	1. The equation $P^k\varphi=\varphi$ is equivalent to
	\begin{equation*}
	\varphi_{i_0}=(P^k\varphi)_{i_0}=\sum_{i_1,\dots, i_k}p_{i_0i_1}\dots p_{i_{k-1}i_k}T^{-1}_{\g(i_0)}\dots T^{-1}_{\g(i_{k-1})}\varphi_{i_k},
	\end{equation*}
	hence, since the $T_{g}$'s are unitary,
	\begin{equation}\label{eq:norm-asmQn}
	\norm{\varphi_{i_0}}_{L^2}\le \sum_{i_1,\dots, i_k}p_{i_0i_1}\dots p_{i_{k-1}i_k}\norm{\varphi_{i_k}}_{L^2}.	\end{equation}
	Multiplying these inequalities by $p_{i_0}$ and summing them up for all ${i_0}\in\Xi$ we obtain
	\begin{equation*}
	\sum_{i_0} p_{i_0}\norm{\varphi_{i_0}}_{L^2}\le \sum_{i_k} \biggl[\sum_{i_0,\dots,i_{k-1}}p_{i_0} p_{i_0i_1}\dots p_{i_{k-1}i_k}\biggr]\norm{\varphi_{i_k}}_{L^2}=\sum_{i_k} p_{i_k} \norm{\varphi_{i_k}}_{L^2}.
	\end{equation*}
	Therefore, for each $i_0$ inequality \eqref{eq:norm-asmQn} is indeed an equality, and the vector $(\norm{\varphi_i}_{L^2})_{i\in\Xi}$ is a right eigenvector of the stochastic matrix $\pi^k$, where the matrix $\pi=(p_{ij})$ is defined by~\eqref{eq:Parry-pij}. Corollary~\ref{cor:PiN-positive} and the Perron--Frobenius theorem yield that $(1,\dots,1)$ is the only eigenvector of $\pi$ with nonnegative coordinates up to scaling, hence all $\varphi_i$ have the same $L^2$-norm.
	
	Finally, in the Hilbert space $L^2(X,\mu)$ the triangle inequality \eqref{eq:norm-asmQn} attains equality only if all nonzero summands are proportional to each other with positive coefficients, whence $\varphi_{i_0}=c\cdot T^{-1}_{\g(i_0)}\dots T^{-1}_{\g(i_{n-1})}\varphi_{i_n}$. Calculating the $L^2$-norms of both sides, we get $c=1$.
	
	2. Similarly, $(P^*)^kP^k\varphi=\varphi$ yields
	\begin{multline*}
	\varphi_{j_k}=\sum_{\substack{j_0,\dots,j_{k-1}\\i_1,\dots,i_k}}
	\biggl[\frac{p_{j_0}p_{j_0j_1}\dots p_{j_{k-1}j_k}}{p_{j_k}}p_{j_0i_1}p_{i_1i_2}\dots p_{i_{k-1}i_k} \times\\[-1ex]
	T_{\g(j_{k-1})}\dots T_{\g(j_1)} T^{-1}_{\g(i_1)}\dots T^{-1}_{\g(i_{k-1})}\varphi_{i_k}\biggr].
	\end{multline*}
	The remaining proof is the same as in the first statement: $(\norm{\varphi_i})_{i\in\Xi}$ is a right eigenvector of the stochastic matrix $(\pi^*)^k\pi^k$ with positive entries, where $(\pi^*)_{ij}=p_j p_{ji}/p_i$;
	hence the $L^2$-norms of all $\varphi_i$'s are equal, and the same argument with the triangle inequality completes the proof.
\end{proof}

\begin{lemma}\label{lem:eqn-to-inv}
	Let $M$ and $N_0$ be defined as in Proposition~\ref{prop:tail-paths}
	and Corollary~\ref{cor:PiN-positive}.
	Then for any $l\ge l^*:=\max(2M,N_0)$ the following holds: if a function $\varphi\in L^2(Y,\nu)$ satisfies equalities
	\begin{equation}\label{eq:lem-eqn-to-inv}
	T_{\g(i_1)}^{-1}\dots T_{\g(i_{l-1})}^{-1}\varphi_{i_l}=T_{\g(j_1)}^{-1}\dots T_{\g(j_{l-1})}^{-1}\varphi_{j_l}
	\end{equation}
	for all admissible sequences $i_0\to i_1\to\dots\to i_l$, $j_0\to j_1\to\dots\to j_l$ with $i_0=j_0$, then $\varphi(x,k)$ does not depend on $k \in \Xi$:
	$\varphi(x,k)=\varphi^\circ(x)$, and $\varphi^\circ(x)$ is $G_0^2$-invariant.
\end{lemma}

\begin{remark}If \eqref{eq:lem-eqn-to-inv} holds for all pairs of sequences of a given length $l$, then it holds for any pair of sequences $i_0\to i_1\to\dots\to i_{l'}$, $j_0\to j_1\to\dots\to j_{l'}$ of length $l'\le l$ with $i_0=j_0$. Indeed, append an arbitrary prefix $i_{-(l-l')}\to\dots\to i_0$ to these sequences and apply \eqref{eq:lem-eqn-to-inv} to the resulting sequences of length $l$.
	One can see that $T_{\g(i_{-(l-l')+1})}^{-1}\dots T_{\g(i_{0})}^{-1}$ cancels out and we arrive at \eqref{eq:lem-eqn-to-inv} for the original sequences of length $l'$.
\end{remark}

Let us first deduce Assumptions~\ref{asm:Qn} and~\ref{asm:QstarmQm} from Lemma~\ref{lem:eqn-to-inv}.

\begin{cor}\label{cor:asm-eqns}
	In the ergodic case, that is, assuming $\I_{G_0^2}$ is trivial, Assumptions~\ref{asm:Qn} and~\ref{asm:QstarmQm} hold for the operator $Q$ defined by \eqref{eq:PandU} and \eqref{eq:QVW}.
\end{cor}

\begin{proof}
	1. Suppose that $Q^n\varphi=\varphi$. Choose $s$ such that $l=2ns\ge l^*$. Then $P^l\varphi=(Q^n)^s\varphi=\varphi$. Therefore,
	\eqref{eq:lem-eqn-to-inv} holds, as both sides are equal to $T_{\g(i_0)}\varphi_{i_0}$ by \eqref{eq:asmQn-comp}.
	Lemma~\ref{lem:eqn-to-inv} then implies that all $\varphi_j$ are equal to the same function $\varphi^\circ$, where $\varphi^\circ$ is $G_0^2$-invariant and hence, by ergodicity,  constant.
	
	2. Suppose that $(Q^*)^mQ^m\varphi=\varphi$, where $m$ satisfies $2m\ge l^*$. Proposition~\ref{prop:eqns-to-systems} implies that \eqref{eq:lem-eqn-to-inv} holds for $\varphi$ with $l=2m$, so $\varphi$ is constant.
\end{proof}

It remains to prove Lemma~\ref{lem:eqn-to-inv}.

\begin{proof}[\proofname\ of Lemma~\ref{lem:eqn-to-inv}]
	For every $e\in G_0$ let $\underline\H^e$, $\underline\T^e$, $\underline i^e$, $\underline j^e$, $\mathbf{h}(e)$, and $\mathbf{t}(e)$ be defined as in Proposition~\ref{prop:tail-paths}.
	Recall that $\T^e_0=\H^e_0=\R$ and define $g_e,h_e\in G$ such that $\T^e_{M-1}=g_e\R$, $\H^e_{-M+1}=h_e\R$, whence
	\begin{equation*}
	g_e=\g(i^e_0)^{-1}\dots \g(i^e_{M-2})^{-1},\quad h_e=\g(j^e_{-1})\dots\g(j^e_{-M+1}).
	\end{equation*}
	Denote $\psi_e=T_{g_e}\varphi_{\mathbf{t}(e)}$.
	
	Take any $e_1,e_2$ and choose $\hat e\ne e_1^{-1},e_2^{-1}$. Let $\underline\S^\alpha$ ($\alpha=1,2$) be the paths from the first part of Proposition~\ref{prop:HT-combin} applied to $e_\alpha$ and $\hat e$, and let $\underline k^\alpha=(k^\alpha_{-M}\to\dots\to k^\alpha_{M-1})$ be the corresponding sequences of states. Then $\S^\alpha_{-M+1}=\H^{\hat e}_{-M+1}=h_{\hat e}\R$, $\S^{\alpha}_{M-1}=\T^{e_\alpha}_{M-1}=g_{e_\alpha}\R$, hence
	\begin{equation*}
	g_{e_\alpha}=h_{\hat e}\g(k^\alpha_{-M+1})^{-1}\dots \g(k^\alpha_{M-2})^{-1}.
	\end{equation*}
	Therefore,
	\eqref{eq:lem-eqn-to-inv} for the sequences $\underline k^\alpha$ yields
	\begin{equation*}
	T_{h_{\hat e}^{-1}}T_{g_{e_1}^{\phantom{1}}}\varphi_{\mathbf{t}(e_1)}=T_{h_{\hat e}^{-1}}T_{g_{e_2}^{\phantom{1}}}\varphi_{\mathbf{t}(e_2)},
	\end{equation*}
	whence $\psi_{e_1}=\psi_{e_2}$. We thus obtain that all $\psi_e$ are equal to the same function $\psi^\circ$.
	
	Now  again take any $e_1,e_2$, choose $\hat e\ne e_1^{-1},e_2^{-1}$, and apply the same argument to the paths from the second part of Proposition~\ref{prop:HT-combin} for $\hat e$, $e_\alpha$ and $\tilde e_\alpha$. We have that
	\begin{equation*}
	T_{h_{\hat e}^{-1}e_1}\psi_{\tilde e_1}=T_{h_{\hat e}^{-1} e_1 g_{\tilde e_1}^{\phantom{1}}}\varphi_{\mathbf{t}(\tilde e_1)}=T_{h_{\hat e}^{-1} e_2 g_{\tilde e_2}^{\phantom{1}}}\varphi_{\mathbf{t}(\tilde e_2)}=T_{h_{\hat e}^{-1} e_2}\psi_{\tilde e_2},
	\end{equation*}
	Therefore, $T_{e_1^{-1}e_2}\psi^\circ=\psi^\circ$, so $\psi^\circ$ is $G_0^2$-invariant. Then the function $T_h\psi^\circ$ is independent of $h\in G_0$; denote it by $\psi^\bullet$.
	The function $\psi^\bullet$ is also $G_0^2$-invariant and $T_h\psi^\bullet=\psi^\circ$ for all $h\in G_0$.
	
	Finally, fix $\tilde\imath=\mathbf{t}(e)$ and take any $\tilde\jmath\in\Xi$. With $N = N_0$ as in Corollary~\ref{cor:PiN-positive}, we can consecutively choose $i_{N-1},\dots,i_1,i_0$ so that the sequence $i_0\to\dots\to i_{N-1}\to i_N=\tilde\imath$ is admissible,
	and then, since $(\Pi^{N_0})_{i_0\tilde\jmath}>0$, we can choose $j_1,\dots,j_{N-1}$ such that $i_0=j_0\to j_1\to\dots\to j_{N-1}\to j_N=\tilde\jmath$ is also admissible.
	Using \eqref{eq:lem-eqn-to-inv} for these sequences and taking into account the definition of $\psi_e$, we see that
	\begin{equation*}
	\varphi_{\tilde\jmath}=T_{\g(j_{N-1})\dots\g(j_1)\g(i_1)^{-1}\dots\g(i_{N-1})^{-1}\g(i^e_{M-2}) \ldots \g(i^e_{0})}\psi^\circ.
	\end{equation*}
	The number $L=2N+M-3$ of the generators in the product is the same for all $\tilde\jmath$, hence all  the $\varphi_{\tilde\jmath}$ are the same; they are equal to either $\psi^\circ$ or $\psi^\bullet$ depending on the parity of~$L$.
\end{proof}

\subsection{Checking Assumption~\ref{asm:ineq}}
\label{subsec:proof-ineq}

\begin{lemma}\label{lem:asm-ineq}
	Assumption~\ref{asm:ineq} holds for the operators defined by~\eqref{eq:PandU} and \eqref{eq:QVW}. Precisely, inequality~\eqref{eq:asm-ineq} holds for $a=6$ and $b=2$.
\end{lemma}

The proof rests on a rather complicated geometric statement which is needed to compare terms on the two sides of estimate~\eqref{eq:asm-ineq} in Assumption~\ref{asm:ineq}.  The underlying meaning of this geometric statement, made precise in Lemma~\ref{lem:wye}, is the following. 
Consider a thickened path $\underline{\S}=(\S_0=\A,\dots,\S_{2n}=\B)$.
Then if this thickened path does not belong to a small set of ``exceptional'' paths, it can be embedded into a triangle $\A\B\C$ of thickened paths such that (1) the lengths of $\A\C$ and $\C\B$ are not more than $n+\mathrm{const}$, and (2) the triangle has, informally speaking, ``zero angles'' in all its vertices. In the simplest case of free group ``zero angle'' in vertex $\A$ means the coincidence of the levels of the paths $\A\B$ and $\A\C$ that are adjacent to $\A$. One can easily see that here we have no exceptional paths: choose $\C$ to be any neighbour of $\S_n$ other than $\S_{n-1}$ and $\S_{n+1}$, then $\C\A=(\C,\S_n,\S_{n-1},\dots,\S_0)$ and  $\C\B=(\C,\S_n,\S_{n+1},\dots,\S_{2n})$ have length $n+1$.

In the general case the states of our Markov chain are not uniquely recovered from the configurations of the domains, so we say that ``zero angle'' in $\A$ means the coincidence of the first elements in the sequences of states generating $\A\B$ and $\A\C$. Another amendment in the formal statement of the lemma is that we deal with any sequences generates by our Markov chain, not only with thickened paths. 

To find such a triangle we cut the sequence $\A\B$ near its midpoint and then modify the halves to construct $\A\C$ and $\C\B$. This is done by the lemmas from Section~\ref{sec:operations}, which allow us to keep track of the generating sequences of states throughout the modification process.

%%%%%%%%%%%%%%%%%%%%
\begin{figure}[hbt]
	\centering
	\includegraphics{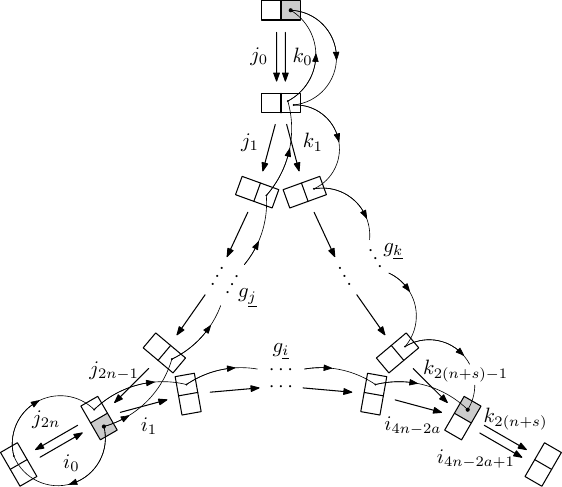}
	\caption{Illustrating Lemma~\ref{lem:wye} and formula~\eqref{eq:wye-identity}}
	\label{fig:wye}
\end{figure}
%%%%%%%%%%%%%%%%

Denote by $\Paths_{r}$ the set of all admissible sequences $\underline i=(i_0\to\dots\to i_r)$ of states in the Markov chain and let  $\lambda>\nobreak1$ be the Perron--Frobenius eigenvalue of its adjacency matrix~$\Pi$.

\begin{lemma}\label{lem:wye}
	For all sufficiently large $N$, there exists an exceptional subset $E_{2N-1}\subset\Paths_{2N-1}$ with $\# E_{2N-1}=O(\lambda^{N})$,  such that the following holds for every $\underline i\in\Paths_{2N-1}\setminus E_{2N-1}$:
	there exist $\alpha\in\{1,2,3,4\}$, $\beta\in\{-1,0,1,2\}$, and admissible sequences $\underline j=(j_l)_{l=0}^{N-\beta+\alpha-1}$, $\underline k=(k_l)_{l=0}^{N+\beta+\alpha-1}$ with the following properties.
	\begin{enumerate}[itemsep=0pt,label={(\roman*)}]
		\item\label{item:wye-states} $j_0=k_0$, $j_{N-\beta+\alpha-1}=\iota(i_0)$, $k_{N+\beta+\alpha-1}=i_{2N-1}$.
		\item\label{item:wye-domains} Let $\underline\S=(
		\S_0,\dots,\S_{2N})$ be any sequence of domains generated by $\underline i$  
		and let   
		$\underline\U=(\U_0,\dots,\U_{N-\beta+\alpha})$, $\underline\V=(\V_0,\dots,\V_{N+\beta+\alpha})$ be the unique sequences generated by $\underline j$ and $\underline k$ respectively with the property that $\U_{N-\beta+\alpha}=\S_0$, $\V_{N+\beta+\alpha}=\S_{2N}$.
		Then $\U_0=\V_0$. 
		\item\label{item:wye-minimal} $j_1\ne k_1$.
	\end{enumerate}
 Moreover the mapping $\underline i\mapsto (\underline j,\underline{\vphantom{j}k})$ with  $(\underline j,\underline{\vphantom{j}k})$ satisfying (i)-(iii) is injective.
\end{lemma}

This lemma is illustrated in Figure~\ref{fig:wye}. Every state from the sequences $\underline{\vphantom{j}i}, \underline j, \underline{\vphantom{j}k}$ is represented by a straight arrow, while the past and the future domains of the state are shown as the pairs of squares near the start and the end of this arrow. Other details of this figure, including the numbering for elements of $\underline j$ and $\underline{\vphantom{j} k}$, which is different from that in the statement of the lemma, are discussed below when proving equality~\eqref{eq:wye-identity}.

\begin{remark}\label{rem:item-minimal} 
In the lemma, $\beta$  can be set to zero except  in the special case of Remark~\ref{rem:specialcase}.

Statements~\ref{item:wye-states} and~\ref{item:wye-domains} remain true if the same admissible sequence $\underline t=(t_l)_{l=-s}^0$ with $t_0=j_0=k_0$ is prepended to both $\underline j$ and $\underline{\vphantom{j}k}$. Statement~\ref{item:wye-minimal} requires $\underline j$ and $\underline{\vphantom{j}k}$ to be chosen without this common initial part. The possibility of adding or removing such a common initial segment will be used later. 

\end{remark}

\begin{proof}[Proof of Lemma~\ref{lem:asm-ineq} assuming Lemma~\ref{lem:wye}.]The values of $a$ and $b$ in the statement of Lemma~\ref{lem:asm-ineq} are in fact
	$b=\max|\beta|$, $a=b+\max\alpha$, where the possible values of $\alpha$ and $\beta$ are described in the course of the proof of Lemma~\ref{lem:wye}, see Claims~\ref{clm:cut-point}  and~\ref{clm:restore-i}. 
	
Let us examine the terms on each side of~\eqref{eq:asm-ineq}. First, for a nonnegative function $\varphi\in L^1(Y,\nu)$ we have
	\begin{multline}\label{eq:equal-ineq-LHS}
	(WQ^{2n-a}\varphi)_l=(UP^{4n-2a+1}\varphi)_l\\
	=\sum_{i_1,\dots,i_{4n-2a+1}}p_{\iota(l),i_1}p_{i_1i_2}\dots p_{i_{4n-2a}i_{4n-2a+1}} T^{-1}_{\u(l)}T^{-1}_{\g(\iota(l))}T^{-1}_{\g(i_1)}\dots T^{-1}_{\g(i_{4n-2a})}\varphi_{i_{4n-2a+1}}.
	\end{multline}
The coefficient in a term of the last sum is nonzero if and only if the sequence $\iota(l)\to i_1\to\dots\to i_{4n-2a+1}$ is admissible. Since $(p_{ij})$ is the matrix for the Parry measure, formula~\eqref{eq:Parry-prod} yields
	\begin{equation}\label{eq:est-ineq-LHS}
	(WQ^{2n-a}\varphi)_l\le \widetilde{C}_1 \lambda^{-4n} \sum_{\substack{\underline i\in\Paths_{4n-2a+1},\\ i_0=\iota(l)}}
	T_{\u(i_0)}T^{-1}_{\g(i_0)}T^{-1}_{\g(i_1)}\dots T^{-1}_{\g(i_{4n-2a})}\varphi_{i_{4n-2a+1}}.
	\end{equation}
	where we use $\omega(\iota(l)) = \omega(l)^{-1}$, see Lemma~\ref{lem:g-and-u}.
	Similarly,
	\begin{multline*}
	((Q^*)^nQ^{n+s}\varphi)_l=\sum_{\substack{j_{2n-1},\dots,j_0,\\k_1,\dots, k_{2n+2s}}}
	\frac{p_{j_0}}{p_{l}} p_{j_{2n-1}l}p_{j_{2n-2}j_{2n-1}}\dots p_{j_0j_1}p_{j_0k_1}p_{k_1k_2}\dots p_{k_{2n+2s-1}k_{2n+2s}}\\
	{}\times T_{\g(j_{2n-1})}\dots T_{\g(j_0)}T^{-1}_{\g(j_0)}T^{-1}_{\g(k_1)}\dots T^{-1}_{\g(k_{2n+2s-1})}\varphi_{k_{2n+2s}},
	\end{multline*}
	hence~\eqref{eq:Parry-prod} yields the estimate
	\begin{multline}\label{eq:est-ineq-RHS}
	((Q^*)^nQ^{n+s}\varphi)_l\\
	{}\ge\widetilde{C}_2 \lambda^{-4n}
	\sum_{\substack{\underline j\in\Paths_{2n},\\\underline k\in\Paths_{2n+2s},\\ j_0=k_0,j_{2n}=l}}
	T_{\g(j_{2n-1})}\dots T_{\g(j_0)}T^{-1}_{\g(k_0)}T^{-1}_{\g(k_1)}\dots T^{-1}_{\g(k_{2n+2s-1})}\varphi_{k_{2n+2s}}.
	\end{multline}
	Here $\widetilde{C}_2$ is chosen in such a way that this inequality holds for any $s$ with $|s|\le b$.
	Apply Lemma~\ref{lem:wye} to a sequence $\underline i$ from~\eqref{eq:est-ineq-LHS} with $N=2n-a+1$. There are $O(\lambda^{2n})$ sequences in the exceptional set $E_{4n-2a+1}$,
	and the corresponding terms in~\eqref{eq:equal-ineq-LHS} comprise $(A_n\varphi)_l$. Thus $\|A_n\|=O(\lambda^{-2n})$, hence the series $\sum_n \|A_n\|$ converges.
	
	Suppose now that $\underline i\notin E_{4n-2a+1}$. Choose paths $\underline{\tilde\jmath}\in\Paths_{2n-a-\beta+\alpha}$ and $\underline{\vphantom{j}\tilde k}\in\Paths_{2n-a+\beta+\alpha}$ as in Lemma~\ref{lem:wye}.
	Denote $\eta=a-\alpha+\beta\ge 0$, consider any admissible sequence $\underline t=(t_{-\eta}\to\dots\to t_0=j_0=k_0)$, and adjoin $\underline t$ to both $\underline{\tilde\jmath}$ and $\underline{\vphantom{j}\tilde k}$ to construct $\underline j\in\Paths_{2n}$ and $\underline{\vphantom{j}k}\in\Paths_{2n+2\beta}$.
	
	Let us prove that the terms in~\eqref{eq:est-ineq-LHS} and \eqref{eq:est-ineq-RHS} (with $s=\beta$) corresponding to this choice of  $\underline{\vphantom{j}i}$, $\underline j$, and~$\underline{\vphantom{j} k}$ are equal.
	Indeed, $l=\iota(i_0)=j_{2n}$ and $k_{2n+2s}=i_{4n-2a+1}$, so it remains to prove that
	\begin{equation}\label{eq:wye-identity}
	\underbrace{\u(i_0)\g(i_0)^{-1}\g(i_1)^{-1}\dots \g(i_{4n-2a}^{-1})}_{g_{\underline i}}=
	\underbrace{\g(j_{2n-1})\dots \g(j_0)}_{g_{\underline j}}\underbrace{\g(k_0)^{-1}\dots \g(k_{2n+2s-1})^{-1}}_{g_{\underline k}},
	\end{equation}
	where we define $g_{\underline i}$, $g_{\underline j}$, and $g_{\underline k}$ as shown.
	
	Statements~\ref{item:wye-states} and~\ref{item:wye-domains} of Lemma~\ref{lem:wye} hold for $\underline j$ and $\underline{\vphantom{j}k}$, so let $\underline\S$, $\underline\U$, and $\underline\V$ be generated by $\underline{\vphantom{j} i}$, $\underline j$, $\underline{\vphantom{j} k}$ as in statement~\ref{item:wye-domains}.
	Let $h\R$ be the right domain in $\S_1$, then the definitions of $\g(\,\cdot\,)$ and $\u(\,\cdot\,)$ in formula \eqref{eq:g-and-u} imply that $h\u(i_0)\R$ is the left domain in $\S_0$, $h\u(i_0)\g(i_0)^{-1}\R$ is the left domain in $\S_1$,~$\dots$, $hg_{\underline i}\R$ is the left domain in $\S_{4n-2a+1}$.
	
	On the other hand, as $\S_0=\U_{2n+1}$ and $i_0=\iota(j_{2n})$, we have that $\S_1=\U_{2n}$ and $h\R$ is the left domain in $\U_{2n}$.
	The same argument as above gives us that $hg_{\underline j}\R$ is the left domain in $\U_0$, which coincides with the left domain in $\V_0$, so $hg_{\underline j}g_{\underline{\vphantom{j}k}}\R$ is the left domain in $\V_{2n+2s}$. But as $\V_{2n+2s+1}=\S_{4n-2a+2}$ and $k_{2n+2s}=i_{4n-2a+1}$, we obtain that the left domains in $\V_{2n+2s}$ and $\S_{4n-2a+1}$ coincide. Therefore, $hg_{\underline i}\R=hg_{\underline j}g_{\underline{\vphantom{j}k}}\R$, so \eqref{eq:wye-identity} holds.
	This is illustrated in Figure~\ref{fig:wye}: the curved arrows link the domains $hg\R$, where $g$ is an initial segment of either the left-hand or the right-hand side of \eqref{eq:wye-identity}; the shaded squares correspond to $g=\mathrm{id}$ (left), $g=g_{\underline j}$ (top), $g=g_{\underline i}=g_{\underline j}g_{\underline{\vphantom{j}k}}$ (right).
	
	Therefore, we have proved that for every term in the right-hand side of \eqref{eq:est-ineq-LHS} except those with $\underline i\in E_{4n-2a+1}$,  there exists an equal term in the right-hand side of \eqref{eq:est-ineq-RHS} for some $s$ with $|s|\le b$.

	Finally we check that different terms in \eqref{eq:est-ineq-LHS} correspond to different terms in \eqref{eq:est-ineq-RHS}. 
	Indeed, by the last statement of Lemma~\ref{lem:wye} different sequences $\underline i\in\Paths_{4n-2a+1}\setminus E_{4n-2a+1}$ correspond to different pairs 
	$(\underline{\tilde\jmath},\underline{\vphantom{j}\tilde k})$. On the other hand, by item~\ref{item:wye-minimal} of this lemma the pair $(\underline{\tilde\jmath},\underline{\vphantom{j}\tilde k})$ is uniquely determined by the pair $(\underline{j},\underline{\vphantom{j}k})$: to see this find maximal $s\ge 0$ such that $j_l=k_l$ for all $l=0,\dots,s$ and remove the common initial segment $(j_0,\dots,j_{s-1})$ from $\underline{j}$ and~$\underline{\vphantom{j} k}$. Therefore, different sequences $\underline i$ yield different pairs $(\underline j,\underline{\vphantom{j} k})$, hence
	\begin{equation*}
	\sum_{\substack{\underline i\in\Paths_{4n-2a+1}\setminus E_{4n-2a+1},\\ i_0=\iota(l)}}
	T_{g_{\underline i}}\varphi_{i_{4n-2a+1}}\le
	\sum_{s=-b}^b
	\sum_{\substack{\underline j\in\Paths_{2n},\\\underline k\in\Paths_{2n+2s},\\ j_0=k_0,j_{2n}=l}}
	T_{g_{\underline j}g_{\underline k}}\varphi_{k_{2n+2s}}.
	\end{equation*}
	Combining this inequality with \eqref{eq:est-ineq-LHS} and \eqref{eq:est-ineq-RHS}, we establish \eqref{eq:asm-ineq}.
\end{proof}

\begin{proof}[\proofname\ of Lemma~\ref{lem:wye}]
The proof will be carried out in a number of steps. At various points we perform certain operations from Section~\ref{sec:operations} on the sequences in question and then check that the number of sequences for which this alters states so as to make the requirement  (i) in the statement of the lemma impossible   is $O(\lambda^N)$.

As in Lemma~\ref{lem:change-states}, let $n_0$ be the maximal value of $n(v)$ for all vertices of $\R$, $n_0=2$ if $\R$ has no vertices inside $\DD$. From now on we assume that $N>n_0+5$; otherwise one can set $E_{2N-1}:=\Paths_{2N-1}$. 
Take any $\underline{i}\in \Paths_{2N-1}$ and consider a sequence $\underline\S=(\S_0,\dots,\S_{2N})$ that is generated by $\underline i$; let $s_l=\S_l\cap\S_{l+1}$. 

\smallskip

\noindent\textit{Step 1.} We begin by splitting $\underline\S$ at a suitable point $\S_{N - \beta} $ near $\S_N$, where $\beta$ is chosen as in the next claim.   Note that   if $N(\R) \geq 4$ we can take $\beta = 0$ and the proof simplifies. 
	
	\begin{claim}\label{clm:cut-point}There exists $\beta\in\{-1,0,1,2\}$ and a domain $\A\in\S_{N-\beta}$ with a side $\tilde s$ not belonging to $s_{N-\beta-1}\cup s_{N-\beta}$.
	In the special case from Remark~\ref{rem:specialcase} we also require  that $\tilde s$ has an end $v$ that either belongs to $\dd\DD$ or is incident to at most $n(v)-1$ domains in~$\underline\S$.
	\end{claim}
	
	\begin{proof}
		Assume first that there exists a level $\S_{N-\beta}$, $\beta\in\{-1,0,1,2\}$ which contains two fundamental domains. Then $\S_{N-\beta}$ has $2N(\R)$ sides with at most six of them included in $s_{N-\beta-1}\cup s_{N-\beta}$. Any other side in $\dd\S_{N-\beta}$ can be chosen as $\tilde s$ with $\A$ being the domain in $\S_{N-\beta}$ adjacent to $\tilde s$.
		
		Thus further consideration is needed only when $N(\R)=3$ and both states $i_{N-\beta-1}$ and $i_{N-\beta}$ are of type $E$. The states of type~$E$ exists only if there are two adjacent vertices in $\dd\R$ that lie inside $\DD$, hence we are in the special case from Remark~\ref{rem:specialcase}. However, $(E\to E)$-transition needs a vertex $u$ with $n(u)=2$, and this is ruled out in this remark. We have thus proved that there exists a side $\tilde s\not\subset s_{N-\beta-1}\cup s_{N-\beta}$. Note that in the special case the side $\tilde s$ is not incident to the common vertex of two fundamental domains in $\S_{N-\beta}$, whence $\tilde s$ is non-compact.
		
		Now assume that each $\S_{N-\beta}$, $\beta\in\{-1,0,1,2\}$, contains only one domain.
		Then $s_{N-1}\cup s_{N}$ consists of two sides of $\S_N$. In the non-special case any other side in $\dd\S_N$ can be chosen as $\tilde s$. In the special case this can fail if the side $s=\dd\S_N\setminus (s_{N-1}\cup s_{N})$ is compact. Then let $u$ be the common vertex of $s$ and $s_{N-1}$. If $u$ is incident to $n(u)\ge 3$ levels of $\underline\S$, then $u$ is incident to the only domain in $\S_{N-2}$. Therefore, $\tilde s=\dd\S_{N-1}\setminus(s_{N-2}\cup s_{N-1})$ is the only side of $\S_{N-1}$ non-adjacent to $u$, hence $\tilde s$ is non-compact.
	\end{proof}
	
\noindent\textit{Step 2.} Having split the sequence $\underline\S$ into two halves at $\S_{N-\beta}$, the next step is to narrow these halves, reducing $\S_{N-\beta}$ to the domain $\A$ chosen as above in Claim~\ref{clm:cut-point} (and fixed for the remainder of this proof). This we do by applying Lemma~\ref{lem:narrowing} to the sequences $(i_l)_{l=0}^{N-\beta-1}$ and $(\S_l)_{l=0}^{N-\beta}$ to obtain sequences $(j'_l)_{l=0}^{N-\beta-1}$ and $(\U'_l)_{l=0}^{N-\beta}$ with $\U'_{N-\beta}=\A$. Similarly, from $(i_l)_{l=N-\beta}^{2N-1}$ and $(\S_l)_{l=N-\beta}^{2N}$ we obtain sequences $(k_l')_{l=N-\beta}^{2N-1}$ and $(\V_l')_{l=N-\beta}^{2N}$ with $\V'_{N-\beta}=\A$.

	 \begin{claim}\label{clm:narrow-ok}
		Let $E^{(1)}_{2N-1}$ be the set of sequences $\underline i\in\Paths_{2N-1}$ such that $j'_0\ne i_0$ or $k'_{2N-1}\ne i_{2N-1}$. Then $\# E^{(1)}_{2N-1}=O(\lambda^N)$.
	\end{claim}
	
	\begin{proof}
		As one can see from the second statement of Lemma~\ref{lem:narrowing}, $j'_0\ne i_0$ implies that $\S_{n_0-1}\ne\U'_{n_0-1}$.Hence all states in the sequence $(i_l)_{l=n_0-1}^{N-\beta-1}$ are of types $C$ and $E_\alpha$, where $\alpha=L$ if $\A=\L_{N-\beta}$ and $\alpha=R$ if $\A=\R_{N-\beta}$. This means that the states $(i_l)_{l=n_0-1}^{N-\beta-1}$ are uniquely determined by $i_{N-\beta-1}$. 
		Therefore, there are finitely many possibilities for $(i_l)_{l=0}^{N-\beta-1}$
		and $O(\lambda^N)$ possibilities for $(i_l)_{l=N-\beta}^{2N-1}$.  
	\end{proof}
	
\noindent\textit{Step 3.} Assume now that $\underline i\notin E^{(1)}_{2N-1}$. The next step is to shift the numbering in the sequences constructed above and invert the first pair, after which we join  a short head sequence as defined in Proposition~\ref{prop:tail-paths} to the beginning of each of them.

 Let $M$ be the number from Proposition~\ref{prop:tail-paths},
then we define $\underline j''_+=(j''_l)_{l=M}^{M+N-\beta-1}$, $\underline\U''_+=(\U''_l)_{l=M}^{M+N-\beta}$ as
	\begin{equation*}
	\U''_l=\U'_{M+N-\beta-l},\qquad
	j''_l=\iota(j'_{M+N-\beta-1-l})
	\end{equation*}
	and 
	$\underline k''_+=(k''_l)_{l=M}^{M+N+\beta-1}$, $\underline\V''_+=(\V''_l)_{l=M}^{M+N+\beta}$ as
	\begin{equation*}
	\V''_l=\V'_{N-\beta-M+l},\qquad
	k''_l=k'_{N-\beta-M+l}.
	\end{equation*}
	Further, let $\A=a\R$ and let $\tilde e$ be the label on the side $\tilde s$ inside $\A$. Define
	\begin{equation*}
	\underline j''_-=\underline k''_-=(j^{\tilde e}_{l-M})_{l=0}^{M-1},\qquad
	\underline\U''_-=\underline\V''_-=(a\H^{\tilde e}_{l-M})_{l=0}^{M},
	\end{equation*}
	where $\underline j^{\tilde e}$ and $\underline\H^{\tilde e}$ are defined in Proposition~\ref{prop:tail-paths}.
		
	Apply Lemma~\ref{lem:joining} to join $\underline\U''_-$ and $\underline\U''_+$. 
	Except in the special case of Remark~\ref{rem:specialcase} 
	this is possible  because $  j^{\tilde e}_{-1}$ is type $A_0$
so that  the path $\underline\U''_-$ adds only one domain incident to each end $u$ of $\tilde s$, so for the union of these paths the vertex $u$ is either convex or minimally concave. In the special case of Remark~\ref{rem:specialcase} and compact side~$\tilde s$ this is amended as follows: $\underline\U''_-$ adds two domains to one of the ends of $\tilde s$; by choosing~$\H^{\tilde e}$ to be either the path from Figure~\ref{fig:tail-paths}j or its mirror image we make this end to be the vertex~$v$ from Claim~\ref{clm:cut-point}.
	
	Thus by joining  $\underline\U''_-$ to $\underline\U''_+$ we obtain 
new sequences and states which we rename as  $\underline\U=(\U_l)_{l=0}^{N+M-\beta}$ and $\underline j=(j_l)_{l=0}^{N+M-\beta-1}$. We have   $j_0=j^{\tilde e}_{-M}$ by the construction of the path $\underline\H^{\tilde e}$. 
	Similarly, we join $\underline\V''_-$ and $\underline\V''_+$ to obtain
	$\underline\V=(\V_l)_{l=0}^{N+M+\beta}$ and $\underline k=(k_l)_{l=0}^{N+M+\beta-1}$, and we have $k_0=j^{\tilde e}_{-M}=j_0$,  as well as $\U_0=a\H^{\tilde e}_{-M}=\V_0$.
	
	\begin{claim}\label{clm:join-ok}
		Let $E^{(2)}_{2N-1}$ be the set of sequences $\underline i\in\Paths_{2N-1}\setminus E^{(1)}_{2N-1}$ such that $j_{N+M-\beta-1}\ne j''_{N+M-\beta-1}$ or $k_{N+M+\beta-1}\ne k''_{N+M+\beta-1}$. Then $\# E^{(2)}_{2N-1}=O(\lambda^N)$.
	\end{claim}
	
	\begin{proof} 
		We first consider changes to states made at the joining step to see to what extent the joining changes the states $j$, see Lemma~\ref{lem:joining}. As above, $j_{N+M-\beta-1}\ne j''_{N+M-\beta-1}$ implies $\U_{M+N-\beta-n_0}\ne \U''_{M+N-\beta-n_0}$. Assume that $\tilde s$ lies on the left boundary of $\S_{N-\beta}$. Then each  level $(\U''_l)_{l=M}^{M+N-\beta-n_0}$ contains only one domain and these domains are the consecutive domains adjacent to a geodesic segment on the right boundary of $\underline\U''_+$, that is, the left boundary of $\underline\U'$.
		Moreover the states $(j''_l)_{l=M}^{M+N-\beta-n_0-1}$ are uniquely defined by $j''_M$, so there are finitely many $\underline j''_+$'s (or, equivalently, $\underline j'$'s) such that $j_{N+M-\beta-1}\ne j''_{N+M-\beta-1}$.
		
		Now we  consider changes to states made at the narrowing step, for which we use Lemma~\ref{lem:narrowing}. Let us show that each of these $\underline j'$'s can be obtained from finitely many $(i_l)_{l=0}^{N-\beta-1}$. Indeed, assume that $\U'_{N-\beta-3}\ne \S_{N-\beta-3}$, i.e. that the narrowing step changes at least the four last domains in $(\S_l)_{l=0}^{N-\beta}$. Then for $l=1,2,3$ we have that $\U'_{N-\beta-l}=\L_{N-\beta-l}$ and all $\dd_R\U'_{N-\beta-l}$ belong to the same geodesic segment (see Figure~\ref{fig:narrowing}a). 
		
		We have shown that the same holds for $\dd_L\U'_{N-\beta-l}=\dd_R\U''_{M+l}$, $l=1,2,3$ since the joining step adds domains to all levels up to $\U''_{M+N-\beta-n_0+1}$. Therefore, each of $\dd_{L,R}\U'_{N-\beta-l}$ is either a side or a vertex since the $l = 2$ region is joined to the $l=1$ and $l=3$ regions across one side only.  
		This means that $N(\R)\le 4$ and $\R$ is compact.
		Assumption~\ref{asm:R} now yields that $N(\R)=4$ and each of $\dd_{L,R}\U'_{N-\beta-l}$ is a segment. Thus $\dd\underline\U'$ has a straight angle at every vertex $u$ of $\U'_{N-\beta-2}$. On the other hand, there are only two domains in $\underline\U'$ that are adjacent to $u$, hence $n(u)=2$. This contradicts Assumption~\ref{asm:R}.
		
		Therefore, $j_{N+M-\beta-1}\ne j''_{N+M-\beta-1}$ only  for finitely many sequences $(i_l)_{l=0}^{N-\beta}$ and hence for $O(\lambda^N)$ sequences $\underline i=(i_l)_{l=0}^{2N-1}$.
	\end{proof}
	
\noindent\textit{Step 4.} Assume now that $\underline i\notin E_{2N-1}:=E^{(1)}_{2N-1}\cup E^{(2)}_{2N-1}$. Then we have
	\begin{equation*}
	j_{N+M-\beta-1}=j''_{N+M-\beta-1}=\iota(j'_0)=\iota(i_0)\quad\text{and}\quad
	\U_{N+M-\beta}=\U''_{N+M-\beta}=\U'_0=\S_0.
	\end{equation*}
	Similarly we have $k_{N+M+\beta-1}=i_{2N-1}$ and $\V_{N+M+\beta}=\S_{2N}$.  Also we have seen above that $j_0=k_0$ and $\U_0=\V_0$. 	Therefore, statements~\ref{item:wye-states} and \ref{item:wye-domains} of the lemma hold for the constructed sequences.
	
\smallskip

\noindent\textit{Step 5.} The next claim allows us to prove both (iii) and the final statement of the lemma, that is, that the map $\underline i \to ( \underline j, \underline k)$ is injective.   The claim itself will proved in Step 6 below.

	\begin{claim}\label{clm:restore-i}Assume that we have sequences $\U, \V$ as in (i) and (ii) with $\U_0 = \V_0$ and let $M$ be as in Proposition~\ref{prop:tail-paths} and Step 3 above.  Then:
	
		1. Let $s$ be the maximal number such that $j_l=k_l$ for $l=0,\dots, s$. Then $s<M$.
		
		2. Let $s'$ be the maximal number $l$ such that $\U_l$ and $\V_l$ have a common domain. Then $s'=M$ and $\U_M\cap\V_M=\A$ with $\A$  as in Claim~\ref{clm:cut-point} above.  
		
		3. Lemma~\ref{lem:narrowing} applied to the sequence $(\U_l)_{l=M}^{M+N-\beta}$ and the domain $\A$ yields the sequence $\underline\U''_+$. Similarly, $(\V_l)_{l=M}^{M+N+\beta}$ yields $\underline\V''_+$.
		
 4. Lemma~\ref{lem:joining} applied to the sequences 
		$\underline{\U}'$ and $\underline{\V}'$ produces the original sequence $\underline\S$. 	\end{claim}
	
	 As we have noted in Remark~\ref{rem:item-minimal}, one needs to remove from $\underline j$ and $\underline{\vphantom{j} k}$ their common initial segment to satisfy statement~\ref{item:wye-minimal}.
Item (1) of the claim means that,  after removing their common initial segment, the sequences $\underline j$ and $\underline{\vphantom{j} k}$ belong to $\Paths_{N\pm\beta+\alpha-1}$ with $\alpha=M-s\in\{1,2,3,4\}$, proving (iii).
	
	Now assuming Claim~\ref{clm:restore-i}, let us check that the map $\underline i\mapsto (\underline j,\underline{\vphantom{j}k})$ is injective.  	
	Suppose given $\underline j$ and $\underline{\vphantom{j}k}$ satisfying (i) and  consider the sequences $\underline\U$ and $\underline\V$ constructed as in (ii), so that $\U_0=\V_0$. The domain $\A$ is identified uniquely from  (2) of the  claim, then by (3) one can restore $\underline\U''_+$ and $\underline\V''_+$ (and hence $\underline{\U}'$ and $\underline{\V}'$), and by (4) the original sequence~$\underline\S$. Thus we have a unique sequence of domains $\underline\S$ generated by $\underline i$, together with  its initial and final states $i_0=\iota(j_{N-\beta+\alpha-1})$ and  $i_{2N-1}=k_{N+\beta+\alpha-1}$. Hence we can uniquely restore the whole sequence~$\underline i$ as in item 2 of the first part in the proof of Theorem~\ref{thm:th-path-vs-sphere} as required. 
\smallskip

\noindent\textit{Step 6.} Finally we establish the last claim.
	
	\begin{proof}[Proof of Claim~\ref{clm:restore-i}]
		It is convenient to deal with sequences $\underline i$ which start and end in the states $\Xi_S, \Xi_F$ respectively.
		Thus we begin by showing that  the statements of Claim~\ref{clm:restore-i} for suitably extended sequences imply the same statements for the original ones. 
		
		Consider any admissible sequences $(i_l)_{l=-\delta}^0$ and $(i_l)_{l=2N-1}^{2N-1+\varepsilon}$ such that $i_{-\delta}\in\Xi_S$, $i_{2N-1+\varepsilon}\in\Xi_F$ and denote $\underline{\hat\imath}=(i_l)_{l=-\delta}^{2N-1+\varepsilon}$.
		This extends the sequence $\underline\S$  to a sequence $\underline{\hat\S}=(\S_l)_{l=-\delta}^{2N+\varepsilon}$ generated by $\underline{\hat\imath}$. Provided that $\underline i\notin E_{2N-1}$, we can apply the above procedure of narrowing and joining to these extended sequences. For all sequences involved in this procedure we use the notation as above with an added  hat, for example  $\underline{\hat\U}''_+=(\hat\U''_l)_{l=M}^{M+N-\beta+\delta}$.
		
		Note that the narrowing and the joining in the original procedure does not change the terminal elements in $\underline\S$ and $\underline i$. Therefore, the same operations for the extended sequences do not modify any of the added segments. In other words, the sequences $\underline{\hat\jmath}$, $\underline{\hat k}'$, $\underline{\hat\U}''_+$, etc. are the extensions of the corresponding sequences without hats by the segments $(i_l)_{l=-\delta}^{-1}$, $(i_l)_{l=2N}^{2N-1+\varepsilon}$, $(\S_l)_{l=-\delta}^{-1}$, and $(\S_l)_{l=2N+1}^{2N+\varepsilon}$ or their inversions.
		
		Thus the statements of Claim~\ref{clm:restore-i} for the extended sequences imply the same statements for the original ones. From now on we deal with the extended sequences only.
		
		\smallskip
		
		Above we have constructed a \textsf{Y}-shaped combination of the three thickened paths
		\begin{equation}\label{eq:three-paths}
		\underline{\hat\U}''_+,\quad \underline{\hat\V}''_+,\quad a\underline\H^{\tilde e}, 
		\end{equation}
		all meeting in the domain $\A$. Now we construct a related \textsf{Y}-shaped triple of rays. To do this, consider generic points 
		\begin{equation*}
		O\in\Int\A,\quad X_\U\in\dd_O\U''_{M+N-\beta+\delta},\quad
		X_\V\in\dd_O\V''_{M+N+\beta+\varepsilon},\quad X_\H\in\dd_O(a\H^{\tilde e}_{-M}) 
		\end{equation*}
		such that the lines $\ell(OX_J)$, $J=\U,\V,\H$, do  not contain any vertices of $\TR$, where we recall that $\dd_O$ denotes those part of the boundary of the terminal domain in a path that is not shared with the adjacent domain of this path. 
		The segments $OX_J$ lie inside the corresponding convex sets:
		\begin{equation*}
		OX_\U\subset \bigcup\underline{\hat\U}''_+,\quad
		OX_\V\subset \bigcup\underline{\hat\V}''_+,\quad
		OX_\H\subset \bigcup a\underline{\H}^{\tilde e}.
		\end{equation*}
		These sets are the thickened paths between their ends hence each of these segments crosses all consecutive levels in its respective  thickened path; see Subsection~\ref{subsec:thick-paths-convex}.
		In particular, these segments leave $\A$ via different sides: $OX_\U$ crosses a side from $s_\U:=s_{N-\beta-1}\cap\dd\A$, $OX_\V$ crosses a side from $s_\V:=s_{N-\beta}\cap\dd\A$, and $OX_\H$ crosses $\tilde s$. 
		
		Define $\alpha_J$ to be the ray on $\ell(OX_J)$ that starts at $O$ and contains $X_J$ and let $\alpha^+_J\subset\alpha_J$  be the ray starting at $X_J$.
		The rays $\alpha_J$ cut $\DD$ into three sectors; denote the sector bounded by $\alpha_J$ and $\alpha_{J'}$ by $\Sigma_{JJ'}$.
		
		From these definitions one can see that none  of the curves $\tilde s, s_\U, s_\V$  intersect with the ``opposite'' sector. Hence every path in~\eqref{eq:three-paths} intersects   the ``opposite'' sector only in an appropriate part of the domain $\A$. Indeed, $\bigcup \underline{\hat\U}''_+$ is a convex set. If $x\in \bigcup\underline{\hat\U}''_+ \setminus\A$, there is a point $y\in Ox$ that belongs to $s_\U$. On the other hand, if $x\in\Sigma_{\V\H}$, then $y\in Ox\subset \Sigma_{\V\H}$ and we arrive at a contradiction.
		
		The ends of $\tilde s$ belong to $\Sigma_{\U\H}$ and $\Sigma_{\V\H}$, we denote them by $v_\U$, and $v_\V$ respectively. Then the domains added to $\underline{\hat\U}''$ by the application of Lemma~\ref{lem:joining} belong to $\Sigma_{\U\H}$. Indeed, the set $\bigcup\underline{\hat\U}\setminus \bigcup\underline{\hat\U}''$ is a connected set that contains $v_\U\in\Sigma_{\U\H}$ on its boundary. This set cannot intersect the curve $X_\U OX_\H$ which lies inside $\bigcup\underline{\hat\U}''$. And if it intersects, say, $\alpha^+_\U$, then the intersection $\alpha_\U\cap \bigcup\underline{\hat\U}$ is a segment that goes beyond $X_\U$. This means that 
		$\bigcup\underline{\hat\U}$ contains the domain bordering $\U''_{M+N-\beta+\delta}$ at $X_\U$. But we have required that no domains adjacent to $\U''_{M+N-\beta+\delta}$ are added to $\underline{\hat\U}''$.
		
		We now pass to the proof of the statement of the claim.
		
		1. This statement follows directly from the second one: if $j_l=k_l$ for $l=0,\dots,M$ then $\U_0=\V_0$ yields $\U_l=\V_l$ for $l=1,\dots,M+1$, hence $\U_{M+1}\cap\V_{M+1}$ is nonempty.
		
		2. Let us show that $\bigcup\underline{\hat\U}\cap\bigcup\underline{\hat\V}=
		\bigcup a\underline\H^{\tilde e}$. Indeed, the domains in $\bigcup\underline{\hat\U}$ fall into three classes: (a)~those from $a\underline\H^{\tilde e}$, (b)~those from $\bigcup\underline{\hat\U}''_+\setminus\A$, (c)~those added by the joining. The first two classes are disjoint since $\underline{\hat\U}$ is a thickened path between its ends, so its different levels do not intersect.
		The domains in $\bigcup\underline{\hat\V}$ are similarly classified into the classes (a), (b${}'$), and (c${}'$). The classes (b) and (b${}'$) belong to the two different halves of the thickened path $\underline{\hat\S}$ and hence do not intersect. The classes $\text{(b)}\cup\text{(c)}$ and (c${}'$) do not intersect since the former contains no domains intersecting $\Sigma_{\V\H}$, while the latter lies in this sector.
		Hence the intersection $\bigcup\underline{\hat\U}\cap\bigcup\underline{\hat\V}$ consists of the domains of the class (a) only.
		
		3. This follows directly from Corollary~\ref{cor:join-narrow}.
		
		4. A joining as in (3) is a minimal convex union of fundamental domains that contains both $\underline{\hat\U}''_+$ and $\underline{\hat\V}''_+$, hence it lies inside $\bigcup\underline{\hat\S}$. 
		Since $\underline i\notin E_{2N-1}$, the union $\bigcup\underline{\hat\U}''_+\cap\bigcup\underline{\hat\V}''_+=
		\bigcup\underline{\hat\U}'\cap\bigcup\underline{\hat\V}'$ contains all domains in $\S_l$ with $l\le n_0-1$ or $l\ge2N-n_0+1$. Therefore, the joining adds no domains to  these levels and hence yields an admissible sequence $\underline{\hat{\textbf{\i}}}$ with $\mathbf{i}_l=i_l$ for $l\le 0$ and $l\ge 2N-1$. 
		Then each of $\underline{\hat{\textbf{\i}}}$ and $\underline{\hat{\imath}}$ belongs to $\Paths^{S\to F}_{2N-1+\delta+\varepsilon}$ and generates the thickened path between $\S_{-\delta}$ and $\S_{2N+\varepsilon}$, hence $\underline{\hat{\textbf{\i}}}=\underline{\hat{\imath}}$ by~Theorem~\ref{thm:th-path-vs-sphere}.
	\end{proof}
	
	This completes the proof of Lemma~\ref{lem:wye}.\end{proof}

\subsection{Conclusion of the proof of Theorem~\ref{thm:main}}
\label{subsec:conclusion}

Take an ergodic decomposition of the measure $\mu$ with respect to the action of the subgroup of $G$ generated by $G_0^2=\{g_1g_2:g_1,g_2\in G_0\}$ and consider an ergodic $G_0^2$-invariant measure $\tilde\mu$.

Note that in general the operator $P$ does not preserve the measure $\tilde\mu\times p$, but the operators $Q,V,W$ defined by~\eqref{eq:QVW} do, as they contain only terms of the form $f\circ T_{g_1}\circ T_{g_2}$ for $g_1,g_2\in G_0$.
Formula~\eqref{eq:Sn-as-PkU} then yields
\begin{equation*}
\widetilde{\mathbf{S}}_{2n}(f)=\lambda^{2n-1}\sum_{j\in \Xi_S} h_j(Q^{n-1}V\varphi^{(f)})_j.
\end{equation*}
Note also that $\#S(2n)$ equals the number of paths from $\Xi_S$ to $\Xi_F$ of  length $2n$, thus
\begin{equation*}
\#S(2n)=\sum_{\substack{i\in\Xi_S,\\ j\in\Xi_F}}(\Pi^{2n-1})_{ij}=C\lambda^{2n-1}(1+o(1)),
\end{equation*}
whence
\begin{equation*}
\mathbf{S}_{2n}(f)=\tilde C \sum_{j\in \Xi_S} h_j(Q^{n-1}V\varphi^{(f)})_j\cdot (1+o(1)).
\end{equation*}
Now we apply Theorem~\ref{thm:convergence} to the operators \eqref{eq:QVW} acting on the space $L^1(Y,\tilde\nu)$, where $\tilde{\nu}=\tilde{\mu}\times p$. Recall that we have checked Assumptions~\ref{asm:Qn}--\ref{asm:ineq} for these operators in Corollary~\ref{cor:asm-eqns} and Lemma~\ref{lem:asm-ineq}. Hence we obtain that the following holds for $\tilde\mu$-almost every $x$:
\begin{itemize}
	\item $\mathbf{S}_{2n}(f)(x)$ converges to some limit, which we denote as $\tilde f(x)$.
	\item $\tilde f(x)=\tilde f(T_{g_1g_2}x)$ for any $g_1,g_2\in G_0$.
\end{itemize}
The second item results from the fact that $\tilde f$ is constant $\tilde\mu$-almost everywhere.

Therefore, the set of $x\in X$ such that these two conditions hold, is of full measure with respect to every convex combination of the ergodic measures,
in particular, with respect to the initial measure $\mu$.
Thus $\lim_{n\to\infty}\mathbf{S}_{2n}(f)(x)$ exists $\mu$-almost surely and is $G_0^2$-invariant. On the other hand, for every $A\in \I_{G_0^2}$ one has
\begin{equation*}
\int_A f\,d\mu=\int_A \mathbf{S}_{2n}(f)\,d\mu\to \int_A \tilde f\,d\mu,
\end{equation*}
whence $\tilde f=\EE(f|\I_{G_0^2})$.
The proof of Theorem~\ref{thm:main} is now complete.

\section{Proof of Theorem~\ref{thm:convergence}}
\label{sec:proof-convergence}

In this section we prove Theorem~\ref{thm:convergence}, which gives conditions for the  pointwise convergence of powers of a Markov operator. This result is a generalization of Theorem~1 in \cite{Buf-Annals}, and the proof here follows the same general scheme.
After defining the   space of trajectories  corresponding to $Q$, we prove first  that $Q$ is mixing, next that the tail sigma-algebra of the space of trajectories is trivial, and finally use this to prove convergence for functions in $L\log L$ both in $L^1$  and pointwise.

\subsection{The space of trajectories}\label{subsec:trajectories}

Recall that $Q\colon L^1(Z,\eta)\to L^1(Z,\eta)$ is a measure-preserving Markov operator.

The \emph{space of trajectories} corresponding to $Q$ is the space $(\mathbf{Z},\mathbb{P}_Q)$, where $\mathbf Z=Z^{\mathbb Z}$ with the usual Borel sigma-algebra~$\mathcal B_{\mathbf Z}$, and the measure $\mathbb P_Q$ is defined below.
This is essentially an application of the Ionescu Tulcea Extension Theorem, where the stochastic kernels depend only on the previous element of the trajectory and are the same:
\begin{equation}\label{eq:kernel}
	\mathbb{P}_Q(z,A)=\mathbb{P}_{Q,z}(A):=Q[\mathbf{1}_A](z).
\end{equation} 
For the details of this construction we refer the reader to~\cite[Ch.~14]{Klenke}. However, the direct application of this approach faces the following difficulty: the right-hand side of~\eqref{eq:kernel} is defined for a fixed $A$ up to a modification on a set of $z$ of zero measure, so we cannot assert that $\mathbb{P}_{Q,z}$ is a sigma-additive measure.

We circumvent this problem by another approach to defining~$\mathbb P_Q$ below. First, to motivate our definition, let us pretend for a moment that  $\mathbb{P}_{Q,z}$ is indeed a sigma-additive measure on $Z$ for any $z$.
Then for an integral with respect to this measure we have
\begin{equation}\label{eq:kernel-int}
	\int_{w\in Z} f(w)\,d\mathbb{P}_{Q,z}(w)=Q[f](z)
\end{equation}
(for $f=\mathbf{1}_A$ this is~\eqref{eq:kernel}, then use linearity and the monotone convergence theorem).

Now if we say that the conditional distribution of $z_n$ with respect to $(z_m,\dots,z_{n-1})$ should be equal to $\mathbb{P}_{Q,z_{n-1}}$, we get the following formula for the probability of a cylinder set:
\begin{multline*}
	\mathbb P_Q\{z_m\in A_m,\dots, z_n\in A_n\}={}\\
	\int_{z_m\in A_m} 
	\biggl[\int_{z_{m+1}\in A_{m+1}}\biggl[\dots
	\biggl[\int_{z_n\in A_n} d\mathbb P_{Q,z_{n-1}}(z_n)\biggr]\dots\biggr]d\mathbb P_{Q,z_m}(z_{m+1})\biggr] d\eta(z_m) 
\end{multline*}
(compare with \cite[Theorem 14.22]{Klenke}).
This formula can be rewritten as follows:
\begin{multline}\label{eq:prob-traj}
	\mathbb P_Q\{z_m\in A_m,\dots, z_n\in A_n\}={}\\
	\mathbb P_m^n(A_m\times\dots\times A_n):=
	\EE(\mathbf{1}_{A_m}\cdot
	Q(\mathbf{1}_{A_{m+1}}\cdot Q(\dots Q(\mathbf{1}_{A_n})\dots)))
\end{multline}
Indeed, the innermost integral equals $\mathbb{P}_{Q,z_{n-1}}(A_n)=Q(\mathbf{1}_{A_n})(z_{n-1})$, then we apply~\eqref{eq:kernel-int} for all integrals going from inside out.

Now we may \emph{define} the measure~$\mathbb{P}_Q$ as the measure with finite-dimen\-sional distributions $\mathbb P_m^n$ given by~\eqref{eq:prob-traj}.  Let us check that these $\mathbb{P}_m^n$ satisfy assumptions of the Kolmogorov Extension Theorem.
	
\begin{lemma}\label{lem:prob-traj}1) For any $\varphi\in L^1(Z,\eta)$ we have $\EE(Q\varphi)=\EE \varphi$.\\
	2) $\mathbb P_m^n$ is a finitely-additive measure on the semi-ring of cylinders.\\ 
	3) If $B_k\subset A_k$, then
	\begin{equation*}
		\mathbb P_m^n(A_m\times\dots\times A_n)-\mathbb P_m^n(B_m\times\dots\times B_n)\le \sum_{k=m}^n \eta(A_k)-\eta(B_k).
	\end{equation*}
	4) The measure $\mathbb P_m^n$ is $\sigma$-additive. \\
	5) The distributions~$\mathbb P^n_m$ are consistent:
	\begin{equation*}
		\mathbb P_m^n(A_m\times\dots\times A_n)=\mathbb P_{m-1}^n(Z\times A_m\times\dots\times A_n)=\mathbb P_m^{n+1}(A_m\times\dots\times A_n\times Z).
	\end{equation*}
\end{lemma}

\begin{proof}
	1. Since both sides are $L^1$-continuous, it is sufficient to consider $\varphi\in L^\infty(Z,\eta)$: $-C\le \varphi\le C$. Then $C-\varphi\ge 0$, hence
	\begin{equation*}
		C-\EE(Q\varphi)=\EE(Q(C-\varphi))=\|Q(C-\varphi)\|_{L^1}\le \|C-\varphi\|_{L^1}=\EE(C-\varphi)=C-\EE\varphi.
	\end{equation*} 
	Therefore, $\EE(Q\varphi)\ge\EE \varphi$, and the same argument for $-\varphi$ yields $\EE(Q\varphi)\le\EE\varphi$.
	
	2. As usual, this is reduced to the case $\mathbb P_m^n(C_1)+\mathbb P_m^n(C_2)=\mathbb P_m^n(C_1\sqcup C_2)$, 
	where $C_{1,2}$ have the same projections on all coordinates except one, and this case is clear.
	
	3. This follows from the inclusion
	\begin{equation*}
		A_m\times\dots\times A_n\setminus B_m\times\dots\times B_n\subset\bigcup_{k=m}^n Z\times\dots\times Z\times (A_k\setminus B_k)\times Z\times\dots\times Z
	\end{equation*}
	and the first statement of the lemma.
	
	4. Since $(Z,\eta)$ is a Lebesgue space, we may assume that it is a union of a segment and an at most countable set of atoms. Then the usual proof works: let $\hat C=\bigsqcup_{i=1}^\infty C_i$, then $\sum_{i=1}^\infty \mathbb P_m^n(C_i)\le \mathbb P_m^n(\hat C)$ follows from the finite additivity, and to obtain the opposite inequality, find open cylinders $D_i\supset C_i$ with $\mathbb P_m^n(D_i)\le \mathbb P_m^n(C_i)+\varepsilon/2^i$ and a compact cylinder $\hat D\subset \hat C$ with $\mathbb P_m^n(\hat D)\ge \mathbb P_m^n(\hat C)-\varepsilon$. These cylinders are constructed coordinate-wise using the estimate from the previous item. Then $\hat D$ is covered by $D_i$'s and hence by a finite number of them. Finite additivity then yields $\mathbb P_m^n(\hat D)\le \sum_{i=1}^N \mathbb P_m^n(D_i)\le  \sum_{i=1}^\infty\mathbb P_m^n(D_i)$, hence $\mathbb P_m^n(\hat C)\le \sum_{i=1}^\infty\mathbb P_m^n(C_i)+2\varepsilon$. It remains to take a limit as $\varepsilon\to +0$.  This proves the $\sigma$-additivity on the semi-ring of cylinders, and the Carath\'edory Extension Theorem then extends $\mathbb P_m^n$ to a $\sigma$-additive measure on the Borel $\sigma$-algebra on $Z^{n-m+1}$. 
		
	5. This is straightforward using the first statement of this lemma.
\end{proof}

Therefore, $\mathbb{P}_Q$ defined by~\eqref{eq:prob-traj} exists by the Kolmogorov Extension Theorem.

It is also clear from the definition that \emph{the left shift map} $\sigma\colon \mathbf{Z}\to\mathbf{Z}$,
$(\sigma(\mathbf{z}))_n=z_{n+1}$ preserves the measure $\mathbb P_Q$.

We can clearly define a measure $\mathbb{P}_{Q^*}$ in a similar way. The following calculation relates $\mathbb{P}_Q$ and $\mathbb{P}_{Q^*}$. We have
\begin{multline*}
	\mathbb P_Q\{z_m\in A_m,\dots, z_n\in A_n\}=
	\langle \mathbf{1}_{A_m},
	Q(\mathbf{1}_{A_{m+1}}\cdot Q(\dots Q(\mathbf{1}_{A_n})\dots))\rangle={}\\
	\langle Q^*(\mathbf{1}_{A_m}),
	\mathbf{1}_{A_{m+1}}\cdot Q(\dots Q(\mathbf{1}_{A_n})\dots)\rangle=
	\langle \mathbf{1}_{A_{m+1}}\cdot(Q^*(\mathbf{1}_{A_m})),
	 Q(\dots Q(\mathbf{1}_{A_n})\dots)\rangle={}\\
	 \dots=\mathbb{P}_{Q^*}\{z_{-n}\in A_n,\dots,z_{-m}\in A_m\}.
\end{multline*} 
In other words, $\mathbb{P}_{Q^*}$ is a pullback of $\mathbb{P}_Q$ under the time-reversal map $(z_n)_{n=-\infty}^\infty\mapsto (z_{-n})_{n=-\infty}^\infty$.

We now derive an  important result  concerning conditional expectations. Let $\F_k^l$, $k,l\in\mathbb {Z} \cup\{+\infty,-\infty\}$ be the minimal complete sigma-algebras such that all functions $\pi_j\colon\mathbf{z}=(z_n)\mapsto z_j$ are measurable for $k\le j\le l$. For brevity we write $\F_n^n=\F_n$.
Let us also recall that the \emph{tail sigma-algebra} is defined as
\begin{equation*}
\F_{\mathrm{tail}}=\bigcap_{n=0}^\infty \F_n^\infty.
\end{equation*}
For any function $\varphi\in L^1(Z,\eta)$ we define the functions $\varphi^r\in L^1(\mathbf Z,\mathbb P_Q)$, $r\in\mathbb Z$ by the formula $\varphi^r(\mathbf{z})=\varphi(z_r)$.
The next lemma shows how to find the conditional expectation of a function on $\mathbf{Z}$ that depends on only one coordinate with respect to the sigma-algebra generated by some other coordinates.

\begin{lemma}\label{lem:CondExp}For any $\varphi\in L^1(Z,\eta)$ and $n>0$
\begin{equation}\label{eq:CondExp}
	\begin{gathered}
		\EE(\varphi^r|\F^{-n+r}_{-\infty})(\mathbf z)=\EE(\varphi^r|\F_{-n+r})(\mathbf z)=(Q^n\varphi)(z_{-n+r}),\\
		\EE(\varphi^r|\F_{n+r}^{+\infty})(\mathbf z)=\EE(\varphi^r|\F_{n+r})(\mathbf z)=((Q^*)^n\varphi)(z_{n+r}).
	\end{gathered}
\end{equation}
\end{lemma}
 Note that this lemma proves the starting point of our heuristic approach above, the formula~\eqref{eq:kernel}: the first equality in~\eqref{eq:CondExp} for $n=1$ and $\varphi=\mathbf{1}_A$ yields $\mathbb{P}_Q(z_r \in A\mid z_{r-1},z_{r-2},\dots)=Q[\mathbf{1}_A](z_{r-1})$.  
 
\begin{proof}[Proof of Lemma~\ref{lem:CondExp}]
	 Let us check that $\EE(\varphi^r|\F_{-\infty}^{-n+r})(\mathbf{z})=(Q^n\varphi)(z_{-n+r})$.  The sets of the form $\{z_m\in A_m,\dots, z_{-n+r}\in A_{-n+r}\}$ for all $m\le -n+r$ and all Borel $A_j$'s
	generate the sigma-algebra $\F_{-\infty}^{-n+r}$, hence it is sufficient to check that for all such sets $B$ we have 
	\begin{equation}\label{eq:CondExp-proof}
		\EE(\mathbf{1}_{B}(\mathbf z)\cdot \varphi(z_r))=\EE(\mathbf{1}_{B}(\mathbf z)(Q^n\varphi)(z_{-n+r})).
	\end{equation}
	To do this, observe that
	\begin{equation*}
		\EE_{\mathbb{P}_Q}(\mathbf{1}_{A_m}(z_m)\cdots\mathbf{1}_{A_{k-1}}(z_{k-1})\cdot\psi(z_{k}))=
		\EE_\eta (\mathbf{1}_{A_m}\cdot Q(\dots Q(\mathbf{1}_{A_{k-1}}\cdot Q(\psi))\dots)).
	\end{equation*}
	Indeed, for $\psi=\mathbf{1}_{A_k}$ this is~\eqref{eq:prob-traj}; the general case follows by the linearity and $L^1$-continuity of both sides.
	Therefore, both sides of~\eqref{eq:CondExp-proof} are equal to 
	\begin{equation*}
	  \EE_\eta (\mathbf{1}_{A_m}\cdot Q(\dots Q(\mathbf{1}_{A_{-n-1+r}}\cdot Q(\mathbf{1}_{A_{-n+r}}\cdot Q^n(\varphi)))\dots)).	
	\end{equation*}
	
	 Also, since $\F_{-n+r}\subset\F_{-\infty}^{-n+r}$ and we have seen that $\EE(\varphi^r|\F_{-\infty}^{-n+r})$ is $\F_{-n+r}$-measurable, we obtain $\EE(\varphi^r|\F_{-\infty}^{-n+r})=\EE(\varphi^r|\F_{-n+r})$. This proves the first formula in~\eqref{eq:CondExp}, and the second one  follows by time reversal as above.
\end{proof}

 We have seen in the previous lemma that the conditional expectation of a function depending only on $r$-th coordinate, $r>k$, with respect to the ``past'' $\sigma$-algebra $\F_{-\infty}^k$ depends only on $z_{k}$, the last coordinate in this interval. This statement can be extended to any function that is measurable with respect to the ``future'' $\sigma$-algebra $\F_{k+1}^\infty$. 

\begin{lemma}\label{lem:EE-depend}
	Assume  that a function $\Phi\in L^1(\mathbf{Z},\mathbb{P}_Q)$ is $\F_{k+1}^\infty$-measurable. Then $\EE(\Phi|\F_{-\infty}^k)$ depends on $z_k$ only.
\end{lemma}

\begin{proof}
	If $\Phi=\mathbf{1}_C$, where $C=\{z_{k+1}\in C_{k+1},\dots, z_{k+s}\in C_{k+s}\}$, we use the same argument as for~\eqref{eq:CondExp-proof} and obtain
	\begin{equation*}
		\EE(\Phi|\F_{-\infty}^k)(\mathbf{z})=
		Q(\mathbf{1}_{C_{k+1}}\cdot Q(\mathbf{1}_{C_{k+2}}\cdot Q(\dots Q(\mathbf{1}_{C_{k+s}})\dots)))(z_k).
	\end{equation*}
	The general case follows by linearity and the fact that $\Phi_n\to\Phi$ in $L^1$ yields $\EE(\Phi_n|\G)\to\EE(\Phi|\G)$ in $L^1$.
\end{proof}

\subsection{Mixing of the operator $Q$}

We  start by proving mixing for $\tilde Q=Q^m$ where $m$ is defined in Assumption~\ref{asm:QstarmQm}.

\begin{lemma}\label{lem:mixingQm}Let $\tilde Q$ be a measure-preserving Markov operator on $L^1(Z,\eta)$ such that the equation $\tilde Q^*\tilde Q\varphi=\varphi$ has only constant solutions in $L^2(Z,\eta)$. Then for any $\varphi,\psi\in L^2(Z,\eta)$ we have
	\begin{equation}\label{eq:mixing}
	\langle \tilde Q^n\varphi,\psi\rangle=\int_Z \tilde Q^n\varphi\cdot \overline\psi\,d\eta\to \int_Z\varphi\,d\eta \int_Z\overline\psi\,d\eta\text{ as }n\to\infty.
	\end{equation}
\end{lemma}

\begin{proof}The statement follows from the mixing of the shift map $\sigma$ in the trajectory space $(\mathbf{Z},\mathbb{P}_{\tilde Q})$. To obtain the latter we shall prove that $\sigma$ has the $K$-property: that is, there exists a sub-sigma-algebra $\mathcal K$ of the Borel sigma-algebra $\mathcal B_{\mathbf Z}$ such that $\mathcal K\subset \sigma\mathcal K$, $\bigvee_{n=0}^\infty \sigma^n\mathcal K=\mathcal B_{\mathbf Z}$, $\bigcap_{n=0}^\infty \sigma^{-n}\mathcal K=\{\varnothing, \mathbf Z\}$.
	
By the Rokhlin--Sinai theorem (see \cite{RokhSin}, \cite[Ch.~18]{Glasner}) the $K$-property is equivalent to the triviality of the Pinsker sigma-algebra $\Pi(\sigma)$ (the smallest sigma-algebra containing all  measurable partitions of zero entopy). Consider $\F_-=\F_{-\infty}^0$. Then $\sigma\F_-\subset\F_-$ and $\bigvee_{k\in\mathbb Z} \sigma^k\F_-=\mathcal B_{\mathbf Z}$. Thus $\Pi(\sigma^{-1})\subset\F_-$ (see, e.g., Lemma 18.7.3 in \cite{Glasner}). Similarly, for $\F_+=\F_0^\infty$ one has $\Pi(\sigma)\subset \F_+$. 	Therefore, $\Pi(\sigma)=\Pi(\sigma^{-1})\subset \F_-\cap \F_+=\F_0$.
	
We have proved that any $\Pi(\sigma)$-measurable function $\varphi\in L^2(\mathbf Z,\mathbb{P}_{\tilde Q})$ depends only on the zeroth coordinate: $\varphi(\mathbf z)=\varphi_0(z_0)$; with the notation of~\ref{subsec:trajectories},  $\varphi = (\varphi_0)^0$. More generally, $\varphi(\mathbf z)=\varphi_k(z_k)$.

Now we can calculate $\EE(\varphi|\F_{-1})$ in two ways. On the one hand, $\varphi$ is $\F_{-1}$-measurable, so it equals $\varphi=\varphi_{-1}(z_{-1})$, on the other hand, it equals $\EE((\varphi_0)^0|\F_{-1})=(\tilde Q\varphi_0)(z_{-1})$ by~\eqref{eq:CondExp}. Hence~$\varphi_{-1}=\tilde Q\varphi_0$. Similarly, from $\EE(\varphi|\F_0)=\varphi=\EE((\varphi_{-1})^{-1}|\F_0)$ we obtain $\varphi_0=\tilde Q^*\varphi_{-1}$. Therefore, $\varphi_0=\tilde Q^*\tilde Q\varphi_0$, hence by assumption of the lemma $\varphi_0=\mathrm{const}$, thus $\Pi(\sigma)$ is trivial.
\end{proof}

\begin{cor}\label{cor:mixingQ}The operator $Q$ is also mixing, that is, \eqref{eq:mixing} holds for $Q$ instead of $\tilde Q$.
\end{cor}

\begin{proof}The sequence $(\langle Q^n\varphi,\psi\rangle)_{n\ge 0}$ is the union of the subsequences $(\langle Q^{nm+r}\varphi,\psi\rangle)_{n\ge 0}$, each of which converges to the desired limit by Lemma~\ref{lem:mixingQm} applied to the pair of functions $(Q^r\varphi,\psi)$.
\end{proof}

\subsection{Triviality of the tail sigma-algebra}

The next step is to prove that the tail sigma-algebra for $Q$ is trivial. First, we  prove that the tail sigma-algebra cannot be \emph{totally nontrivial}, that is,
it cannot contain infinitely many different sets (up to sets of measure zero).
The proof follows that of Lemma 6 in \cite{Buf-Annals}, which is a version of the 0--2 law in the form of Kaimanovich \cite{Kaiman}.

\begin{lemma}\label{lem:0-2}For a measure-preserving Markov operator $R$ on $L^1(Z,\eta)$ the following holds.
	If the tail sigma-algebra of~$R$ is totally nontrivial then for any $b\in\mathbb N$ and any $\eps>0$
	there exist nonnegative functions $\varphi,\psi\in L^\infty(Z,\eta)$ with averages equal to~$1$ such that
	\begin{equation}\label{eq:0-2}
	\limsup_{n\to\infty}\langle (R^*)^{n+b}\varphi,(R^*)^n\psi\rangle_{L^2(Z,\eta)}+\dots+\langle (R^*)^{n-b}\varphi,(R^*)^n\psi\rangle_{L^2(Z,\eta)}<\eps.
	\end{equation}
\end{lemma}

\begin{proof}Let $(\mathbf Z,\mathbb P_{R})$ be the corresponding trajectory space.
	If $\F_{\mathrm{tail}}$ contains infinitely many subsets, it contains a subset of arbitrarily small measure.
	Indeed, split $\mathbf Z=A_1^{(2)}\sqcup A_2^{(2)}$, where each set $A_i^{(2)}\in\F_{\mathrm{tail}}$ has nonzero measure.  Then
	at least one of these parts can be split into two sets of nonzero measure (otherwise $\F_{\mathrm{tail}}$ contains only finitely many sets,
	the union of some of the $A_i^{(2)}$). Repeating this procedure, we get $\mathbf Z =A_1^{(n)}\sqcup\dots\sqcup A_n^{(n)}$. Then the measure of at least one of $A_j^{(n)}$ is not more than $1/n$.
	
	Take any set $A\in\F_{\mathrm{tail}}$ with $\mathbb P_R(A)<1/(2b+1)$. Then the set $B=\mathbf Z\setminus \bigcup_{s=-b}^{b} \sigma^s(A)$
	has positive measure. Denote
	\begin{equation*}
	\Phi(\mathbf{z})=\mathbf{1}_{A}(\mathbf{z})/\mathbb P_R(A),\qquad \Psi(\mathbf{z})=\mathbf{1}_{B}(\mathbf{z})/\mathbb P_R(B).
	\end{equation*}
	Observe that $\Phi$ and $\Psi$ are nonnegative  $\F_{\mathrm{tail}}$-measurable functions, bounded by some constant $M$, and with expectations  equal to~$1$. Moreover $(\Phi\circ\sigma^{-j})\cdot \Psi=0$ for $j=-b,\dots,b$.
	
	Set $\varphi_k=\EE(\Phi|\F_{-\infty}^k)$, $\psi_k=\EE(\Psi|\F_{-\infty}^k)$. By Lemma~\ref{lem:EE-depend}  $\varphi_k(\mathbf z)$ depends only on $z_k$, so abusing notation we use the same symbol $\varphi_k$ for the corresponding function in $L^1(Z,\eta)$. For example, we will write $\varphi_k\circ \sigma^j(\mathbf z)=\varphi_k(z_{k+j})$.
	
	Clearly, $\varphi_k$ and $\psi_k$ are nonnegative and bounded by $M$.
	Therefore, the martingale convergence theorem gives that $\varphi_k\to\Phi$, $\psi_k\to\Psi$ in $L^1(\mathbf Z,\mathbb P_R)$.
	Moreover, $\varphi_k(z_{k-j})=\varphi_k\circ\sigma^{-j}( {\mathbf z)}$ and $\varphi_k\circ\sigma^{-j} \to\Phi\circ \sigma^{-j}$.
	Hence
	\begin{equation*}
	 \EE(\varphi_k(z_{k-j})|\F_{\mathrm{tail}})\to\EE(\Phi\circ\sigma^{-j}|\F_{\mathrm{tail}})=\Phi\circ\sigma^{-j}(\mathbf z),\quad
	\EE(\psi_k(z_k)|\F_{\mathrm{tail}})\to\Psi(\mathbf z)
	\end{equation*}
	in $L^1(\mathbf Z,\mathbb P_R)$.
	Since all these functions are bounded by the same constant~$M$, for large $k$ we have that
	\begin{equation*}
	\int_{\mathbf Z} \EE(\varphi_k(z_{k-j})|\F_{\mathrm{tail}})\EE(\psi_k(z_k)|\F_{\mathrm{tail}})\,d\mathbb P_R<\frac{\eps}{2b+1}.
	\end{equation*}
 Applying the second formula in \eqref{eq:CondExp} to $\varphi^{k-j}_k(\mathbf z)= \varphi_k(z_{k-j})$ and $n+j$ in place of $n$, we obtain $\EE(\varphi_k(z_{k-j})|\F_{n+k}^\infty)=[(R^*)^{n+j}\varphi_k](z_{n+k})$.  
Hence for any $j=-b,\dots,b$
	\begin{multline*}
	\int_Z [(R^*)^{n+j}\varphi_k](z_{n+k})\cdot[(R^*)^{n}\psi_k](z_{n+k})\,d\eta=
	\int_{\mathbf Z} \EE(\varphi_k(z_{k-j})|\F_{n+k}^\infty)\cdot
	\EE(\psi_k(z_k))|\F_{n+k}^\infty)\,d\mathbb P_R\\
	\to
	\int_{\mathbf Z} \EE(\varphi_k(z_{k-j})|\F_{\mathrm{tail}})\cdot
	\EE(\psi_k(z_k))|\F_{\mathrm{tail}})\,d\mathbb P_R<\frac{\eps}{2b+1}\text{ as }n\to\infty.
	\end{multline*}
	Therefore, the functions $\varphi_k$ and $\psi_k$ for large $k$ satisfy~\eqref{eq:0-2}.
\end{proof}

\begin{lemma}Under the assumptions of Theorem~\ref{thm:convergence} the tail sigma-algebra for $Q^*$ cannot be totally nontrivial.
\end{lemma}

\begin{proof}
	Assuming the contrary, the inequality \eqref{eq:0-2} in Lemma~\ref{lem:0-2} for $R=Q^*$ yields that for some nonnegative functions $\varphi,\psi$ with their averages equal to $1$ and for all sufficiently large $n$ we have
	\begin{equation}\label{eq:lem-tot-nontr}
		\langle (Q^{n+b}+\dots+Q^{n-b})\varphi,Q^n\psi\rangle_{L^2(Z,\eta)}<\eps.
	\end{equation}
	On the other hand, by Assumption~\ref{asm:ineq} the left-hand side of \eqref{eq:lem-tot-nontr} is not less than
	\begin{equation*}
	\frac1C \langle WQ^{2n-a}\varphi-A_n\varphi,\psi\rangle=
	\frac1C \langle Q^{2n-a}\varphi,W^*\psi\rangle-\frac1C\langle A_n\varphi,\psi\rangle\to \frac1C+0.
	\end{equation*}
	Here we use Corollary~\ref{cor:mixingQ}; note that the average values of both $\varphi$ and $W^*\psi$ are equal to~$1$. Therefore, for large~$n$ the left-hand side of \eqref{eq:lem-tot-nontr} is larger than $1/C-\eps$, so taking $\eps<1/2C$ we arrive at a contradiction.
\end{proof}

\begin{lemma}Under the assumptions of Theorem~\ref{thm:convergence} the tail sigma-algebra for $Q$ cannot be totally nontrivial.
\end{lemma}

\begin{proof}
 Consider the trajectory space $(\mathbf{Z},\mathbb{P})$ for the infinite sequence $\dots,V,W,V,W,\dots$ of Markov operators, that is,
	\begin{equation*}
	\mathbb{P}(z_{2n+1}\in A\mid z_{2n})=V[\mathbf 1_A](z_{2n}),\quad
	\mathbb{P}(z_{2n+2}\in A\mid z_{2n+1})=W[\mathbf 1_A](z_{2n+1}).
	\end{equation*}
	In other words, we use the construction from Subsection~\ref{subsec:trajectories}, but with~\eqref{eq:prob-traj} replaced by
	\begin{equation}\label{eq:prob-traj-diff}
		\mathbb P\{z_m\in A_m,\dots, z_n\in A_n\}=
		\EE(\mathbf{1}_{A_m}\cdot
		R_m(\mathbf{1}_{A_{m+1}}\cdot R_{m+1}(\dots R_{n-1}(\mathbf{1}_{A_n})\dots))),
	\end{equation}
	where $R_{2k}=V$, $R_{2k+1}=W$ for all $k\in\mathbb Z$. In fact, Lemma~\ref{lem:prob-traj} holds for the finite-dimensional distributions~\eqref{eq:prob-traj-diff} with any sequence of Markov operators~$(R_k)$. In our case we have that 
	\begin{equation*}
		\mathbb P\{z_{2k}\in A_k,z_{2(k+1)}\in A_{k+1}\dots, z_{2l}\in A_l\}=
		\mathbb P_{Q}\{z_{k}\in A_k,z_{k+1}\in A_{k+1}\dots, z_{l}\in A_l\},		
	\end{equation*}
	hence the projection $\pi_0\colon\mathbf{z}=(z_n)\mapsto (z_{2n})$ maps the trajectory space $(\mathbf{Z},\mathbb{P})$ to the trajectory space $(\mathbf Z, \mathbb P_Q)$ for the operator~$Q=VW$. Similarly, $\pi_1\colon\mathbf{z}=(z_n)\mapsto (z_{2n+1})$ maps it to the trajectory space for $Q^*=WV$. Therefore, the total non-triviality of the tail sigma-algebras in the trajectory spaces for $Q$ and $Q^*$ is equivalent respectively to that of the sigma-algebras
	\begin{equation*}
	\F_{\mathrm{tail},0}=\bigcap_n \bigvee_{2k\ge n}\F_{2k}\quad\text{and}\quad
	\F_{\mathrm{tail},1}=\bigcap_n \bigvee_{2k+1\ge n}\F_{2k+1}
	\end{equation*}
	in the trajectory space $(\mathbf{Z},\mathbb P)$.
	Since we already know that $\F_{\mathrm{tail},1}$ cannot be totally non-trivial, it is sufficient to prove that
	\begin{equation*}
	\F_{\mathrm{tail},j}=\F_{\mathrm{tail},\mathbf Z}:=\bigcap_n\bigvee_{k\ge n}\F_k.
	\end{equation*}
	Clearly, $\F_{\mathrm{tail},j}\subset\F_{\mathrm{tail},\mathbf Z}$. Let us prove the converse inclusion.	
	Consider any $A\in\F_{\mathrm{tail},\mathbf Z}$ and check that, say, $A\in\F_{\mathrm{tail},0}$. Indeed, $A\in \bigvee_{m\ge 2n}\F_{m}$ for every $n$, and we can eliminate any finite number of $\F_k$ with odd $k$ from this formula:
	\begin{equation}\label{eq:remove-odds}
	A\in \F_{2n}\vee\F_{2n+2}\vee\dots\vee\F_{2(n+s-1)}\vee\bigvee_{m\ge 2(n+s)}\F_m.
	\end{equation}
	Consider the conditional probability $\mathbb{P}({\,\cdot\,}\mid z_{2n},z_{2n+2},\dots)$ with respect to the sigma-algebra $\bigvee_{k\ge n}\F_{2k}$. As~\eqref{eq:remove-odds} shows, with respect to this conditional probability $A$ depends only on the ``odd tail'' $\bigvee_{k\ge n+s}\F_{2k+1}$. But since the odd coordinates $z_{2n+1},\dots,z_{2(n+s)+1},\dots$ are independent for  fixed even coordinates $z_{2n},\dots,z_{2(n+s)},\dots$, by Kolmogorov's 0--1 Law we obtain that $A$ is trivial with respect to this conditional probability, so $A$ is measurable with respect to $\bigvee_{k\ge n}\F_{2k}$, and hence $A\in\F_{\mathrm{tail},0}$.
\end{proof}

\begin{lemma}Under the assumptions of Theorem~\ref{thm:convergence} the tail sigma-algebra for $Q$ is trivial.
\end{lemma}

\begin{proof}It remains to eliminate the case in which $\F_{\mathrm{tail}}$ contains only finitely many different sets.
	Assume that $\mathbf Z=A_1\sqcup\dots\sqcup A_r$, $r>1$, where each $A_j\in\F_{\mathrm{tail}}$ has no nontrivial subsets belonging to $\F_{\mathrm{tail}}$.
	The shift map $\sigma$ interchanges these subsets, whence for $A=A_1$ there exists $n$ such that $\sigma^nA=A$. As in Lemma~\ref{lem:0-2}, we define $\Phi=\mathbf{1}_A/\mathbb P_Q(A)$ and $\varphi_k(z_k)=\varphi_k(\mathbf{z})=\EE(\Phi|\F_{-\infty}^{k})$.
	Then
	\begin{equation*}
	\EE(\Phi|\F_{-\infty}^k)\circ\sigma^n=\EE(\Phi\circ\sigma^n|\F_{-\infty}^{k+n})
	=\EE(\Phi|\F_{-\infty}^{k+n})=\varphi_{k+n},
	\end{equation*}
	hence
	\begin{multline*}
	\varphi_{k+n}(z_k)=\varphi_{k+n}\circ\sigma^{-n}(\mathbf{z})=\EE(\Phi|\F_{-\infty}^k)=\\
	 =\EE(\EE(\Phi|\F_{-\infty}^{k+n})|\F_{-\infty}^k)=\EE(\varphi_{k+n}(z_{k+n})|\F_{-\infty}^k)=[Q^n\varphi_{k+n}](z_k).
	\end{multline*}
	Thus we arrive at the equation $\varphi_{k+n}(z_k)=[Q^n\varphi_{k+n}](z_k)$ and Assumption~\ref{asm:Qn} implies that $\varphi_{k+n}$ is constant.
	Taking averages, we get $\EE(\varphi_{k+n})=\EE(\Phi)=1$, thus $\varphi_l\equiv 1$ for all~$l$. But this contradicts  the convergence $\varphi_l\to\Phi\not\equiv 1$, which was obtained in proof of Lemma~\ref{lem:0-2}.
\end{proof}

\subsection{Convergence}
\begin{prop}[see \cite{Kaiman}; {\cite[Propositions 4, 5]{Buf-Annals}}]\label{prop:L12-conv}
	For a measure-preserving Markov operator $R$ on~$(Z,\eta)$ with  trivial tail sigma-algebra we have
	$R^n\varphi\to \int_Z \varphi\,d\eta$, where the convergence takes place in $L^1$ for $\varphi\in L^1(Z,\eta)$ and
	in $L^2$ for $\varphi\in L^2(Z,\eta)$.
\end{prop}

It remains to prove almost everywhere pointwise convergence for functions in $L^p, p>1$ and in  $L \log L$. Recall that the norm in $L \log L(Z, \eta)$ can be defined by the \emph{Orlicz--Luxemburg norm} 
\begin{equation*}
 \|\varphi\|_{L\log L} = \inf \biggl\{ c:  \int_Z \frac{ |\varphi|}{c} \cdot \log (\frac{ |\varphi|}{c}+e) d \eta \leq 1\biggr\},
\end{equation*}
 see for example~\cite{Zygmund}. In particular, since 
$ \|\varphi\|_{L^1} = \inf \bigl\{ c:   \int_Z (|\varphi|/c) d \eta \leq 1\bigr\}$,
we have
  $\|\varphi\|_{L^1} \leq  A \|\varphi\|_{L\log L}$ for some constant $A$.
More generally, we have the following maximal inequalities:
\begin{lemma}[{\cite[Lemma 8]{Buf-Annals}}]\label{lem:max-ineq}
	For a measure-preserving Markov operator $R$ on $(Z,\eta)$ for any $p>1$ there exists a constant $A_p>0$ such that for any nonnegative function $\varphi\in L^p(Z,\eta)$ we have
	\begin{equation*}
	\Bigl\|\sup_{n\ge 0} (R^*)^nR^n\varphi\Bigr\|_{L^p}\le A_p\|\varphi\|_{L^p}.
	\end{equation*}
	Similarly, there exists a constant $A_{\log}>0$ such that for any nonnegative function $\varphi\in L\log L(Z,\eta)$ we have
	\begin{equation*}
	\Bigl\|\sup_{n\ge 0} (R^*)^nR^n\varphi\Bigr\|_{L^1}\le A_{\log}\|\varphi\|_{L\log L}.
	\end{equation*}
\end{lemma}

\begin{cor}
	The following inequalities hold for any $s\in\mathbb{Z}$:
	\begin{equation}\label{eq:max-AV-ineq}
	\Bigl\|\sup (R^*)^nR^{n+s}\varphi\Bigr\|_{L^p}\le A_p\|\varphi\|_{L^p},
	\quad
	\Bigl\|\sup (R^*)^nR^{n+s}\varphi\Bigr\|_{L^1}\le A_{\log}\|\varphi\|_{L\log L},
	\end{equation}
	the suprema here being  taken over all $n$ such that both $n$ and $n+s$ are nonnegative. 
	\end{cor} 
	\begin{proof}We treat the $L\log L$ norm, the $L^p$ norms are similar. For $s>0$, apply the lemma to $R^s\varphi$ and use the inequality $\|R^s\varphi\|_{L\log L}\le \|\varphi\|_{L\log L}$.
	For $s<0$ we use the inequality
	\begin{equation*}
	(R^*)^{|s|}\Bigl[\sup_{n\ge 0} (R^*)^nR^n\varphi\Bigr]\ge \sup_{n\ge 0} (R^*)^{n+|s|}R^n\varphi
	\end{equation*}
from which we obtain the same inequality  for the $L^1$-norms of both sides. Note also that the $L^1$-norm of the left-hand side does not exceed  $\|\sup_{n\ge 0} (R^*)^nR^n\varphi \|_{L^1}$.    Hence
	\begin{equation*}
	A_{\log}\|\varphi\|_{L\log L}\ge\Bigl\|\sup_{n\ge 0} (R^*)^nR^n\varphi\Bigr\|\ge
	\Bigl\|\sup_{n\ge 0} (R^*)^{n+|s|}R^n\varphi\Bigr\|,
	\end{equation*}
	and it remains to replace $n$ with  $n+s$. 
\end{proof}

\begin{proof}[Proof of Theorem~\ref{thm:convergence}] It remains to  prove pointwise convergence.
	Combining \eqref{eq:max-AV-ineq} for $R=Q$ and Assumption~\ref{asm:ineq} in the form~\eqref{eq:var-asm-ineq}, for any nonnegative function $\varphi\in L\log L(Z,\eta)$ we obtain
	\begin{multline}\label{eq:max-ineq}
	\Bigl\|\sup_{n\ge n_0} Q^{2n-a'}\varphi\Bigr\|_{L^1}\le
	C\Bigl\|V'\Bigl(\sup_{n\ge n_0}\sum_{j=-b}^b (Q^*)^n Q^{n+j}\varphi\Bigr)\Bigr\|_{L^1}+\bigl\|\sup_{n\ge n_0} A_n'\varphi\bigr\|_{L^1}\\
	{}\le (2b+1)A_{\log}C\|\varphi\|_{L\log L}+\sum_{n\ge n_0}\|A_n'\varphi\|_{L^1}\le B_{\log}\|\varphi\|_{L\log L}.
	\end{multline}
	Decomposing a function $\varphi$ into its positive and negative parts we obtain~\eqref{eq:max-ineq} for all real-valued $\varphi\in L\log L(Z,\eta)$
	with a larger $B_{\log}$. The same estimates hold for the $L^p$-norm with $p>1$.
	
	Now consider a real-valued function $\varphi\in L^2(Z,\eta)$ with  zero average. Applying \eqref{eq:max-ineq} to $(Q^{2k}\varphi)$ we have
	\begin{equation*}
	\Bigl\|\sup_{m\ge n_0+k} Q^{2m-a'}\varphi\Bigr\|_{L^2}=\Bigl\|\sup_{n\ge n_0} Q^{2n+2k-a'}\varphi\Bigr\|_{L^2}\le B_2\|Q^{2k}\varphi\|_{L^2}.
	\end{equation*}
	Since the right-hand side tends to zero by Proposition~\ref{prop:L12-conv}, the sequence  $Q^{2m-a'}\varphi$ tends to zero almost everywhere and in $L^2$ as $m\to\infty$.
	
	We now extend pointwise convergence to all $\varphi\in L\log L$.
	Namely, for a real-valued function $\varphi\in L\log L(Z,\eta)$ with zero average, consider $\varphi'\in L^2(Z,\eta)$ with zero average such that
	$\|\varphi-\varphi'\|_{L\log L}\le\eps/B_{\log}$. Then almost surely we have
	\begin{equation*}
	\limsup_{n\to\infty} |Q^{2n-a'}\varphi(z)|\le
	\limsup_{n\to\infty} |Q^{2n-a'}\varphi'(z)|+
	\limsup_{n\to\infty} |Q^{2n-a'}(\varphi-\varphi')(z)|.
	\end{equation*}
	By convergence for functions in $L^2$, the first term in the right-hand side equals zero, while, by the maximal inequality,  the second satisfies 
	$$\|\limsup_{n\to\infty} |Q^{2n-a'}(\varphi-\varphi')(z)|\|_{L^1}\le B_{\log}\|\varphi-\varphi'\|_{L\log L}\le \eps.
	$$ Therefore we have $\limsup |Q^{2n-a'}\varphi(z)|\le\delta$ outside a set of measure less than $\eps/\delta$  for any $\delta>0$. Taking 
	$\eps=1/l^2 $ and then $\delta=1/l $ with $l \to \infty$
	we obtain that this upper limit equals zero almost everywhere. The convergence in $L^1$ follows from the same decomposition:
	\begin{equation*}
	\|Q^{2n-a'}\varphi(z)\|_{L^1}\le \|Q^{2n-a'}\varphi'\|_{L^1}+\|Q^{2n-a'}(\varphi-\varphi')(z)\|_{L^1},
	\end{equation*}
	where the first term tends to zero even with the $L^2$-norm instead of $L^1$, and the second term is less than $\|\varphi-\varphi'\|_{L^1}\le \eps/B_{\log}$.
	
	Finally, combining the  convergence $Q^{2m-a'}\varphi\to 0$ already obtained with the same convergence for $Q\varphi$ in place of $\varphi$, we conclude that $Q^n\varphi\to 0$ almost everywhere as claimed.
\end{proof}

\end{document}